\documentclass[a4 paper, 11 pt]{article}
\usepackage{amsmath,amsthm,amssymb,amsfonts,amscd, graphicx}
\usepackage{amssymb, amsmath, amsthm}
\usepackage{graphicx}
\usepackage{tikz-cd}
\usepackage{hyperref}
\usepackage{caption}
\usepackage{subcaption}

\usepackage[all]{xy}
\usepackage[colorinlistoftodos]{todonotes}
\usepackage{soul, xcolor}
\setstcolor{red}
\usepackage{float}
\usepackage{comment}
\usepackage{mathtools}
\usepackage[normalem]{ulem}

\usepackage[left=1 in,top=1.4 in, right=1 in, bottom=1.4 in]{geometry}
\usepackage[mathlines]{lineno}
\definecolor{cadmiumgreen}{rgb}{0.0, 0.42, 0.24}
\newtheorem{theorem}{Theorem}[section]
\newtheorem{lemma}[theorem]{Lemma}
\newtheorem{proposition}[theorem]{Proposition}
\newtheorem{corollary}[theorem]{Corollary}
\newtheorem{problem}[theorem]{Problem}
\theoremstyle{definition}
\newtheorem{definition}[theorem]{Definition}
\newtheorem{remark}[theorem]{Remark}

\numberwithin{equation}{section}
\usepackage{lineno}

\newenvironment{claim}[1]{\par\noindent\underline{Claim:}\space#1}{}
\newenvironment{claimproof}[1]{\par\noindent\underline{Proof:}\space#1}{\hfill $\blacksquare$}
\newcommand{\ex}{\mathrm{ex}}
\newcommand{\EX}{\mathrm{EX}}

\newcommand{\calH}{\mathcal{H}}
\newcommand{\calL}{\mathcal{L}}

\newcommand{\x}{\mathrm{x}}
\newcommand{\comments}[1]{}
\newcommand{\Pa}{\mathcal{P}}

\usepackage{authblk}
\title{\bf Minimum spectral radius of graphs of fixed order and dissociation number and its connection to Tur\'an problems}
\author{Dheer Noal Desai\thanks{Department of Mathematical Sciences, The University of Memphis, Memphis, TN 38152. E-mail: {\tt dndesai@memphis.edu}} \and Vishal Gupta \thanks{Department of Mathematics, University of Rochester, Rochester, NY 14627. E-mail: {\tt vishalgupta@rochester.edu}} } 
\date{\today}
\begin{document}
\maketitle

\begin{abstract}
    Let $\mathcal{D}_{n,\tau}$ be the set of all simple connected graphs of order $n$ and dissociation number $\tau.$ 
    In this paper, we study the minimum size and the minimum spectral radius of graphs in $\mathcal{D}_{n,\tau}$ in connection with Tur\'an-type problems for complete multipartite graphs.  
    We characterize the Tur\' an graphs for several complete multipartite graphs where the size of one of the partite sets is much smaller than the size of the remaining partites. This extends a result of Erd\H{o}s and Simonovits \cite{ErdosSimCompletemultipartite}. Additionally, we prove some stability results to get the structure of graphs without such a forbidden complete multipartite subgraph, and close to Tur\'an number of edges. 
    As an application, we show that a graph with the minimum spectral radius in $\mathcal{D}_{n,\tau}$ must be a graph with the minimum size in $\mathcal{D}_{n, \tau}$ when $n$ is sufficiently large and satisfies some parity conditions. We then describe a few structural properties of graphs with the minimum spectral radius in $\mathcal{D}_{n,\tau}$. 
    For even dissociation numbers and any order $n$, we compute the minimum size of a graph in $\mathcal{D}_{n,\tau}$ and use it to characterize the graphs in $\mathcal{D}_{n, 4}$ that attain the minimum size and the minimum spectral radius. 
    We also apply the stability results to upper bound the minimum number of edges and spectral radius for connected graphs with a given $d$-independence number when the order of the graph is sufficiently large.
    Finally, we derive two new bounds on the value of $\tau(G)$ for a given graph $G$. 
\medskip


\end{abstract}
\section{Introduction}

Let $G = (V, E)$ be a simple undirected graph. The order of $G$, denoted as $v(G)$, refers to the number of vertices in $G$, and the size of $G$, denoted as $e(G)$, refers to the number of edges in $G$. Let $V(G) = \{v_1, v_2, \ldots, v_n\}$. The adjacency matrix $A(G) = (a_{ij})$ of $G$ is a $(0,1)$-symmetric matrix of order $n$ with $a_{ij} = 1$  if vertices $v_i$ and $v_j$ are adjacent in $G$, and $a_{ij} = 0$ otherwise. Denote by $v_i\sim v_j$ (or resp. $v_i\not\sim v_j$) if the vertices $v_i, v_j$ are adjacent (or resp. not adjacent) in $G$. 
The largest eigenvalue of the adjacency matrix $A(G)$ is called the \textit{spectral radius} of $G$, and is denoted by $\rho(G)$. 
The \emph{principal eigenvector} of $G$ is the eigenvector corresponding to $\rho(G)$ that has all entries positive and Euclidean norm one. 
For a vector $x \in \mathbb{R}^n$, we use $x(v)$ to denote the component of $x$ corresponding to vertex $v\in V(G)$.
Denote by $N(v)$, the set of vertices in $G$ that are adjacent to $v$ in $G$ and $d(v) = |N(v)|$ the degree of the vertex $v$. For a subset $A$ of $V(G)$, denote by $d_A(v)$ the number of neighbors of $v$ in $A$.
For two subsets (or subgraphs) $A, B$ of $V(G)$ (or $G$), we use $e(A, B)$ to denote the number of edges in $G$ with one endpoint in $A$ and the other in $B$.
Further, 
$G \setminus A$ denotes the vertex-induced subgraph of $G$ on $V(G)\setminus A$ (or $V(G)\setminus V(A)$, respectively).
For two graphs $G_1 = (V_1, E_1)$ and $G_2 = (V_2, E_2)$, their join $G_1\vee G_2$ is the graph with vertex set $V = V_1\cup V_2$ and edge set $E = E_1\cup E_2\cup \{uv \, : \, u\in V_1, v\in V_2\}$.
We write $[k]$ for the set of first $k$ natural numbers. Denote by $K_n$, $P_n$, and $C_n$, respectively, the complete graph, the path graph, and the cycle graph on $n$ vertices unless stated otherwise in a section.

A central graph invariant of our paper is the dissociation number,  defined as follows.
\begin{definition}
 A subset $S\subseteq V(G)$ is called a \emph{dissociation set} in $G$ if the graph induced by the vertices in $S$, denoted as $G[S]$, has maximum degree at most $1$. The size of a largest dissociation set of $G$ is called the \emph{dissociation number} of $G$, denoted by $\tau(G).$    
\end{definition}
One can see the dissociation number $\tau(G)$ as a generalization of the independence number $\alpha(G)$ of $G$. In 1962, Ore \cite{Ore} posed a question to compute the minimum graph size from the set of all connected graphs of order $n$ and independence number $\alpha$. In this paper, we study the corresponding problem for dissociation number as well as the related problem of finding graphs with the minimum spectral radius among all connected graphs of order $n$ and dissociation number $\tau$. We study these two problems in connection with Tur\'an-type problems for complete multipartite graphs (see Theorem~\ref{thm: turan connection}).

Given a fixed family $\mathcal{F}$ of graphs, the Tur\'an number of $\mathcal{F}$, denoted as $\ex(n, \mathcal{F})$, is the maximum number of edges in a graph on $n$ vertices that does not contain any graph  $F\in\mathcal{F}$ as a subgraph.  Let $\EX(n, \mathcal{F})$ denote the set of $\mathcal{F}$-free graphs of order $n$ and size $\ex(n, \mathcal{F})$. We refer to these graphs as \emph{edge maximizers} (or \emph{edge extremal} graphs, as commonly used in the literature). Similar to the definition of $\ex(n, \mathcal{F})$, we use $\ex_{cc}(n, \mathcal{F})$ to denote the maximum number of edges in any $\mathcal{F}$-free graph on $n$ vertices with a connected complement, and $\EX_{cc}(n \mathcal{F})$ to denote the collection of all $\mathcal{F}$-free graphs on $n$ vertices with $\ex_{cc}(n, \mathcal{F})$ edges that have a connected complement. 
We denote by $\mathcal{D}_{n,\tau}$ the set of all connected graphs on $n$ vertices and dissociation number $\tau$. We denote the minimum spectral radius in $\mathcal{D}_{n,\tau }$ by $\rho_{min}(n,\tau)$ and the minimum size in $\mathcal{D}_{n,\tau}$ by $e_{min}(n,\tau)$. We refer to any graph in $\mathcal{D}_{n,\tau}$ whose spectral radius equals $\rho_{min}(n,\tau)$ as a \emph{spectral minimizer} and any graph in $\mathcal{D}_{n,\tau}$ with $e_{min}(\mathcal{D}_{n,\tau})$ number of edges as an \emph{edge minimizer}.

Determining spectral minimizers in $\mathcal{D}_{n,\tau}$ is a type of Brualdi-Solheid problem. In 1986, Brualdi and Solheid \cite{BruS} proposed the problem of characterizing graphs that achieve the maximum or minimum spectral radius in a given class of graphs. Since then, the problem has been extensively studied by several authors (see \cite{bellmax, BruSol, Vishal, DasM, DuShi, JiLu, JinZhang, Lou, LP, rowlinsonmax,  Simic, SimicMarziBelardo}, for example). Most papers deal with graphs that maximize the spectral radius under some prescribed structural invariants. However, studying graphs with the minimum spectral radius in a fixed graph family is equally important. For example, a smaller spectral radius in network models is directly linked to better virus protection and resistance to epidemic spread, making it essential for designing more resilient networks (see, for example, \cite{WangC}). Over the past two decades, several authors have studied the problem of characterizing graphs with fixed order and independence number that maximize or minimize the spectral radius (see \cite{DasM, DuShi, JiLu, JinZhang, Lou, Xu}). 
Recently, Huang, Geng, Li, and Zhou \cite{Huang} investigated this problem for a fixed dissociation number. The problem of determining the graphs that attain the maximum spectral radius in $\mathcal{D}_{n,\tau}$ turns out to be trivial. This is because the Perron-Frobenius Theorem guarantees that adding an edge to a connected graph strictly increases its spectral radius. Therefore, if we take the join of a maximum matching on a set of $\tau$ vertices and the complete graph $K_{n-\tau}$, we obtain the graph with the maximum spectral radius in $\mathcal{D}_{n,\tau}$ (see \cite[Theorem 1.1]{Huang}). However, the problem is not as easy for spectral minimizers. Huang et al. showed that for $\tau \geq \lceil \frac{2n}{3} \rceil$, a spectral minimizer is a tree. They also determined the spectral minimizers for $\tau\in\{2,\lceil \frac{2n}{3} \rceil, n-1, n-2\}.$ In \cite{dissociation2025paper}, Zhao, Liu, and Xiong resolved the case when $\tau =\lceil\frac{2n}{3} \rceil -1$. For the remaining values of $\tau$, the problem was open. 

In extremal graph theory, determining the Tur\'an numbers for a given forbidden graph is a central problem. In Section~\ref{sec: Edge minimizers and connection to Turan problems}, we establish a connection between the problem of finding the minimum size of a graph in $\mathcal{D}_{n, \tau}$ and the Tur\'an-type problems. In Section~\ref{sec: Turan number of odd cocktail party graphs L_{2k+1} for any n}, we calculate the Tur\'an number for odd cocktail party graphs (see Theorem~\ref{edge extremal graph}) and construct a set of graphs with size equal to the Tur\'an number. In Section~\ref{sec: Edge minimizer for given even dissociation number}, we then use it to determine the minimum number of edges $e_{min}(n,\tau)$ and a set of edge minimizers (refer to Theorem~\ref{edge minimizer}) within $\mathcal{D}_{n,\tau}$ for even values of $\tau$. As a result, we also obtain an upper bound on the minimum spectral radius of graphs in $\mathcal{D}_{n, \tau}$. In Section~\ref{sec: Spectral minimizers for fixed order n and  dissociation number tau =4}, we characterize the spectral minimizers in $\mathcal{D}_{n,4}$ (see Theorems~\ref{spec minimizer1} and \ref{spec minimizer2}).

Let $K_m(r_1, r_2, \ldots,$ $r_m)$  be the complete $m$-partite graph with parts of sizes $r_i$ for $1 \le i \le m$.  In Section~\ref{sec: edge minimizers d-independence}, we extend a result of Erd\H{o}s-Simonovits (see Theorem~\ref{thm: turan numbers multipartite graphs}) to prove that when $r_2 \ge (r_1 - 1)! + 1$ and $r_1 \le r_2 \le \ldots \leq r_m$, then any graph in $\EX_{cc}(n, K_{q+1}(r_1, \cdots, r_{q+1}))$  can be obtained by deleting $q-1$ edges from a graph in $\EX(n, K_{q+1}(r_1,$ $  \cdots, r_{q+1}))$ (see Theorem~\ref{thm: connected complement close to turan numbers multipartite graphs}). We apply these results in Sections~\ref{sec: lower bounds on spectral minimizers for d-independnce number} and \ref{sec: spectral minimizers for even dissociation numbers}. In Section~\ref{sec: lower bounds on spectral minimizers for d-independnce number}, we consider further generalizations of the independence number.
\begin{definition}
A subset of vertices in $G$ is said to be \emph{$d$-independent} if the subgraph induced by the subset has maximum degree at most $d$. The cardinality of the largest $d$-independent set in a graph $G$ is called the \textit{$d$-independence number} of $G$, denoted by $i_d(G)$.
\end{definition}
Note that when $d = 0$ or $d = 1$, the $d$-independence numbers of $G$ are just the independence number $\alpha(G)$, and the dissociation number $\tau(G)$, respectively. We apply Theorem~\ref{thm: connected complement close to turan numbers multipartite graphs} to derive lower bounds for the number of edges and spectral radius for graphs with a given $d$-independence number $s$, when the order $n$ is sufficiently large (see Proposition~\ref{thm: lower bound for spectral radius given d-independence number}).

Let $d$ be an even natural number. The cocktail party graph $CP_d$, also known as the hyperoctahedral graph, is the complement of a perfect matching in the complete graph $K_d$.  Lov\'asz and Pelik\'an \cite{LP} proved that among all simple connected graphs of order $n$ and size $n-1$, the path graph $P_n$ has the smallest spectral radius. 
From \cite[Proposition 3.1]{Vishal} we have that among all simple connected graphs of order $n$ and size $n$ (unicyclic graphs), the cycle graph $C_n$ is the unique spectral minimizer. This motivates the following two problems.
\begin{problem}\label{prob: path minimizer}
    If we replace each vertex of a tree of order $k$ with the cocktail party graph $CP_{m}$ such that the resultant graph is connected and has size $k-1 + e(CP_m)k$, then is is true that the path graph $P_k$ will yield a graph with the smallest spectral radius?

\end{problem}
\begin{problem}\label{prob: cycle minimizer}
    If we replace each vertex of a unicyclic graph of order $k$ with the cocktail party graph $CP_m$ such that the resultant graph is connected and has size $k+ e(CP_m)k$, then is it true that the cycle graph $C_k$ will yield a graph with the smallest spectral radius.
\end{problem}
These two problems and spectral radius minimization problems in general are quite challenging.  
In Section~\ref{sec: spectral minimizers for even dissociation numbers}, we show that our original problem of determining the spectral minimizers in $\mathcal{D}_{n, 2k}$ can be reduced to Problem~\ref{prob: path minimizer}. We prove that any $G$ which is a spectral minimizer in $\mathcal{D}_{n,2k}$ is also an edge minimizer in $\mathcal{D}_{n,2k}$ under some conditions on $n$ (see Theorem~\ref{spec is edge minimizer}). We will then describe a few structural properties of a spectral minimizer in $\mathcal{D}_{n,2k}$.  Next, we resolve Problem~\ref{prob: cycle minimizer} and prove a stronger result for $m  =\frac{n}{k}$. We show that among all graphs in $\mathcal{D}_{n,2k}$ that are not in the set of the edge minimizers, the aligned CP-cycle (see Definition~\ref{def: cp cycle}) has the smallest spectral radius.

In 1981, Yannakakis \cite{Yannakakis} first studied the problem of computing the value of $\tau(G)$ and showed that it is an NP-hard problem for bipartite graphs. Since then several bounds on the value of $\tau(G)$ have been obtained (see \cite{Bock}, for example). We conclude this paper with two new bounds on $\tau(G)$ (see Propositions~\ref{prop: hoffman type bound} and \ref{probab lower bound}).

\section{Edge minimizers and connection to Tur\' an problems}
\label{sec: Edge minimizers and connection to Turan problems}
In this section, we discuss an observation that connects the problem of finding graphs that have the minimum number of edges in $\mathcal{D}_{n, \tau}$ to Tur\'an-type problems. 
Recall that the cocktail party graph $CP_d$ is the complement of a perfect matching in the complete graph $K_d$, where $d$ is an even natural number. For odd values of $d$, we define an odd cocktail party graph $L_d$ to be the join of the cocktail party graph $CP_{d-1}$ and the complete graph $K_1$. 
\begin{lemma}\label{tau upper bound 1}
    Let $G = (V, E)$ be a graph on $n$ vertices and let $d>2$ be an even integer. If the complement graph $\overline{G}$ is $CP_d$-free, then $\tau(G) \le d-1.$ 
\end{lemma}

\begin{proof}
   Suppose for the sake of contradiction that $\tau(G) = k \geq d$. Then, there exists a $k$-subset $S\subseteq V$ such that the induced subgraph $G[S]$ is isomorphic to  $l(K_2)\cup(k-2l)K_1$ for some nonnegative integer $l$. Therefore, $\overline{G}[S]$ is isomorphic to $CP_{2l}\vee K_{k-2l}$ which contains $CP_k$ if $k$ is even, or $L_k$ if $k$ is odd. Thus, $CP_d\subseteq \overline{G}$, a contradiction.
\end{proof}
 Similarly, for odd values of $d$, we have the following result.
\begin{lemma}\label{tau upper bound 2}
  Let $G$ be a graph on $n$ vertices and let $d>3$ be an odd natural number. If the complement graph $\overline{G}$ is $L_d$-free, then $\tau(G) \le d-1.$  
\end{lemma}

In the next result, we show that the upper bound on the dissociation number we obtained in the previous two lemmas is attained if the complement graph of $G$ in the lemma is, in fact, an edge maximizer that does not contain the cocktail party graph or the odd cocktail party graph.
\begin{theorem}\label{thm: turan connection}
    If $H$ is an edge maximizer for $L_d$ (or $CP_d$ if $d$ is even), then $\tau(\overline{H})=d-1.$
\end{theorem}

 \begin{proof}
  If $d$ is odd and $H$ is an edge maximizer for $L_d$, then adding an edge (say $e$) to $H$ creates a copy $L$ of $L_d$ in $H+e$. Note that $CP_{d-1}\subseteq L-e$. Therefore, $CP_{d-1}\subseteq H$ and thus $\tau(\overline{H})\geq d-1$.  Similarly, if $d$ is even and $H$ is an edge maximizer for $CP_d$, adding an edge (say $e$) to $H$ creates a copy $C$ of $CP_d$ in $H+e$. Note that $L_{d-1}\subseteq C-e$. Therefore, $L_{d-1}\subseteq H$ and thus $\tau(\overline{H})\geq d-1$. This proves the result.
  \end{proof}

\section{Tur\'an number of odd cocktail party graphs $L_{2k+1}$ for any $n$} 
\label{sec: Turan number of odd cocktail party graphs L_{2k+1} for any n}
Let $k\in\mathbf{N}, k\geq 2$. We define a set of graphs $G_{n,2k}$ as follows.
\begin{definition}
 Let $n_1, \ldots, n_k\in \mathbb{N}$ such that $\sum_{i=1}^k n_i=n$ and for any $i\not=j$, $|n_i-n_j|\leq 2$ (with $|n_i-n_j|= 2$ only when both $n_i$ and $n_j$ are even). The set $G_{n,2k}$ consists of graphs obtained by adding a maximum matching within each part of the complete $k$-partite graph $K_{n_1, \ldots, n_k}$. 
\end{definition}
See Figures~\ref{graph in Gn2k} and \ref{graphs in G64} for examples. Note that all graphs in $G_{n, 2k}$ have the same size. We use $e(G_{n, 2k})$ to denote the size of the graphs in the set $G_{n, 2k}$. In this section, we will show that the graphs in the set $G_{n, 2k}$ are edge maximizers for the odd cocktail party graph $L_{2k+1}$ and hence $e(G_{n, 2k}) = ex(n, L_{2k+1})$. We will need the following lemma.
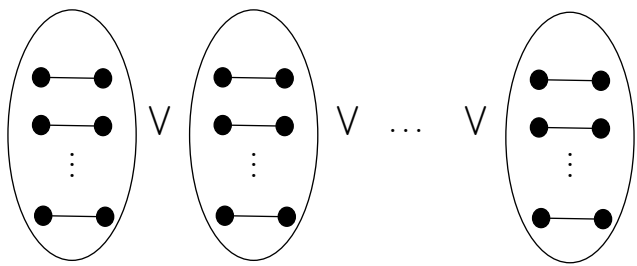
\begin{figure}
\label{fig: G_n,2k when n congruent to 0 mod 2k}
\begin{center}
\tikzset{every picture/.style={line width=0.5pt}} 
\begin{tikzpicture}[x=0.8pt,y=0.9pt,yscale=-1,xscale=1]

\draw   (78,513.5) .. controls (78,484.51) and (91.43,461) .. (108,461) .. controls (124.57,461) and (138,484.51) .. (138,513.5) .. controls (138,542.49) and (124.57,566) .. (108,566) .. controls (91.43,566) and (78,542.49) .. (78,513.5) -- cycle ;
\draw  [fill={rgb, 255:red, 0; green, 0; blue, 0 }  ,fill opacity=1 ] (122.56,485.72) .. controls (124.77,485.75) and (126.53,487.57) .. (126.5,489.77) .. controls (126.47,491.98) and (124.65,493.75) .. (122.44,493.71) .. controls (120.23,493.68) and (118.47,491.86) .. (118.5,489.66) .. controls (118.53,487.45) and (120.35,485.68) .. (122.56,485.72) -- cycle ;
\draw  [fill={rgb, 255:red, 0; green, 0; blue, 0 }  ,fill opacity=1 ] (93.56,485.29) .. controls (95.77,485.32) and (97.53,487.14) .. (97.5,489.34) .. controls (97.47,491.55) and (95.65,493.32) .. (93.44,493.28) .. controls (91.23,493.25) and (89.47,491.43) .. (89.5,489.23) .. controls (89.53,487.02) and (91.35,485.25) .. (93.56,485.29) -- cycle ;
\draw    (122.5,489.72) -- (93.5,489.28) ;
\draw  [fill={rgb, 255:red, 0; green, 0; blue, 0 }  ,fill opacity=1 ] (122.56,505.72) .. controls (124.77,505.75) and (126.53,507.57) .. (126.5,509.77) .. controls (126.47,511.98) and (124.65,513.75) .. (122.44,513.71) .. controls (120.23,513.68) and (118.47,511.86) .. (118.5,509.66) .. controls (118.53,507.45) and (120.35,505.68) .. (122.56,505.72) -- cycle ;
\draw  [fill={rgb, 255:red, 0; green, 0; blue, 0 }  ,fill opacity=1 ] (93.56,505.29) .. controls (95.77,505.32) and (97.53,507.14) .. (97.5,509.34) .. controls (97.47,511.55) and (95.65,513.32) .. (93.44,513.28) .. controls (91.23,513.25) and (89.47,511.43) .. (89.5,509.23) .. controls (89.53,507.02) and (91.35,505.25) .. (93.56,505.29) -- cycle ;
\draw    (122.5,509.72) -- (93.5,509.28) ;
\draw  [fill={rgb, 255:red, 0; green, 0; blue, 0 }  ,fill opacity=1 ] (123.56,543.72) .. controls (125.77,543.75) and (127.53,545.57) .. (127.5,547.77) .. controls (127.47,549.98) and (125.65,551.75) .. (123.44,551.71) .. controls (121.23,551.68) and (119.47,549.86) .. (119.5,547.66) .. controls (119.53,545.45) and (121.35,543.68) .. (123.56,543.72) -- cycle ;
\draw  [fill={rgb, 255:red, 0; green, 0; blue, 0 }  ,fill opacity=1 ] (94.56,543.29) .. controls (96.77,543.32) and (98.53,545.14) .. (98.5,547.34) .. controls (98.47,549.55) and (96.65,551.32) .. (94.44,551.28) .. controls (92.23,551.25) and (90.47,549.43) .. (90.5,547.23) .. controls (90.53,545.02) and (92.35,543.25) .. (94.56,543.29) -- cycle ;
\draw    (123.5,547.72) -- (94.5,547.28) ;
\draw   (163,513.5) .. controls (163,484.51) and (176.43,461) .. (193,461) .. controls (209.57,461) and (223,484.51) .. (223,513.5) .. controls (223,542.49) and (209.57,566) .. (193,566) .. controls (176.43,566) and (163,542.49) .. (163,513.5) -- cycle ;
\draw  [fill={rgb, 255:red, 0; green, 0; blue, 0 }  ,fill opacity=1 ] (207.56,485.72) .. controls (209.77,485.75) and (211.53,487.57) .. (211.5,489.77) .. controls (211.47,491.98) and (209.65,493.75) .. (207.44,493.71) .. controls (205.23,493.68) and (203.47,491.86) .. (203.5,489.66) .. controls (203.53,487.45) and (205.35,485.68) .. (207.56,485.72) -- cycle ;
\draw  [fill={rgb, 255:red, 0; green, 0; blue, 0 }  ,fill opacity=1 ] (178.56,485.29) .. controls (180.77,485.32) and (182.53,487.14) .. (182.5,489.34) .. controls (182.47,491.55) and (180.65,493.32) .. (178.44,493.28) .. controls (176.23,493.25) and (174.47,491.43) .. (174.5,489.23) .. controls (174.53,487.02) and (176.35,485.25) .. (178.56,485.29) -- cycle ;
\draw    (207.5,489.72) -- (178.5,489.28) ;
\draw  [fill={rgb, 255:red, 0; green, 0; blue, 0 }  ,fill opacity=1 ] (207.56,505.72) .. controls (209.77,505.75) and (211.53,507.57) .. (211.5,509.77) .. controls (211.47,511.98) and (209.65,513.75) .. (207.44,513.71) .. controls (205.23,513.68) and (203.47,511.86) .. (203.5,509.66) .. controls (203.53,507.45) and (205.35,505.68) .. (207.56,505.72) -- cycle ;
\draw  [fill={rgb, 255:red, 0; green, 0; blue, 0 }  ,fill opacity=1 ] (178.56,505.29) .. controls (180.77,505.32) and (182.53,507.14) .. (182.5,509.34) .. controls (182.47,511.55) and (180.65,513.32) .. (178.44,513.28) .. controls (176.23,513.25) and (174.47,511.43) .. (174.5,509.23) .. controls (174.53,507.02) and (176.35,505.25) .. (178.56,505.29) -- cycle ;
\draw    (207.5,509.72) -- (178.5,509.28) ;
\draw  [fill={rgb, 255:red, 0; green, 0; blue, 0 }  ,fill opacity=1 ] (208.56,543.72) .. controls (210.77,543.75) and (212.53,545.57) .. (212.5,547.77) .. controls (212.47,549.98) and (210.65,551.75) .. (208.44,551.71) .. controls (206.23,551.68) and (204.47,549.86) .. (204.5,547.66) .. controls (204.53,545.45) and (206.35,543.68) .. (208.56,543.72) -- cycle ;
\draw  [fill={rgb, 255:red, 0; green, 0; blue, 0 }  ,fill opacity=1 ] (179.56,543.29) .. controls (181.77,543.32) and (183.53,545.14) .. (183.5,547.34) .. controls (183.47,549.55) and (181.65,551.32) .. (179.44,551.28) .. controls (177.23,551.25) and (175.47,549.43) .. (175.5,547.23) .. controls (175.53,545.02) and (177.35,543.25) .. (179.56,543.29) -- cycle ;
\draw    (208.5,547.72) -- (179.5,547.28) ;
\draw   (311,514.5) .. controls (311,485.51) and (324.43,462) .. (341,462) .. controls (357.57,462) and (371,485.51) .. (371,514.5) .. controls (371,543.49) and (357.57,567) .. (341,567) .. controls (324.43,567) and (311,543.49) .. (311,514.5) -- cycle ;
\draw  [fill={rgb, 255:red, 0; green, 0; blue, 0 }  ,fill opacity=1 ] (355.56,486.72) .. controls (357.77,486.75) and (359.53,488.57) .. (359.5,490.77) .. controls (359.47,492.98) and (357.65,494.75) .. (355.44,494.71) .. controls (353.23,494.68) and (351.47,492.86) .. (351.5,490.66) .. controls (351.53,488.45) and (353.35,486.68) .. (355.56,486.72) -- cycle ;
\draw  [fill={rgb, 255:red, 0; green, 0; blue, 0 }  ,fill opacity=1 ] (326.56,486.29) .. controls (328.77,486.32) and (330.53,488.14) .. (330.5,490.34) .. controls (330.47,492.55) and (328.65,494.32) .. (326.44,494.28) .. controls (324.23,494.25) and (322.47,492.43) .. (322.5,490.23) .. controls (322.53,488.02) and (324.35,486.25) .. (326.56,486.29) -- cycle ;
\draw    (355.5,490.72) -- (326.5,490.28) ;
\draw  [fill={rgb, 255:red, 0; green, 0; blue, 0 }  ,fill opacity=1 ] (355.56,506.72) .. controls (357.77,506.75) and (359.53,508.57) .. (359.5,510.77) .. controls (359.47,512.98) and (357.65,514.75) .. (355.44,514.71) .. controls (353.23,514.68) and (351.47,512.86) .. (351.5,510.66) .. controls (351.53,508.45) and (353.35,506.68) .. (355.56,506.72) -- cycle ;
\draw  [fill={rgb, 255:red, 0; green, 0; blue, 0 }  ,fill opacity=1 ] (326.56,506.29) .. controls (328.77,506.32) and (330.53,508.14) .. (330.5,510.34) .. controls (330.47,512.55) and (328.65,514.32) .. (326.44,514.28) .. controls (324.23,514.25) and (322.47,512.43) .. (322.5,510.23) .. controls (322.53,508.02) and (324.35,506.25) .. (326.56,506.29) -- cycle ;
\draw    (355.5,510.72) -- (326.5,510.28) ;
\draw  [fill={rgb, 255:red, 0; green, 0; blue, 0 }  ,fill opacity=1 ] (356.56,544.72) .. controls (358.77,544.75) and (360.53,546.57) .. (360.5,548.77) .. controls (360.47,550.98) and (358.65,552.75) .. (356.44,552.71) .. controls (354.23,552.68) and (352.47,550.86) .. (352.5,548.66) .. controls (352.53,546.45) and (354.35,544.68) .. (356.56,544.72) -- cycle ;
\draw  [fill={rgb, 255:red, 0; green, 0; blue, 0 }  ,fill opacity=1 ] (327.56,544.29) .. controls (329.77,544.32) and (331.53,546.14) .. (331.5,548.34) .. controls (331.47,550.55) and (329.65,552.32) .. (327.44,552.28) .. controls (325.23,552.25) and (323.47,550.43) .. (323.5,548.23) .. controls (323.53,546.02) and (325.35,544.25) .. (327.56,544.29) -- cycle ;
\draw    (356.5,548.72) -- (327.5,548.28) ;

\draw (105,513) node [anchor=north west][inner sep=0.75pt]    {$\vdots $};
\draw (190,513) node [anchor=north west][inner sep=0.75pt]    {$\vdots $};
\draw (338,513) node [anchor=north west][inner sep=0.75pt]    {$\vdots $};
\draw (142,500) node [anchor=north west][inner sep=0.75pt]    {$\bigvee $};
\draw (230,500) node [anchor=north west][inner sep=0.75pt]    {$\bigvee $};
\draw (290,500) node [anchor=north west][inner sep=0.75pt]    {$\bigvee $};
\draw (255,510) node [anchor=north west][inner sep=0.75pt]    {$\ldots$};
\end{tikzpicture}
\end{center}
\caption{The graph in $G_{n,2k}$ when $n\equiv 0 \pmod {2k}.$}
\label{graph in Gn2k}
\end{figure}

\begin{lemma}\label{min degree}
Let $k\geq 2$ be an integer.  A graph $G$ on $2k+1$ vertices is $L_{2k+1}$ free if and only if $\delta(G)\leq 2k-2.$
\end{lemma}
\begin{proof}
    If $G$ contains $L_{2k+1}$ as a subgraph then $\delta(G)\geq 2k-1.$ Conversely, if $\delta(G)\geq 2k-1,$ then by the Hand-shaking lemma $G$ contains a vertex of degree $2k$ and hence $G$ contains a $L_{2k+1}.$
\end{proof}

\begin{remark}
    Similarly, we can prove that a graph $G$ on $2k$ vertices is $CP_{2k}$ free if and only if $\delta(G)\leq 2k-3.$

\end{remark}
\begin{theorem}\label{edge extremal graph}
  For any order $n$, the graphs in $G_{n,2k}$ are edge maximizers for $L_{2k+1}$. 
\end{theorem}

\begin{proof}
      We will prove this using induction on $n.$ For $n \leq 2k,$ the complete graph $K_n$ is clearly the edge maximizer graph for $L_{2k+1}$. Suppose the claim is true for all $n<N$. For $n = N$ let $\Gamma$ be an edge extremal graph for $L_{2k+1}.$ Let $H$ be a $2k$-vertex subgraph of $\Gamma$ of maximum possible size. By Lemma~\ref{min degree}, each vertex of $\Gamma\setminus H$ is adjacent to at most $2k-2$ vertices of $H$. 
      Therefore, $e(H, \Gamma\setminus H)\leq (2k-2)(N-2k)$, and we have 
      \[e(\Gamma) = e(H) + e(H, \Gamma\setminus H) + e(\Gamma\setminus H) \leq \binom{2k}{2} + (2k-2)(N-2k) + e(G_{N-2k,2k})=e(G_{N,2k}).\]
\end{proof}

\section{Edge minimizer for a given even dissociation number}
\label{sec: Edge minimizer for given even dissociation number}
In this section, we use the connection we made in Theorem~\ref{thm: turan connection} and the set of edge maximizers for a given odd cocktail party graph, obtained in the previous section, to compute the minimum size $e_{min}(\mathcal{D}_{n,\tau})$ of graphs in $\mathcal{D}_{n,\tau}$ when the dissociation number $\tau $ is even. Let $\tau = 2k$, for some positive integer $k$. For $k=1,$ we have the following.

\begin{proposition}
    Depending on whether $n$ is even or odd, the cocktail party graph $CP_n$ or the odd cocktail party graph $L_n$, is the edge minimizer among all graphs with dissociation number equal to 2. 
\end{proposition}
\begin{proof}
    Let $G$ be an edge minimizer in $\mathcal{D}_{n,2}$. We observe that $\delta(G)\ge n-2$ because if a vertex $v\in V(G)$ has degree $d(v)\le n-3$, then $\tau(G)\geq 3$ since $v$ along with any two of its non-neighbors forms a dissociation set of size 3, which is a contradiction. By the Handshaking lemma, $e(G) = \frac{1}{2}\left(\sum_{v\in V(G)}d(v)\right)$. Therefore, if $n$ is even, $G$ is an $n-2$-regular graph and $G =CP_n$ since it is the unique $n-2$ regular graph; otherwise when $G$ is odd, $G$ has exactly one vertex of degree $n-1$ and hence $G = L_n$. This completes the proof.
\end{proof}

\begin{lemma}\label{lemma for edge minimizer}
    Let $G_1$ be a graph in $\mathcal{D}_{n,k}$ and let $G_2$ be an edge minimizer in $\mathcal{D}_{n, k+1}$. Then the size of $G_1$ is greater than or equal to the size of $G_2$, i.e., $e(G_1)\geq e(G_2).$
\end{lemma}
\begin{proof}
Let $S\subseteq V(G_1)$ be a dissociation set of $G_1$ that contains $k$ vertices. For any vertex $v\in V(G_1)\setminus S$, we denote its set of neighbors in $V(G_1)\setminus S$ by $N_1(v) = \{u_1, \ldots, u_\alpha\}$ and its set of neighbors in $S$ by $N_2(v) = \{ s_1, \ldots, s_\beta\}$. Pick a vertex $v\in V(G_1)\setminus S$. Note that $\beta \not = 0,$ otherwise we get a contradiction on $\tau(G_1) = k$. If $\alpha \not=0$, then delete the edge $vu_i$ and add the edge $u_1u_i$ (if edge $u_1 u_i$ is not already present in $G_1$) for all $1<i\leq \alpha$. Similarly, delete the edge $v s_i$ and add $u_1 s_i$ (if $u_1 s_i$ is not already present in $G_1$) for all $1\leq i \leq \beta$. If $\alpha=0$, then pick some other vertex from the set $V(G_1)\setminus S$ that has neighbors in $V(G_1)\setminus S$ and repeat the above process. If there is no such vertex in $V(G_1)\setminus S$, then the subgraph induced by $V(G_1)\setminus S$ is a coclique.
 In that case, pick two vertices $v,u \in V(G_1)\setminus S$ such that the distance between them is minimum in $G_1$. 
 Then, since the maximum degree in the graph induced by $S$ is at most 1, the shortest path from $u$ to $v$ in $G_1$ is either $P_1 = (u, s, v)$ or $P_2 = (u, s, s', v)$, for some $s, s' \in S.$ If it is path $P_1$, then same as above, delete $v s_i$ and add $us_i$ (if $u_1 u_i$ is not already present in $G_1$) for all $s_i \in N_2(v)$, and if it is path $P_2$, then delete $vs_i$ for all $s_i \in N_2(v)$, but only add those $u s_i$ for which  $s_i \not= s'$, and finally add the edge $u v$. Denote this new graph by $G'_1$. 
 Note that  $e(G'_1)\leq e(G_1)$ and  $G'_1$ is connected. 
 To check the dissociation number, first observe that $\tau(G_1') \ge k$ since $S$ is also a dissociation set in $G_1'$. Further, suppose $\tau(G_1') \geq k+2.$ Let $S'\subseteq V(G_1')$ be a max dissociation set. Note that $v\in S'$, otherwise we get a contradiction on $\tau(G_1) = k.$ The set $S'\setminus \{v\}$ has $k+1$ or more elements, which again contradicts $\tau(G_1) = k.$ Therefore, $G'_1\in \mathcal{D}_{n, k+1}$ and hence, $e(G_2)\leq e(G_1)$.    
\end{proof}

The following theorem follows immediately from the previous lemma and the fact that $\tau(P_n) = \lceil\frac{2n}{3}\rceil$.
\begin{theorem}\label{tree edge minimizer}
  For  $\tau \geq \lceil\frac{2n}{3}\rceil,$ an edge minimizer in $\mathcal{D}_{n,\tau}$ is a tree.  
\end{theorem}

Assume $4\leq \tau \leq \lfloor\frac{2n}{3}\rfloor$ for the rest of this section. Note that if $\tau = 2k$ for $k \in \mathbb{N}$, then $n\geq 3k$. Let $\overline{G_{n, 2k}} = \{ G \mid  \overline{G}\in G_{n,2k}\}$. By Theorems~\ref{thm: turan connection} and \ref{edge extremal graph}, we deduce that the dissociation number $\tau(G)$ of a graph $G\in\overline{G_{n,2k}}$ equals $2k$. However, note that $G$ is not connected, it has at least 
$k$ components. 
We construct a new set of graphs $\mathcal{T}_{n,2k}$ as follows. 
\begin{definition}
Take all the graphs from the set $\overline{G_{n,2k}}$ that have exactly $k$ components, i.e., $n_i\geq 3$ for all $i\in [k]$. 
 In each of these graphs, connect the $k$ components by adding $k-1$ new edges to form a connected graph. The set $\mathcal{T}_{n,2k}$ is the collection of all resultant graphs.
\end{definition}
 
See Figures~\ref{graphs in G64} and \ref{graphs in M64} for an example. 
Denote by $e(\mathcal{T}_{n,2k})$ the size of graphs in the set $\mathcal{T}_{n,2k}$. We will show that the graphs in $\mathcal{T}_{n,2k}$ are edge minimizers in $\mathcal{D}_{n,2k}$, i.e., $e(\mathcal{T}_{n,2k}) = e_{min}(\mathcal{D}_{n,2k})$. To achieve this, we will need the following two results.
\begin{figure}
    \centering
    \begin{tikzpicture}[every path/.append style={thick}, scale = 1.2]
       \draw (0,0) ellipse (0.5cm and 1cm);
     \draw (1.5,0) ellipse (0.5cm and 1cm); 
     \foreach \x in {(-0.2,0.3), (0.2,0.3), (1.3,0.3), (1.7,0.3), (-0.2, -0.3), (0.2, -0.3)}
     {\draw[fill] \x circle[radius = 0.05cm];}
     \draw (-0.2, -0.3)--(0.2, -0.3);
     \draw (-0.2,0.3)--(0.2,0.3);
     \draw (1.3,0.3)--(1.7,0.3);
     \node[] at (0.75, 0) {$\vee$};
     \end{tikzpicture}
     \hspace{2cm}
 \begin{tikzpicture}[every path/.append style={thick}, scale = 1.2]
       \draw (0,0) ellipse (0.5cm and 1cm);
     \draw (1.5,0) ellipse (0.5cm and 1cm); 
     \foreach \x in {(-0.2,0.3), (0.2,0.3), (1.3,0.3), (1.7,0.3), (0, -0.3), (1.5, -0.3)}
     {\draw[fill] \x circle[radius = 0.05cm];}
     \draw (-0.2,0.3)--(0.2,0.3);
     \draw (1.3,0.3)--(1.7,0.3);
      \node[] at (0.75, 0) {$\vee$};
     \end{tikzpicture}
    \caption{The graphs in $G_{6,4}$.}
    \label{graphs in G64}
\end{figure}
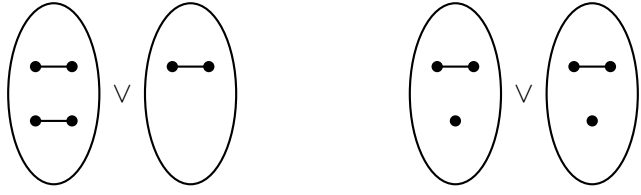
\begin{figure}
    \centering
\begin{tikzpicture}[every path/.append style={thick}, scale = 1.2]
       \draw (0,0) ellipse (0.5cm and 1cm);
     \draw (1.5,0) ellipse (0.5cm and 1cm); 
     \foreach \x in {(-0.2,0.3), (0.2,0.3), (1.3,0.3), (1.7,0.3), (0, -0.3), (1.5, -0.3)}
     {\draw[fill] \x circle[radius = 0.05cm];}
     \draw (0.2,0.3)--(0,-0.3)--(-0.2,0.3);
     \draw (1.3,0.3)--(1.5,-0.3)--(1.7,0.3);
     \draw (0.2,0.3)--(1.3,0.3);
     \end{tikzpicture}
     \hspace{1cm}
     \begin{tikzpicture}[every path/.append style={thick}, scale = 1.2]
       \draw (0,0) ellipse (0.5cm and 1cm);
     \draw (1.5,0) ellipse (0.5cm and 1cm); 
     \foreach \x in {(-0.2,0.3), (0.2,0.3), (1.3,0.3), (1.7,0.3), (0, -0.3), (1.5, -0.3)}
     {\draw[fill] \x circle[radius = 0.05cm];}
     \draw (0.2,0.3)--(0,-0.3)--(-0.2,0.3);
     \draw (1.3,0.3)--(1.5,-0.3)--(1.7,0.3);
     \draw (0,-0.3)--(1.5,-0.3);
     \end{tikzpicture}
     \hspace{1cm}
     \begin{tikzpicture}[every path/.append style={thick}, scale = 1.2]
       \draw (0,0) ellipse (0.5cm and 1cm);
     \draw (1.5,0) ellipse (0.5cm and 1cm); 
     \foreach \x in {(-0.2,0.3), (0.2,0.3), (1.3,0.3), (1.7,0.3), (0, -0.3), (1.5, -0.3)}
     {\draw[fill] \x circle[radius = 0.05cm];}
     \draw (0.2,0.3)--(0,-0.3)--(-0.2,0.3);
     \draw (1.3,0.3)--(1.5,-0.3)--(1.7,0.3);
     \draw (0,-0.3)--(1.3,0.3);
     \end{tikzpicture}
    \caption{The graphs in $\mathcal{T}_{6,4}$.}
    \label{graphs in M64}
\end{figure}
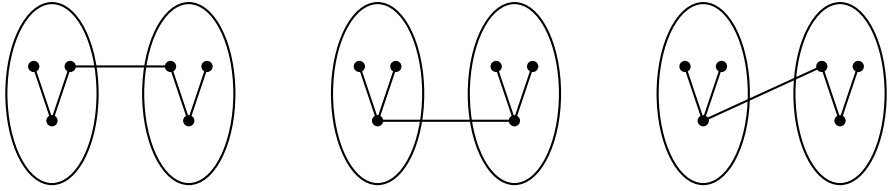
\begin{proposition}\cite{minimumkpathvertexcover}\label{prop: diss lower bound for trees}
    Let $G$ be a tree on $n$ vertices. Then its dissociation number $\tau(G)\geq \lceil\frac{2n}{3}\rceil$.
\end{proposition}
\begin{lemma}\label{lem: tree with k P3}
    Let $T$ be any tree on $3k$ vertices obtained by connecting $k$ copies of $P_3$ using $k-1$ edges. Then $\tau(T) = 2k$.
\end{lemma}
\begin{proof}
    We note that since $\tau(P_3)=2$, we have $\tau(T)\leq 2k$. Furthermore, using Proposition~\ref{prop: diss lower bound for trees} we get that $\tau(T)\geq 2k$. Therefore, $\tau(T)=2k$.
\end{proof}
\begin{theorem}\label{edge minimizer}
  The graphs in the set $\mathcal{T}_{n,2k}$ are edge minimizers in $\mathcal{D}_{n,2k}$ of size $e = \binom{n}{2}-ex(n, L_{2k+1}) + k-1.$
\end{theorem}
\begin{proof}
For $n=3k,$ the result follows from Lemma~\ref{lem: tree with k P3}.
For $n = 3k+r, 1\leq r\leq k$, let $\Gamma$ be an edge minimizer in $\mathcal{D}_{n,2k}$,
and let $S\subseteq \Gamma$ be a dissociation set of size $2k$ with the maximum number of vertices of degree 0 in the subgraph induced by $S$. Thus, each vertex in $\Gamma\setminus S$ is adjacent to at least 2 vertices in $S$ as otherwise, we either get a contradiction on $\tau(\Gamma) =2k$, or on the maximality of the size of the independence set in $S$.
Therefore, $e(\Gamma\setminus S, S)\geq 2k+2r.$ By the pigeonhole principle, there exists an $r$-vertex subset $S_r\subseteq S$ such that $e(\Gamma\setminus S, S_r)\geq 2r$. Since $\tau(\Gamma\setminus S_r)\leq 2k$, by Lemma~\ref{lemma for edge minimizer}, we have $e(\Gamma\setminus S_r)\geq e(\mathcal{T}_{3k,2k}) = 3k-1$.  Therefore, $$e(\Gamma)\geq e(S_r) + e(\Gamma\setminus S_r, S_r) + e(\Gamma\setminus S_r) \geq 2r + 3k-1 = e(\mathcal{T}_{3k+r,2k}).$$ 
Next, for $n = 4k+r, 1\leq r\leq k.$ Let $\Gamma$ be an edge minimizer in $\mathcal{D}_{n,2k}$. Same as above, let $S\subseteq \Gamma$ be a dissociation set of size $2k$ with the maximum number of vertices of degree 0 in the subgraph induced by $S$. Thus, each vertex in $\Gamma\setminus S$ is adjacent to at least 2 vertices in $S$.
Therefore, $e(\Gamma\setminus S, S) = \lambda\geq 4k+2r.$ By averaging, there exists a $k$-vertex subset $S^1_k\subseteq S$ such that $e(\Gamma\setminus S, S^1_k) = 2k+r+\epsilon, \epsilon\geq 0.$ Hence $e(\Gamma\setminus S, S\setminus S^1_k)=\lambda -2k-r-\epsilon\geq 2k+r-\epsilon.$ By the pigeonhole principle, there exists an $r$-vertex subset $S_r\subseteq S\setminus S^1_k$ such that $e(\Gamma\setminus S, S_r)\geq 3r-\epsilon .$ Thus, $e(\Gamma\setminus S, S^1_k\cup S_r)\geq 2k+r+\epsilon + 3r -\epsilon =    2k+4r.$ Since $\tau(\Gamma\setminus (S^1_k\cup S_r))\leq 2k$, by Lemma~\ref{lemma for edge minimizer}, we have $e(\Gamma\setminus (S^1_k\cup S_r))\geq e(\mathcal{T}_{3k,2k}) = 3k-1$.
Therefore,
$$ e(\Gamma)\geq e(S^1_k\cup S_r) + e(\Gamma\setminus (S^1_k\cup S_r), S^1_k\cup S_r) + e(\Gamma\setminus (S^1_k\cup S_r))\geq 4r+5k-1 = e(\mathcal{T}_{4k +r,2k} ).$$ 
This proves the claim is true for all $3k\leq n\leq 5k.$ Suppose the claim is true for all $n<N$.
Let $\Gamma$ be an edge minimizer in $\mathcal{D}_{N,2k}$. 
Let $S\subseteq \Gamma$ be a largest dissociation set with maximum number of vertices of degree 0 in the subgraph induced by $S$. Thus, each vertex in $\Gamma\setminus S$ is adjacent to at least 2 vertices in $S$. Therefore, $e(\Gamma\setminus S, S )\geq 2(N-2k)$.
Since $\tau(\Gamma\setminus S)\leq 2k,$ by Lemma~\ref{lemma for edge minimizer} and the induction hypothesis, we have 
    $$e(\Gamma)= e(S)+ e(\Gamma\setminus S,S ) + e(\Gamma\setminus S)\geq  2(N-2k) + e(\mathcal{T}_{N-2k,2k}) = e(\mathcal{T}_{N,2k}). $$ This completes the proof. 
\end{proof}

\section{Spectral minimizers for fixed order $n$ and  dissociation number $\tau =4$}
\label{sec: Spectral minimizers for fixed order n and  dissociation number tau =4}

Let $G$ be a graph in $\overline{G_{n,2k}}$ with $k$ components. Depending on how we connect these $k$ components of $G$, we get various graphs in $\mathcal{T}_{n,2k}$. Denote by $C_i$ the $i$-th component and by $e_{ij}  (i\not=j)$ the edge connecting the components $C_i$ and $C_j$. Let $n_i$ be the number of vertices in $C_i$. Note that each $C_i$ is isomorphic to $CP_{n_i}$ or $L_{n_i}$ depending on whether $n_i$ is even or odd, respectively. When $n_i$ is odd, then one vertex in $C_i$, call it $v_\Delta^{(i)}$, has degree $n_i-1$ and all remaining vertices have degree $n_i-2$ in the component $C_i$. In what follows, we will show that using the vertex $v_\Delta^{(i)}$ in a connecting edge $e_{ij}$  results in a graph with a larger spectral radius compared to using any other vertex from $C_i$ of degree $n_i-2$. When $\tau=2k=4$, the graph $G$ has only two components $C_1, C_2$. Therefore, we need only one connecting edge $e_{12}$ to get a connected graph. If $n_i$ $(i\in\{1,2\})$ is odd and $v_\Delta^{(i)}$ is incident to the connecting edge $e_{12}$, we represent the resultant graph as $G_{v_\Delta^{(i)}}$. Likewise, if both $n_1$ and $n_2$ are odd, and $e_{12} = v_\Delta^{(1)} v_\Delta^{(2)}$, we represent the resultant graph as $G_{v_\Delta^{(1)}v_\Delta^{(2)}}$.

For the next lemma, we will need the following graph operation. Consider two vertices $u$ and $v$ of a graph $H$. Construct a new graph $H_{v \to u}$ from $H$ by deleting an edge $vx$ and adding a new edge $ux$ for each vertex $x$ (except possibly $u$) that is adjacent to $v$ and not adjacent to $u$. The vertices $u,v$ are adjacent in $H_{v \to u}$ if and only if they are adjacent in $H$.  Kelmans \cite{Kelman} proved that this operation increases the spectral radius of the graph, namely that $\rho(H) \le \rho(H_{v \to u})$ and this procedure is now known as the Kelmans operation.
\begin{lemma}\label{bad connection}
    Let $G_{v_\Delta^{(1)}}$ be a graph in $\mathcal{T}_{n, 4}$ as described above. Consider the graph $G$ which we obtain from $G_{v_\Delta^{(1)}}$ by replacing the connecting edge $e_{12} = v_\Delta^{(1)}w$ with $e_{12} = uw$, where $u\in C_1$ of degree $n_1-2$ in $C_1$, $w\in C_2$. Then, $\rho(G_{v_\Delta^{(1)}}) > \rho(G)$.
\end{lemma}
\begin{proof}
Note that we can obtain the graph $G_{v_\Delta^{(1)}}$ from the graph $G$ by performing the Kelmans operation on the vertices $v_\Delta^{(1)}$ and $u$. Therefore, $\rho(G_{v_\Delta^{(1)}}) > \rho(G)$.
\end{proof}

In general, we prove the following.
\begin{lemma}
    Let $G$ be a graph in $\mathcal{T}_{n,2k}$ such that each of the $k$ parts have at least $k$ vertices in them. Suppose for some $i\not=j$, the connecting edge is $e_{ij} = v_{\Delta}^{(i)}w$, where $w\in C_j$. Then, $G$ is not a spectral minimizer in $\mathcal{T}_{n,2k}$. 
\end{lemma}
\begin{proof}
Since there are only $k-1$ connecting edges and each part has at least $k$ vertices, we can find a vertex $u\in C_i$ of degree $n_i-2$ in $G$ that is not incident to any of the connecting edges. Let $H\in\mathcal{T}_{n,2k}$ be the graph we obtain from $G$ by deleting the edge $v_{\Delta}^{(i)}w$ and adding the edge $uw$. Note that we can obtain $G$ from $H$ by performing the Kelmans operation on the vertices $v_{\Delta}^{(i)}$ and $u$. Therefore, $\rho(G)>\rho(H)$.
\end{proof}

Theorem~\ref{edge extremal graph} gives us that the set of graphs $G_{n, 4} \subset \EX(n, L_5)$, and $ex(n, L_5) = \lfloor \frac{n^2}{4} + \frac{n}{2}\rfloor$ when $n\not\equiv 2 \pmod 4$ and $ex(n, L_5) = \lfloor \frac{n^2}{4} + \frac{n}{2}\rfloor -1$ when $n\equiv 2\pmod 4$.
Note that the set $G_{n,4}$ has only one graph in it when $n\equiv 0,1,3 \pmod 4,$ and has two graphs when $n\equiv 2 \pmod 4.$ For simplicity, when $G_{n,4}$ is a singleton, we consider it as a graph instead of a set. Theorem~\ref{thm: turan connection} and Theorem~\ref{edge minimizer} give us that the set of graphs $\mathcal{T}_{n, 4}$ are edge minimizers in $\mathcal{D}_{n,4}$. 
For $n\geq 8$, let $\widehat{\mathcal{T}}_{n,4}\in \mathcal{T}_{n,4}$ be the graph in which $|n_1-n_2|\leq 1$ and the connecting edge $e_{12} = uv$ is such that the degree of its endpoints $d_{C_1}(u) = n_1-2$ and $d_{C_2}(v) = n_2-2$ in their respective components. We call such a connecting edge \emph{a good connecting edge}. Denote by $C_1-C_2$ any graph isomorphic to a graph where the components $C_1, C_2$ are connected by a good connecting edge in $\mathcal{T}_{n,4}$.
Depending on the value of $n \pmod 4$, $\widehat{\mathcal{T}}_{n,4}$ is the following graph.
\begin{align*}
\widehat{\mathcal{T}}_{n,4} = 
    \begin{cases}
 CP_{\frac{n}{2}}-CP_{\frac{n}{2}}, &\text{ if } n \equiv 0 \pmod4,\\CP_{\frac{n-1}{2}}-L_{\frac{n+1}{2}},  &\text{ if } n \equiv 1 \pmod4, \\
 L_{\frac{n}{2}}-L_{\frac{n}{2}},  &\text{ if } n \equiv 2 \pmod4,\\
  CP_{\frac{n+1}{2}}-L_{\frac{n-1}{2}}, &\text{ if } n \equiv 3 \pmod4.
\end{cases}
\end{align*}

In the remaining part of this section, we will show for $n\geq 8$, $\widehat{\mathcal{T}}_{n,4}$ is the only spectral minimizer in $\mathcal{D}_{n,4}$. For $n\in \{5,6\}$, since the dissociation number of the path graph $P_n$ is $4$, by \cite[Theorem 3]{LP} we conclude that $P_n$ is the unique spectral minimizer in $\mathcal{D}_{n,\tau}$. For $n=7$, since the dissociation number of cycle $C_7$ is $4$, by \cite[Proposition 3.1]{Vishal} we conclude that $C_7$ is the unique spectral minimizer in $\mathcal{D}_{7,4}$. We start with the following result. 

Let $G^-_{n,4}$ be the set of graphs we obtain by deleting an edge from the graphs in the set $G_{n,4}$.
\begin{proposition}\label{unique L5 free}
 For $n\geq 8$, the graphs in the set $\mathcal{T}_{n,4}$ are the only connected graphs of size $e_{min}(\mathcal{D}_{n,4})$ and dissociation number $4$.    
\end{proposition}
\begin{proof}
 To prove this, we will show that for $n\geq 8$ the graphs in the set $G_{n,4}$ and $G^-_{n,4}$ are the only $L_5$-free graphs of size $ex(n, L_5)$ and $ex(n, L_5)-1$, respectively. It suffices to show the latter, 
 as adding an edge to a graph in $G^-_{n,4}$ either gives a graph in  $G_{n,4}$ or creates an $L_5$ as a subgraph. We will prove this using induction on $n$. Let $\Gamma$ be an $L_5$-free graph of order $n$ and size $ex(n, L_5)-1$. 
 Let $H$ be a 4-vertex subgraph of $\Gamma$ of maximum possible size with the vertex set $V(H) = \{h_1, h_2, h_3, h_4\}$. 
 Using Lemma~\ref{min degree}, we see that $e(\Gamma\setminus H)\leq ex(n-4, L_5)$ and any vertex in $\Gamma\setminus H$ has at most two neighbors in $H$, so $e( \Gamma\setminus H, H)\leq 2(n-4)$. From this we conclude $e(H) \ge 5$.  For $n=5$, since $e(H)\geq 5$ and the fifth vertex, call it $v$, is adjacent to at most 2 vertices in $H$, we get the four graphs shown in Figure~\ref{G54-}. 
 Therefore $\Gamma\in G^-_{5,4}$ and $G_{5,4}$ is the only graph of order $5$ and size $8$ that is $L_5$-free. 
 Similarly, for $n=6$, the graphs in the set $G^-_{6,4}$ as well as the two graphs shown in Figure~\ref{G64-} are the only graphs of order $6$ and size $10$ that are $L_5$-free. 
 Note that adding an edge in the graphs in Figure~\ref{G64-} creates an $L_5$. 
 Therefore, graphs in $G_{6,4}$ are the only $L_5$-free graphs of order $6$ and size $11$.\\

 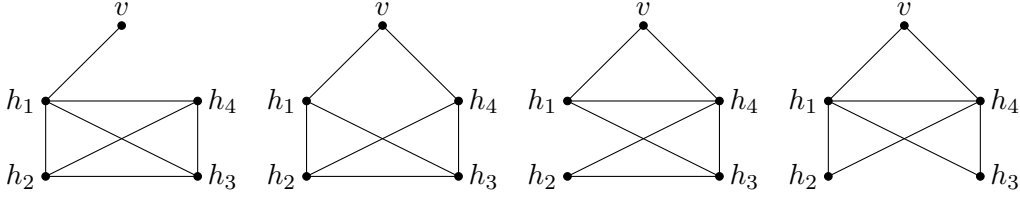
\begin{figure}
        \centering
\begin{tikzpicture}[scale=1]

\foreach \name/\x/\y in {
    v1/0/2, v2/0/1, v3/2/1, v4/2/2, v5/1/3}
    {
        \coordinate (\name) at (\x,\y);
        \draw[fill] (\name) circle[radius=0.05cm];
    }
\foreach \u/\v in {v1/v2, v1/v3, v1/v4, v2/v3, v2/v4, v3/v4, v5/v1}
    {
        \draw (\u) -- (\v);
    }
    
\node[anchor=east] at (v1) {$h_1$};
\node[anchor=east] at (v2) {$h_2$};
\node[anchor=west] at (v3) {$h_3$};
\node[anchor=west] at (v4) {$h_4$};
\node[anchor=south] at (v5) {$v$};
\end{tikzpicture}
\begin{tikzpicture}[scale=1]

\foreach \name/\x/\y in {
    v1/0/2, v2/0/1, v3/2/1, v4/2/2, v5/1/3}
    {
        \coordinate (\name) at (\x,\y);
        \draw[fill] (\name) circle[radius=0.05cm];
    }
\foreach \u/\v in {v1/v2, v1/v3, v2/v3, v2/v4, v3/v4, v5/v1, v5/v4}
    {
        \draw (\u) -- (\v);
    }
    
\node[anchor=east] at (v1) {$h_1$};
\node[anchor=east] at (v2) {$h_2$};
\node[anchor=west] at (v3) {$h_3$};
\node[anchor=west] at (v4) {$h_4$};
\node[anchor=south] at (v5) {$v$};
\end{tikzpicture}
\begin{tikzpicture}[scale=1]

\foreach \name/\x/\y in {
    v1/0/2, v2/0/1, v3/2/1, v4/2/2, v5/1/3}
    {
        \coordinate (\name) at (\x,\y);
        \draw[fill] (\name) circle[radius=0.05cm];
    }
\foreach \u/\v in { v1/v3, v1/v4, v2/v3, v2/v4, v3/v4, v5/v1, v5/v4}
    {
        \draw (\u) -- (\v);
    }
    
\node[anchor=east] at (v1) {$h_1$};
\node[anchor=east] at (v2) {$h_2$};
\node[anchor=west] at (v3) {$h_3$};
\node[anchor=west] at (v4) {$h_4$};
\node[anchor=south] at (v5) {$v$};
\end{tikzpicture}
\begin{tikzpicture}[scale=1]

\foreach \name/\x/\y in {
    v1/0/2, v2/0/1, v3/2/1, v4/2/2, v5/1/3}
    {
        \coordinate (\name) at (\x,\y);
        \draw[fill] (\name) circle[radius=0.05cm];
    }
\foreach \u/\v in {v1/v2, v1/v3, v1/v4, v2/v4, v3/v4, v5/v1, v5/v4}
    {
        \draw (\u) -- (\v);
    }
    
\node[anchor=east] at (v1) {$h_1$};
\node[anchor=east] at (v2) {$h_2$};
\node[anchor=west] at (v3) {$h_3$};
\node[anchor=west] at (v4) {$h_4$};
\node[anchor=south] at (v5) {$v$};
\end{tikzpicture}

\caption{Graphs in $G^-_{5,4}.$}
\label{G54-}
\end{figure}

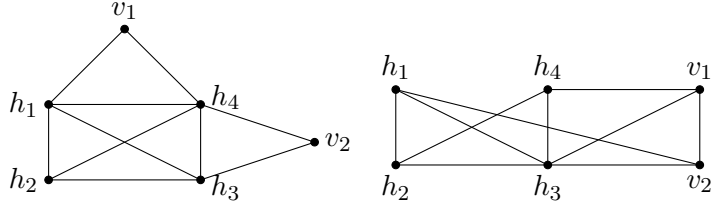
\begin{figure}
    \centering
\begin{tikzpicture}[scale=1]
\foreach \name/\x/\y in {
    v1/0/2, v2/0/1, v3/2/1, v4/2/2, v5/1/3, v6/3.5/1.5}
    {
        \coordinate (\name) at (\x,\y);
        \draw[fill] (\name) circle[radius=0.05cm];
    }
\foreach \u/\v in {v1/v2, v1/v3, v1/v4, v2/v3, v2/v4, v3/v4, v5/v1, v5/v4, v6/v4, v6/v3}
    {
        \draw (\u) -- (\v);
    }
    
\node[anchor=east] at (v1) {$h_1$};
\node[anchor=east] at (v2) {$h_2$};
\node[anchor=west] at (2,0.9) {$h_3$};
\node[anchor=west] at (2,2.1) {$h_4$};
\node[anchor=south] at (v5) {$v_1$};
\node[anchor=west] at (v6) {$v_2$};
\end{tikzpicture}
\begin{tikzpicture}[scale=1]
\foreach \name/\x/\y in {
    v1/0/2, v2/0/1, v3/2/1, v4/2/2, v5/4/2, v6/4/1}
    {
        \coordinate (\name) at (\x,\y);
        \draw[fill] (\name) circle[radius=0.05cm];
    }
\foreach \u/\v in {v1/v2, v1/v3, v2/v3, v2/v4, v3/v4, v5/v4, v6/v1, v6/v3, v5/v3, v5/v6}
    {
        \draw (\u) -- (\v);
    }
    
\node[anchor=south] at (v1) {$h_1$};
\node[anchor=north] at (v2) {$h_2$};
\node[anchor=north] at (v3) {$h_3$};
\node[anchor=south] at (v4) {$h_4$};
\node[anchor=south] at (v5) {$v_1$};
\node[anchor=north] at (v6) {$v_2$};
\end{tikzpicture}
    \caption{Graphs on $6$ vertices with $10$ edges that are not in the set $G^-_{6,4}$.}
    \label{G64-}
\end{figure}
 In the case of $n=7$, if $e(H)=6$, then there are two possibilities: either $e(\Gamma\setminus H, H) = 5$ and $e(\Gamma\setminus H)=3$, or $e(\Gamma\setminus H, H) = 6$ and $e(\Gamma\setminus H)=2$. Let $\{v_1, v_2, v_3\}$ be the vertex set of $\Gamma \setminus H$.  Suppose $e(\Gamma\setminus H, H) = 5$, and let $v_3$ is adjacent to only one vertex of $H$. Consider the graph $\Gamma-v_3$. It has $11$ edges, hence from the case of $n=6$, we conclude that $\Gamma-v_3\in G_{6,4}$. If $v_1, v_2$ are adjacent to $h_1, h_2$, then note that if $v_3$ is adjacent to either $h_1$ or $h_2$, the subgraph induced by the set $X= \{h_1, h_2, v_1, v_2, v_3\}$ contains a $L_5$, which is a contradiction. Thus, $v_3$ is adjacent to either $h_3$ or $h_4$, which proves $\Gamma\in G^-_{7,4}$. If $v_1\sim h_1, h_2$ and $v_2\sim h_3, h_4$, then $v_3$ adjacent to any one vertex of $H$ proves that $\Gamma\in G^-_{7,4}.$ Next, suppose $e(\Gamma\setminus H) = 2.$ Say $v_3\not\sim v_2$. Consider the graph $\Gamma-v_3$. It has $11$ edges, hence from the case of $n=6$, we conclude that $\Gamma-v_3\in G_{6,4}$. Suppose $v_1$ is adjacent to $h_1, h_2$. If $v_3\sim h_1$ and $v_3$ is adjacent to $h_3$ or $h_4$, say $v_3\sim h_3$, then the subgraph induced by the set $X = \{h_1, h_2, h_3, v_1, v_3\}$ contains a $L_5$, a contradiction. Thus, $v_3$ is either adjacent to $h_1, h_2$ or it is adjacent to $h_3, h_4$. Now, suppose $v_2$ is also adjacent to $h_1$ and $h_2$. Then, $v_3$ cannot be adjacent to either $h_1$ or $h_2$. Otherwise, the subgraph induced by the set $X = \{v_1, v_2, v_3, h_1, h_2\}$ would contain a $L_5$, which is a contradiction. Therefore, $\Gamma \in G^-_{7,4}$. Similarly, when $e(H)=5$, $\Gamma$ is either a graph in the set $G^-_{7,4}$ or it is the complement of cycle $C_7$. Finally, note that adding an edge to the complement graph of cycle $C_7$ creates an $L_5$. Therefore, $G_{7,4}$ is the only graph of order $7$ and size $15$ that is $L_5$-free. The proof of the remaining cases when $n \in \{8, 9\}$ uses similar arguments. For completeness, we provide the arguments below.\\
 
 For $n\in\{8,9\}$, we note that $e(H) = 6$. Otherwise, $e(\Gamma\setminus H) = ex(n-4, L_5)$, which implies that $\Gamma\setminus H$ is either $K_4$ (when $n=8$) or $G_{5,4}$ (when $n=9$). However, both of these graphs contain the complete graph $K_4$, which contradicts our choice of $H$. 
 In the case of $n=8$, we have $H = K_4$. This means either $e(\Gamma\setminus H) = 5$ and $e(\Gamma\setminus H, H) = 8$, or $e(\Gamma\setminus H) = 6$ and $e(\Gamma\setminus H, H) = 7$. Let $\{v_1, v_2, v_3, v_4\}$ be the vertex set of $\Gamma\setminus H$. 
 If $e(\Gamma\setminus H) = 5$, let $v_1$ be a vertex of minimum degree in $\Gamma\setminus H$, and if $e(\Gamma\setminus H, H) = 7$, let $v_1$ be the vertex which is adjacent to only one vertex of $H$. Consider the graph $\Gamma - v_1$. It has 15 edges. 
 The previous case for $n=7$ shows that $\Gamma-v_1$ is the graph $G_{7,4}$. Suppose $v_3, v_4$ are adjacent to both $h_3, h_4$. 
 If $v_1$ is not adjacent to $h_3$ and $h_4$, then we are done. Suppose not. Let $v_1\sim h_3$. If $v_1\sim v_3, v_4$, then the subgraph induced by the set $X = \{v_1, v_3, v_4, h_3, h_4\}$ contains $L_5$, a contradiction. If $v_1$ is not adjacent to both $v_3$ and $v_4$, say $v_1\not \sim v_4$, then $v_1$ is adjacent to two vertices in $H$. If $v_1\sim h_3, h_4$, then the subgraph induced by same set $X \{v_1, v_3, v_4, h_3, h_4\}$ contains $L_5$, a contradiction, else say if $v_1\sim h_3, h_2$, then the subgraph induced by the set $X = \{v_1, v_2, v_3, h_2, h_3\}$ contains $L_5$, a contradiction. This proves $\Gamma\in G^-_{8,4}$.\\

 In the case of $n=9$, we have $H = K_4$. Then there are two possibilities: either $e(\Gamma\setminus H) = 7$ and $e(\Gamma\setminus H, H) = 10$, or $e(\Gamma\setminus H) = 8$ and $e(\Gamma\setminus H, H) = 9$. Let $\{v_1, v_2, \ldots, v_5\}$ be the vertex set of $\Gamma\setminus H$. If $e(\Gamma\setminus H)=8,$ then previous case for $n=5$, shows that $\Gamma\setminus H$ is $G_{5,4}$. Let $v_5$ be the vertex of minimum degree in $\Gamma\setminus H$. Consider the graph $\Gamma - v_5$. If $v_5$ is adjacent to only one vertex in $H$, then $\Gamma -v_5$ has $20$ edges and, by the previous case of $n=8$, we conclude that $\Gamma-v_5$ is $G_{8,4}$. Otherwise, if $v_5$ is adjacent to two vertices of $H$, then $\Gamma-v_5$ has $19$ edges and, by the previous case of $n=8$, we conclude that $\Gamma-v_5\in G^-_{8,4}$. Suppose that $v_5\sim v_1, v_2$ in $\Gamma\setminus H$, and that $v_1, v_2$ are not adjacent to $h_3, h_4$. Then we observe that $v_5$ is not adjacent to $h_1$ or  $h_2$, otherwise the subgraph induced by set $X = \{h_1, h_2, v_1, v_2, v_5\}$ would contain an $L_5$, which is a contradiction. This proves $\Gamma \in G^-_{9,4}$.
 The proof arguments are similar for the remaining case when $e(\Gamma\setminus H) = 7$.\\  

 Next, suppose the result is true for all $n<N$. For $n=N$, where $N$ is at least $10$, let $\Gamma$ be an $L_5$-free graph of size $ex(N, L_5)-1.$ 
 Same as above, note that $e(H) = 6$, since otherwise $e(\Gamma\setminus H)= ex(N-4, L_5)$ and by induction hypothesis  $\Gamma\setminus H \cong G_{N-4,4} \supseteq K_4$, contradicting our choice of $H$.  We have the following two cases.\\

{\bf Case 1 :} $e( \Gamma\setminus H, H) =  2(N-4).$ This means each vertex of $\Gamma\setminus H$ is adjacent to exactly two vertices of $H$ and furthermore, $e(\Gamma\setminus H) = ex(N-4, L_5)-1.$ By induction hypothesis, $\Gamma\setminus H \in G^-_{N-4,4}$ (except when $n\in\{10,11\}$, if $\Gamma\setminus H$ is a graph in Figure \ref{G64-} for $n=10$, or $\Gamma\setminus H$ is the complement of cycle $C_7$ for $n=11$, then using similar arguments as above we can show that $\Gamma$ contains $L_5$). Let $\Gamma\setminus H = G - e$, where $e = \{u,v\}$ is an edge in $G$ and $G\in G_{N-4, 4}$ with the partite sets $A, B$. Let $|A|\leq |B|$. Suppose $a_1\in A$ is adjacent to $h_1, h_2$. We need to show all the vertices of $A$ are adjacent to $h_1, h_2$ and all the vertices of $B$ are adjacent to $h_3, h_4$. For $n\not \equiv 3 \pmod 4$, pick a vertex $b\in B$ such that the degree of $b$ in the subgraph induced by $B$ is minimum.  Consider the graph $\Gamma - b$ and note that if $b$ is an endpoint of the edge $e = \{u,v\}$, then $e(\Gamma - b) = ex(N-1, L_5)$ , and by the induction hypothesis, $\Gamma - b \in G_{N-1,4}$, else $e(\Gamma - b) = ex(N-1, L_5) -1$ and by the induction hypothesis, $\Gamma - b \in G^-_{N-1,4}$. Thus, all the vertices of $A$ are adjacent to $h_1, h_2$ and all the vertices of $B\setminus \{b\}$ are adjacent to $h_3, h_4$. It remains to show that $b$ is not adjacent to either $h_1$ or $h_2$. Suppose $b\sim h_1$. 
Then pick two vertices from $A$, say $a_s, a_t$, such that $b\sim a_s, a_t$. 
The subgraph induced by the set $X = \{a_s, a_t, b, h_1, h_2\}$ contains $L_5$, a contradiction. For $n\equiv 3 \pmod 4$, instead of one vertex, we pick two vertices $b_1, b_2$ from the set $B$ with the minimum possible degree in the subgraph induced by the vertex set $B$ in $\Gamma\setminus H$ and repeat the same process as described above.\\

 {\bf Case 2 :}  $e( \Gamma\setminus H, H) =  2(N-4)-1$. This means except for one vertex, say $x$, all vertices of $\Gamma\setminus H$ are adjacent to exactly two vertices in $H$, and $x$ is adjacent to only one vertex in $H$. Furthermore, we get $e(\Gamma\setminus H) = ex(N-4, L_5)$ and hence, $\Gamma\setminus H \in G_{N-4, 4}$ by the induction hypothesis. Suppose $A, B$ are the partite sets of $\Gamma\setminus H$ and let $|A|\leq |B|$. Suppose $a_1\in A$ is adjacent to $h_1, h_2$. We need to show that no vertex of $A$ is adjacent to $h_3$ or $h_4$ and no vertex of $B$ is adjacent to $h_1$ or $h_2$. For $n\not \equiv 3 \pmod 4$, pick a vertex $b\in B$ such that the degree of $b$ in the subgraph induced by the vertex set $B$ is minimum in $\Gamma\setminus H$. Consider the graph $\Gamma - b$. 
 If $b=x$, then $e(\Gamma - b) = ex(N-1, L_5)$ , and by the induction hypothesis, $\Gamma - b \in G_{N-1,4}$, else $e(\Gamma - b) = ex(N-1, L_5) -1$ and by the induction hypothesis, $\Gamma - b \in G^-_{N-1,4}$. Thus, no vertex of $A$ is adjacent to $h_3$ or $h_4$ and no vertex of $B\setminus\{b\}$ is adjacent to $h_1$ or $h_2$. It remains to show that $b$ is not adjacent to either $h_1$ or $h_2$. Suppose $b\sim h_1$. 
 Then pick two vertices from $A$, say  $a_s, a_t$, such that $a_s\sim a_t$. The subgraph induced by the set $X = \{a_s, a_t, b, h_1, h_2\}$ contains $L_5$, a contradiction. For $n\equiv 3 \pmod 4$, instead of one vertex, we pick two adjacent vertices $b_1, b_2$ from the set $B$ and repeat the same process as described above. 
 
 The proof completes by observing that the complement of a graph in the set $G^-_{n,4}$ is either disconnected or isomorphic to a graph in $\mathcal{T}_{n,4}$.
\end{proof}

\begin{theorem}\label{spec minimizer1}
 Let $n \equiv 0 \pmod 4$. The graph $\widehat{\mathcal{T}}_{n,4}$ is the unique spectral minimizer in $\mathcal{D}_{n, 4}.$ 
\end{theorem}

\begin{proof}
Let $e_{12} = uv$ be the connecting edge that links the two cocktail party graphs, $CP_{\frac{n}{2}}$, in $\widehat{\mathcal{T}}_{n,4}$. Let $u'$ (respectively $v'$) be the non-neighbor of $u$ (respectively $v$) in its respective cocktail party graph. Consider the partition of $\widehat{\mathcal{T}}_{n,4}=(V,E)$ into the following three parts: $\{u', v'\}, \{u, v\}, V\setminus\{u',u,v',v\}.$ The corresponding quotient matrix and the characteristic polynomial are 
$$Q_{\widehat{\mathcal{T}}_{n,4}} =\begin{bmatrix}
    0&0&\frac{n}{2}-2\\0&1&\frac{n}{2}-2\\1&1&\frac{n}{2}-4
    \end{bmatrix}$$ 
    and 
    $$P(x) = x^3 + \left(3-\frac{n}{2}\right)x^2-\frac{n}{2}x +\frac{n}{2}-2,$$ 
    respectively. Since the partition is equitable, $\rho(Q_{\widehat{\mathcal{T}}_{n,4}}) = \rho(\widehat{\mathcal{T}}_{n,4}).$ Note that when $x= \frac{n}{2}+\frac{4}{n}-2$, 
    
    $$P\left(\frac{n}{2}+\frac{4}{n}-2\right) = \frac{n}{2}-4+\frac{16}{n}-\frac{48}{n^2}+\frac{64}{n^3}>0$$
    
    for all $n>7$ and since $P(x)$ is monotonically increasing for $x\geq \frac{n}{2}+\frac{4}{n}-2,$ we get $\rho(\widehat{\mathcal{T}}_{n,4})<\frac{n}{2}+\frac{4}{n}-2$. Any other graph $H$ with $\tau(H) = 4$ has average degree $\frac{2e}{n} \geq \frac{n}{2}+\frac{4}{n}-2$ by Proposition~\ref{unique L5 free}, hence $\rho(H)>\rho(G).$ This completes the proof.  
\end{proof}

\begin{theorem}\label{spec minimizer2}
    Let $n\equiv 1,2,3 \pmod 4$. The graph $\widehat{\mathcal{T}}_{n,4}$ is the unique spectral minimizer in $\mathcal{D}_{n, 4}.$ 
\end{theorem}
\begin{proof}
 Let $n\equiv 2 \pmod 4$. There are four non-isomorphic graphs in $\mathcal{T}_{n,4}$. Following the notation of Lemma~\ref{bad connection}, let $v_\Delta^{(i)}\in C_i$ be the vertex of degree $n_i-1$ in $C_i$, if $n_i$ is odd. The four graphs are $G_{v_\Delta^{(1)}}$, where the connecting edge $e_{12} = \{v_\Delta^{(1)}, w\}$, for $w\not=v_\Delta^{(2)} \in C_2$, $G_{v_\Delta^{(1)}v_\Delta^{(2)}}$, where $e_{12} = \{v_\Delta^{(1)}, v_\Delta^{(2)}\}$,
 $\widehat{\mathcal{T}}_{n,4}= L_{\frac{n}{2}}-L_{\frac{n}{2}}$, and $CP_{\frac{n-2}{2}}-CP_{\frac{n+2}{2}}$ (when $n>6$). By Lemma~\ref{bad connection} we get $$\rho(\widehat{\mathcal{T}}_{n,4})<\rho(G_{v_\Delta^{(1)}})<\rho(G_{v_\Delta^{(1)}v_\Delta^{(2)}})$$ as we can obtain $G_{v_\Delta^{(1)}}$ from $\widehat{\mathcal{T}}_{n,4}$ and $G_{v_\Delta^{(1)}v_\Delta^{(2)}}$ from $G_{v_\Delta^{(1)}}$ by doing Kelmans operation. Let $e_{12} = uv$ be the connecting edge that links the two odd cocktail party graphs, $L_{\frac{n}{2}}$, in $\widehat{\mathcal{T}}_{n,4}$. Let $u'$ (respectively $v'$) be the non-neighbor of $u$ (respectively $v$) in its respective cocktail party graph. Consider the partition of $\widehat{\mathcal{T}}_{n,4} = (V,E) $ into the following four parts (in order): $\{u', v'\}, \{u, v\}, \{v_1, v_2\}, V\setminus\{u',u,v',v, v_1, v_2\}.$  The corresponding quotient matrix and the characteristic polynomial are
    $$Q_{\widehat{\mathcal{T}}_{n,4}} = \begin{bmatrix}
        0&0&1&\frac{n}{2}-3\\
        0&1&1&\frac{n}{2}-3\\
        1&1&0&\frac{n}{2}-3\\
        1&1&1&\frac{n}{2}-5
    \end{bmatrix},$$ and
    $$P_{\widehat{\mathcal{T}}_{n,4}}(x) = x^4 +\left(4-\frac{n}{2}\right)x^3 +(2-n)x^2 -3x +\frac{n}{2}-1.$$
    
Since the partition is equitable, $\rho(Q_{\widehat{\mathcal{T}}_{n,4}}) = \rho(\widehat{\mathcal{T}}_{n,4}).$ When $x = \frac{n}{2}+\frac{6}{n}-2,$
$$P_{\widehat{\mathcal{T}}_{n,4}}\left(\frac{n}{2}+\frac{6}{n}-2\right) = \frac{n^2}{4}-3n+24 - \frac{114}{n}+\frac{396}{n^2}-\frac{864}{n^3}+\frac{1296}{n^4}>0$$
for any $n\geq 6$. Since $P_{\widehat{\mathcal{T}}_{n,4}}(x)$ is increasing for $x\geq \frac{n}{2}+\frac{6}{n}-2$, we get that $\rho(\widehat{\mathcal{T}}_{n,4})<\frac{n}{2}+\frac{6}{n}-2$. 
Therefore, any graph $H\in \mathcal{D}_{n,4}$ with average degree $\frac{2e}{n}\geq \frac{n}{2}+\frac{6}{n}-2$ has spectral radius $\rho(H)> \rho(\widehat{\mathcal{T}}_{n,4}).$ By Proposition~\ref{unique L5 free}, it remains to show $\rho(CP_{\frac{n-2}{2}}-CP_{\frac{n+2}{2}})>\rho(\widehat{\mathcal{T}}_{n,4}).$ Let $e_{12} = u_2w_2$ be the edge connecting the two components in $CP_{\frac{n-2}{2}}-CP_{\frac{n+2}{2}}$. Consider the partition of $CP_{\frac{n-2}{2}}-CP_{\frac{n+2}{2}}$ into the following six parts (in order) : $\{u_1\}, \{w_1\},\{u_2\},\{w_2\},$ $V(C_1)\setminus\{u_1,u_2\}, V(C_2)\setminus\{w_1,w_2\},$ where $C_1 = CP_{\frac{n}{2}-1}$, $C_2 = CP_{\frac{n}{2}+1}, u_1, u_2$ is a non-adjacent pair in $C_1$, and $w_1, w_2$ is a non-adjacent pair in $C_2.$ The corresponding quotient matrix and the characteristic polynomial are 
       $$Q_{CP_{\frac{n-2}{2}}-CP_{\frac{n+2}{2}}} = \begin{bmatrix}
        0&0&0&0&\frac{n}{2}-3&0\\
        0&0&0&0&0&\frac{n}{2}-1\\
        0&0&0&1&\frac{n}{2}-3&0\\
        0&0&1&0&0&\frac{n}{2}-1\\
        1&0&1&0&\frac{n}{2}-5&0\\
        0&1&0&1&0&\frac{n}{2}-3
    \end{bmatrix},$$ and  
$$
P_{CP_{\frac{n-2}{2}}-CP_{\frac{n+2}{2}}}(x)=\frac{1}{2}(2x^2 +x(4-n)-4)q(x) +r(x),   $$ 
where $$q(x) = x^4 + \left(6-\frac{n}{2}\right)x^3 + (12-2n)x^2 +(8-2n)x + 1+n-\frac{n^2}{4},$$ and $$r(x) = -\left(\frac{n^3}{8}-\frac{n^2+n}{2}\right)x-\frac{3n^2}{4}+4n-1.$$ Since the partition is equitable, $\rho(Q_{CP_{\frac{n-2}{2}}-CP_{\frac{n+2}{2}}}) = \rho(CP_{\frac{n-2}{2}}-CP_{\frac{n+2}{2}}).$\\
When $x = $ $ \rho(G_{v_\Delta^{(1)}v_\Delta^{(2)}}) = \frac{n-4+\sqrt{n^2-8n+48}}{4}$, we have
\begin{align*}
 P_{CP_{\frac{n-2}{2}}-CP_{\frac{n+2}{2}}}( \rho(G_{v_\Delta^{(1)}v_\Delta^{(2)}})) &= 0 + r( \rho(G_{v_\Delta^{(1)}v_\Delta^{(2)}}))\\ 
 &=\resizebox{0.6\textwidth}{!}{$-\frac{(n^3-4n^2-4n)(n-4+\sqrt{n^2-8n+48})}{32}-\frac{3n^2}{4}+4n-1 <0$}
\end{align*}

for all $n>4.$ This implies $\rho(G_{v_\Delta^{(1)}v_\Delta^{(2)}})<\rho(CP_{\frac{n-2}{2}}-CP_{\frac{n+2}{2}}).$ Therefore, we have 
$$\rho(\widehat{\mathcal{T}}_{n,4})<\rho(G_{v_\Delta^{(1)}})<\rho(G_{v_\Delta^{(1)}v_\Delta^{(2)}})<\rho(CP_{\frac{n-2}{2}}-CP_{\frac{n+2}{2}}).$$

This completes the proof for $n\equiv 2\pmod 4.$ The proof of the remaining two cases follows the same idea and the same calculations.
\end{proof}

\section{Extending a result of Erd\H{o}s-Simonovits and a stability result}
\label{sec: edge minimizers d-independence}

In Theorem~\ref{edge extremal graph} we determined edge maximizer graphs when forbidding $L_{2k+1}$. We used this to determine edge and spectral minimizers among connected graphs on $n$ vertices with dissociation 4 in Theorem~\ref{edge minimizer} and Proposition~\ref{unique L5 free}, respectively. In this section we will characterize edge maximizer graphs for several other complete multipartite graphs. We will also prove a stability result which will allow us to characterize edge maximizer graphs while requiring additionally that their compliments are connected, for these same complete multipartite graphs.
As a corollary, we also determine edge minimizers among connected graphs on $n$ vertices with given even dissociation number, when $n$ satisfies some criteria. 
We will use the results of this section to prove bounds for edge and spectral minimizer results in Section~\ref{sec: lower bounds on spectral minimizers for d-independnce number} and spectral minimizers in Section~\ref{sec: spectral minimizers for even dissociation numbers}. Later, we will use the structure of extremal graphs for forbidden complete multipartite graphs  in a companion paper \cite{DVSpTurcompemult}, on the spectral Tur\' an problem for joins of complete multipartite graphs.

The results in this section are asymptotic in nature and are true for $n$ sufficiently large.
\subsection{Main results and supporting lemmas for Tur\' an problems}
Here and now onward, we will use $K_m(r_1, r_2, \ldots,$ $r_m)$ to denote the complete $m$-partite graph with parts of sizes $r_i$ for $1 \le i \le m$.  We use $G(n)$ to denote a graph on $n$ vertices throughout this section.
\begin{remark}
\label{rem: ErdosSim on general structure of extremal graphs}
In \cite{ErdosSimCompletemultipartite} Erd\H{o}s and Simonovits compile their results from  \cite{limittheoremErdosSimonovits, erdos1966new,  simonovits1968method} where they had already proved the following results for a finite arbitrary family $\mathcal{F}$ of forbidden graphs. Let $q := \min \{\chi(F) \mid F \in \mathcal{F}\}-1$.
\begin{enumerate}
    \item \label{rem: ErdosSim on general structure of extremal graphs 1.} Then $\ex(n, \mathcal{F}) = \binom{n}{2} (1 - \frac{1}{q} + o(1))$ as $n\to \infty$.  
    \item \label{rem: ErdosSim on general structure of extremal graphs 2.} For $n$ sufficiently large and any $G(n) \in \EX(n, \mathcal{F})$,  there exists a positive integer $r$ (that depends on some coloring properties of the elements of $\mathcal{F}$) such that the following hold:
    \begin{enumerate}
        \item[a.] $G(n)$ can be obtained from a graph product $\bigvee_{i=1}^{q} N_i$ by adding and deleting at most $O(n^{2-\frac{1}{r}})$ edges.
        \item[b.] Each part $N_i$ of the join $\bigvee_{i=1}^{q} N_i$ has roughly the same number of vertices, that is, $|N_i| = n_i = \frac{n}{q} + O(n^{1-\frac{1}{r}})$.
        \item[c.] Each vertex $v \in G(n)$ has degree $d(v) > \frac{n}{q}(q-1) - c_1 n^{1-\frac{1}{r}}$, for some suitable constant $c_1$.
        \item[d.] For any constant $\epsilon > 0$, there exists a constant $k_{\epsilon}$, such that there are at most $k_{\epsilon}$ vertices $v \in N_i$, where $v$ is adjacent to at least $\epsilon n_i$ neighbors of $N_i$. 
    \end{enumerate}
\end{enumerate}
    
\end{remark}

Using these results, Erd\H{o}s and Simonovits \cite{ErdosSimCompletemultipartite}  proved the following theorem.
\begin{theorem}
\label{thm: Erdos-Sim comp. multipartite}
    Let $r_1 = 1,2$ or $3$, $r_1 \le r_2 \le \ldots \le r_{q+1}$ be given integers. 
    If $n$ is large enough, then each extremal graph $G(n)$ in $\EX(n, K_{q+1}(r_1, \ldots, r_{q+1}))$ is of the following form:
    \[G(n) = \bigvee_{i=1}^q N_i\] where 
    \begin{enumerate}
        \item $n_i = v(N_i) = \frac{n}{q} + o(n)$;
        \item $N_1$ is an extremal graph for $K_2(r_1, r_2)$;
        \item $N_2, \ldots, N_q$ are extremal graphs for $K_2(1, r_2)$.
    \end{enumerate}    
        Conversely, if $\hat{N_1}, \ldots, \hat{N_q}$ are given graphs such that
    \begin{enumerate}
            \item[4.] there exists an extremal graph $\bigvee_{i=1}^q N_i$ satisfying 1., 2., 3. such that $v(\hat{N}_i) = v(N_i)$;
            \item[5.] $\hat{N}_i$ is an extremal graph for $K_2(r_1, r_2)$;
            \item[6.] $\hat{N}_i$ is an extremal graph for $\{K_2 (1, r_2), K_2(2,2), K_3(1,1,1)\}$ for $i \neq 1$,
    \end{enumerate}
        then $\hat{G}(n) = \bigvee_{i=1}^q\hat{N}_i$ is an extremal graph for $K_{q+1}(r_1, \ldots, r_{q+1})$.
\end{theorem}

In this paper, we generalize Theorem~\ref{thm: Erdos-Sim comp. multipartite} to cover several cases where $r_1$ may be greater than 3, as follows.
\begin{theorem}
\label{thm: turan numbers multipartite graphs}
    Let $r_2 \ge (r_1 - 1)! + 1$, $r_1 \le r_2 \le \ldots\le  r_{q+1}$ be given integers. 
    If $n$ is large enough, then each extremal graph $G(n)$ in $\EX(n, K_{q+1}(r_1, \ldots, r_{q+1}))$ is of the following form:
    \[G(n) = \bigvee_{i=1}^q N_i\] where 
    \begin{enumerate}
        \item $n_i = v(N_i) = \frac{n}{q} + o(n)$, that is, for any positive constant $\delta$ there exists a positive integer $N_{\delta}$ so that for all $n \ge N_{\delta}$ we have $\frac{n}{q} - \delta n \le n_i \le \frac{n}{q} + \delta n$;
        \item $N_1$ is an extremal graph for $K_2(r_1, r_2)$;
        \item $N_2, \ldots, N_q$ are extremal graphs for $K_2(1, r_2)$.
    \end{enumerate}    
        Conversely, if $\hat{N_1}, \ldots, \hat{N_q}$ are given graphs such that
    \begin{enumerate}
            \item[4.] there exists an extremal graph $\bigvee_{i=1}^q N_i$ satisfying 1., 2., 3. such that $v(\hat{N}_i) = v(N_i)$;
            \item[5.] $\hat{N}_i$ is an extremal graph for $K_2(r_1, r_2)$;
            \item[6.] $\hat{N}_i$ is an extremal graph for $\{K_2 (1, r_2), K_2(2,2), K_3(1,1,1)\}$ for $i \neq 1$,
    \end{enumerate}
        then $\hat{G}(n) = \bigvee_{i=1}^q\hat{N}_i$ is an extremal graph for $K_{q+1}(r_1, \ldots, r_{q+1})$.
\end{theorem}

The proof of this generalization duplicates the proof of Erd\H{o}s and Simonovits \cite{ErdosSimCompletemultipartite} and makes use of already known lower and upper bounds for $\ex(n, K_2(r_1, r_2))$ which allow us to show that $\ex(n, K_{2}(r_1, r_2)) = \Theta(n^{2 - \frac{1}{r_1}})$ when  $r_2 \ge (r_1-1)! + 1$, which allows us to observe that $\ex(n, K_2(r_1 - 1, r_2)) = o(\ex(n, K_2(r_1, r_2)))$ as long as $r_2$ is large enough. This helps us get over the main hurdle towards generalizing  Theorem~\ref{thm: Erdos-Sim comp. multipartite}. This main hurdle of showing $\ex(n, K_2(r_1 - 1, r_2)) = o(\ex(n, K_2(r_1, r_2)))$ was suggested in Remark 3 of the same paper.

Results of K\H{o}v\' ari-S\'os-Tur\'an give upper bounds $\ex(n, K_{s,t}) \le Cn^{2-\frac{1}{s}}$. 
\begin{lemma}[K\H{o}v\'ari-S\'os-Tur\'an Theorem. \cite{KovariSosTuran}]
    \label{lem: KST}
    Let $s \le t$ be positive integers. Then as $n \to \infty$,
    \[\ex(n, K_2(s, t)) \le \left(\frac{(t-1)^{\frac{1}{s}}}{2} + o(1)\right)n^{2-\frac{1}{s}}.\]
\end{lemma}

Constructions of $K_{s,t}$-free subgraphs of $K_n$ that have asymptotically (in the exponent) the same number of edges as the upper bound in Lemma~\ref{lem: KST} (consequently showing $\ex(n, K_{s,t}) \ge cn^{2-\frac{1}{s}})$ are known to exist only when $s,t \le 3$ or when $t \ge (s-1)! + 1$ (\cite{alon1999norm, brown1966, erdosrenyisos1966, kollar1996norm}).

The following result of Alon, R\'onyai, and Szab\'o from 1999 gives an asymptotically matching lower bound, as long as $t \ge (s - 1)! + 1$.

\begin{lemma}[Alon, R\'onyai, Szab\'o \cite{alon1999norm}]
\label{lem: Tightness for KST bound when t > s-1!}
For every fixed $s \ge 2$ and $t \ge (s-1)! + 1$ we have
\[\ex(n, K_{s, t}) \ge \frac{1}{2}n^{2- 1/s} - O(n^{2-1/t - c}),\]
where $c>0$ is an absolute constant.
\end{lemma}

\begin{remark}
\label{comparing turan numbers of complete bipartite graphs}
    Combining Lemma~\ref{lem: KST} and \ref{lem: Tightness for KST bound when t > s-1!} allows us to conclude that for $t \ge (s - 1)! + 1$, 
    \[\ex(n, K_{2}(s, t)) = \Theta(n^{2 - \frac{1}{s}}).\]
    Consequently, for $s\ge 4$, $t \ge (s-1)! + 1$ and for any positive integer $t'$, we have 
    \[\ex(n, K_2(s-1, t')) = o(\ex(n, K_2(s, t))).\]
    This means that given any constant $\alpha > 0$, there exists a positive integer $N_{\alpha}$ such that for all $n \ge N_{\alpha}$ we have $\ex(n, K_{s-1, t'}) < \alpha\ex(n, K_{s, t}))$.
\end{remark}

 We will also use the following result regarding forbidding complete bipartite graphs in complete bipartite hosts.
\begin{lemma}
\label{lem: Zarankiewicz with large partite sets}
Let $\delta$ be an arbitrary constant in  $(0,1)$ and $s \le  t$ be positive integers. Let $A$ and $B$ be the partite sets of $K_2(\lceil\delta n\rceil, \lfloor(1-\delta)n\rfloor)$, such that $|A| = \lceil\delta n\rceil$ and $|B| = \lfloor(1-\delta)n\rfloor)$. 
For $n$ large enough, any subgraph of $K_2(\lceil\delta n\rceil, \lfloor(1-\delta)n\rfloor)$ that is $K_{s, t}$-free, where the partite sets of size $s$ and $t$ are contained in $A$ and $B$, respectively, has at most $(t-1)^{\frac{1}{s}}n^{2-\frac{1}{s}}$ edges. 
\end{lemma}

\begin{proof}

Assume to the contrary that $G$ is a subgraph of $K_2(\lceil\delta n\rceil, \lfloor(1-\delta)n\rfloor)$ with $e(G) > (t-1)^{\frac{1}{s}}n^{2-\frac{1}{s}}$ and contains no copy of $K_{s, t}$ where the partite sets of size $s$ and $t$ are contained in $A$ and $B$, respectively.
    We will obtain a contradiction using a standard double counting argument. Let $P$ be the number of ordered pairs $(v, S)$ where $v \in B$ and $S \subset A$ is a set of $s$ neighbors of $v$. Then by Jensen's inequality and the convexity of $f(x) = \frac{x(x-1)\ldots(x-s+1)}{s!}$ we get that for $n$ sufficiently large
    \[P = \sum_{v \in B}\binom{d(v)}{s} \ge |B|\binom{\frac{e(G)}{|B|}}{s} \ge \frac{(e(G))^s}{|B|^{s-1}s!(1+o(1))}.\]
    
    Next, since $G$ has no $K_{s, t}$ with partite sets of size $s$ and $t$ in $A$ and $B$, respectively, we have
    \[P \le \binom{|A|}{s}(t-1) \le \frac{|A|^s}{s!}(t-1).\]
    Thus, $\frac{(e(G))^s}{|B|^{s-1}s!(1+o(1))} \le \frac{|A|^s}{s!}(t-1)$ and therefore 
    \[e(G) \le |A||B|^{1-\frac{1}{s}}(t-1)^{\frac{1}{s}}(1+o(1))\le \delta(1-\delta)^{1-\frac{1}{s}}(t-1)^{\frac{1}{s}}(1+o(1))n^{2-\frac{1}{s}}.\]
    Since $\delta, 1-\delta < 1$, we have a contradiction to the assumption that $e(G) > (t-1)^{\frac{1}{s}}n^{2-\frac{1}{s}}$ for $n$ sufficiently large.
\end{proof}

 We will now work toward proving Theorem~\ref{thm: turan numbers multipartite graphs}. We will require the following lemmas from \cite{ErdosSimCompletemultipartite} that are used to prove Theorem~\ref{thm: Erdos-Sim comp. multipartite}.

 \begin{lemma}(Lemma 1 of \cite{ErdosSimCompletemultipartite})
     \label{lem: product graphs not contain forbidden graph}
    Given positive integers $r_1 \le r_2 \le \ldots \le r_{q+1}$. Let $G_1$ be a graph not containing $K_2(r_1, r_2)$ and let $G_i$ $(i=2, \ldots, q)$ be graphs not containing $K_2(1, r_2),$ $ K_2(2,2), $ and $K_3(1, 1, 1)$. Then $\bigvee_{i = 1}^q G_i$ does not contain $K_{q+1}(r_1, \ldots, r_{q+1})$.
 \end{lemma}

\begin{lemma}(Generalizing Lemma 2 of \cite{ErdosSimCompletemultipartite})
    \label{lem: extremal graphs do not have large internal degree}
  Let $r_2 \ge r_1$ be two positive integers and $\delta$ be a positive constant. Let $G(n)$ be graphs on $n$ vertices which do not contain $K_2(r_1, r_2)$.
  \begin{enumerate}
      \item If either $r_1 = 1$, then for $n$ sufficiently large, every vertex $v$ of $G(n)$ has degree $d(v) \le r_2 - 1 \le \delta n$.
      \item If $r_1 \ge 2$,
      then there exists a positive constant $c_{\delta, r_1, r_2}$ depending only on $\delta, r_1$ and $r_2$ such that if $n$ is sufficienlty large and $G(n)$ has a vertex $v$ with degree $d(v) \ge \delta n$, then 
      \begin{equation*}
      \begin{aligned}
      e(G(n)) &\le \ex(n, K_2(r_1, r_2)) + O(n^{2-\frac{1}{r_1 - 1}}) - \ex(\lceil\delta n\rceil, K_2(r_1, r_2))\\
      &\le (1 - c_{\delta,r_1,  r_2})\ex(n, K_2(r_1, r_2)).
      \end{aligned}
      \end{equation*}
  \end{enumerate}
  
\end{lemma}
\begin{proof}
The first part is straightforward. Assume $r_1 \ge 2$.
    We begin by observing that if $n_1, n_2$ are positive integers such that $n_1 + n_2 = n$, then $\ex(n, \mathcal{F}) \ge \ex(n_1, \mathcal{F}) + \ex(n_2, \mathcal{F})$, for any family of forbidden graphs $\mathcal{F}$. Thus, in particular, $\ex(n, K_{2}(r_1, r_2)) \ge \ex(\lceil\delta n\rceil, K_2(r_1, r_2)) + \ex(\lfloor(1-\delta)n\rfloor, K_{2}(r_1, r_2))$. 
    Next, assume $G(n)$ is a graph not containing $K_2(r_1, r_2)$, but containing a vertex $v$ with degree at least $\delta n$. Let $C$ be a collection of $\lceil\delta n \rceil$ neighbors of $v$. Then the graph induced by the vertices of $C$ does not contain any $K_2(r_1 - 1, r_2)$ since otherwise, $G(n)$ contains $K_2(r_1, r_2)$. 
    Thus, from the result of K\H{o}v\' ari-S\' os-Tur\' an (Lemma~\ref{lem: KST})) we can conclude that for $n$ sufficiently large there are at most $(r_2-1)^{\frac{1}{r_1 - 1}}(\delta n)^{2-\frac{1}{r_1 - 1}}$  (and in fact $\ex(\lceil \delta n\rceil, K_2(r_1 - 1, r_2))$) edges with both endpoints in $C$. Moreover, the bipartite graph induced by the set of all edges with one endpoint in $C$ and the other in $V(G(n)) - $C$ - \{v\}$ does not contain a $K_2(r_1 - 1, r_2)$ where the part with $r_2$ vertices lies in $C$. This is again because $G(n)$ does not contain any $K_2(r_1, r_2)$. Then, by Lemma~\ref{lem: Zarankiewicz with large partite sets}, we can show that the number of such edges is at most $(r_2-1)^{\frac{1}{r_1 - 1}} n^{2-\frac{1}{r_1 - 1}}$. 
    Hence, $e(G(n)) \le 2(r_2-1)^{\frac{1}{r_1 - 1}}n^{2-\frac{1}{r_1 - 1}} + \ex(\lfloor(1 -\delta)n\rfloor, K_2(r_1, r_2)) \le \ex(n, K_2(r_1, r_2))) + 2(r_2-1)^{\frac{1}{r_1 - 1}}n^{2-\frac{1}{r_1 - 1}} - \ex(\lceil \delta n\rceil, K_2(r_1, r_2))$.  Consequently, it follows from Remark~\ref{comparing turan numbers of complete bipartite graphs} that for $n$ sufficiently large the middle term is dominated by the remaining terms, and there exists a constant $c_{\delta, r_1, r_2}$ depending only on $\delta, r_1$ and $r_2$, such that $e(G(n))\le (1 - c_{\delta,r_1,  r_2})\ex(n, K_2(r_1, r_2))$. 
\end{proof}

\begin{remark}
    According to Theorem~\ref{thm: turan numbers multipartite graphs} any extremal graph $G(n)$ in $\EX(n, K_{q+1}(r_1, \ldots, r_{q+1}))$ is a join of $q$ graphs $N_1, N_2, \ldots, N_q$:
    \[G(n) = \bigvee_{i=1}^q N_i.\]
     Therefore, its complement $\overline{G(n)}$ is not connected. Note however, later in this paper, we would like to use $K_{\lceil\frac{s+1}{d+1}\rceil}(a, d+1,\ldots,d+1)$-free graphs that have connected complements and still have close to the Tur\'an number of edges $\ex(n, K_{\lceil\frac{s+1}{d+1}\rceil}(a, d+1,\ldots,d+1)) = e(G(n))$. Their compliments will then produce strong upper bounds for the connected edge minimization problem for a given $d$-independence number $i_d = s$. 
\end{remark}

We will prove that when $r_1$ is relatively small compared to $r_2$, then among all $n$-vertex $K_{q+1}(r_1, \ldots, r_{q+1})$-free graphs that have connected complements, the graphs with the most number of edges are obtained by deleting $q-1$ edges from some graph $G(n) = \bigvee_{i=1}^q N_i$ in $\EX(n, K_{q+1}(r_1, \ldots, r_{q+1}))$.

\begin{theorem}
\label{thm: connected complement close to turan numbers multipartite graphs}
    Let $r_2 \ge (r_1 - 1)! + 1$, $r_1 \le r_2 \le \ldots\leq r_{q+1}$ be given integers. 
    If $n$ is large enough, then
    any graph $G'(n) \in \EX_{cc}(n, K_{q+1}(r_1, \ldots, r_{q+1}))$ is obtained by deleting 
    $q-1$  edges from some graph $G(n)$ in $\EX(n, K_{q+1}(r_1, \ldots, r_{q+1}))$ to make the complement $\overline{G'(n)}$ connected.   
        Conversely, if $G'(n)$ is any graph on $n$ vertices obtained from any graph $G(n) \in \EX(n, K_{q+1}(r_1, \ldots, r_{q+1}))$ by deleting $q-1$  edges of $G(n)$ such that $\overline{G'(n)}$ is connected, 
        then $G'(n) \in \EX_{cc}(n, K_{q+1}(r_1, \ldots, r_{q+1}))$.
\end{theorem}

To prove Theorem~\ref{thm: connected complement close to turan numbers multipartite graphs} we will require the following lemmas and theorems given in this section which help us recover some of the features of extremal graphs expressed in Remark~\ref{rem: ErdosSim on general structure of extremal graphs}.2..

We use the following results of Erd\H{o}s from \cite{erdos1964extremal} and \cite{erdos1967some}.

\begin{lemma}\cite{erdos1964extremal}
\label{lem: large common neighborhoods when large degrees}
    Let $S$ be a set of $N$ elements $y_1, \ldots, y_N$, and let $A_i$, $1 \le i \le m$ be subsets of $S$. Suppose 
    \begin{equation}
        \sum_{i=1}^{m} |A_i| \ge \frac{mN}{w}
    \end{equation}
  for some constant $w$.
    Then, if $m \ge 2l^2 w^l$, there are $l$ distinct $A$'s, $A_{i_1}, \ldots, A_{i_l}$, so that 
    \begin{equation}
        \left|\bigcap_{j=1}^l A_{i_j}\right| \ge \frac{N}{2w^l}.
    \end{equation}
\end{lemma}

\begin{lemma}\cite{erdos1967some}
\label{original lem: large r-partite graphs in asymptotiacally extremal graphs}
    Let $G(n)$ be a graph with $(1 + o(1))\frac{n^2}{2}\left(1 - \frac{1}{r-1}\right)$ edges that does not contain some $K_r(t, t, \ldots, t)$ as a subgraph. Then there exists some $K_{r-1}(p_1, \ldots, p_{r-1})$ which differs from $G(n)$ by $o(n^2)$ edges such that $\sum_{i=1}^{r-1} p_i = n$ and $p_i = (1 + o(1))\frac{n}{r-1}$ for all $i\in [r-1]$. 
\end{lemma}
We will be using the following version of Lemma~\ref{original lem: large r-partite graphs in asymptotiacally extremal graphs}.

\begin{corollary}
\label{lem: large r-partite graphs in asymptotiacally extremal graphs}
    Given positive constants $\delta$ and $\xi$ there exists a natural number $N_{\delta, \xi}$ such that if $n \ge N_{\delta, \xi}$ then any $K_r(t, \ldots, t)$-free graph $G(n)$ on $n$ vertices and at least $\frac{n^2}{2}(1 - \frac{1}{r-1}) - n^{1.9}$ edges must differ from some complete multipartite graph $K_{r-1}(p_1, \ldots, p_{r-1})$ by less than $\xi n^2$ edges, where $\frac{1 - \delta}{r-1}n \le p_i \le \frac{1 + \delta}{r-1}n$ for all $i \in [r-1]$.
\end{corollary}

We will prove the following generalization of a result of Erd\H{o}s (Theorem 3 of \cite{erdos1966new}). 
\begin{theorem}
\label{thm: generalization of a structural result for extremal graphs to almost extremal graphs}

Let $F$ be a fixed graph with $\chi(F) = r$. For any positive constants $\delta$ and $\epsilon$ there exist natural numbers $N_{\delta, \epsilon}$ and $c_{\delta, \epsilon}$ such that if $n \ge N_{\delta, \epsilon}$ then the following is true:
   Any $F$-free graph $G'(n)$ on $n$ vertices 
        and at least $\ex(n, F) - \delta n$ edges can be partitioned into $r-1$ sets $V_1, \ldots, V_{r-1}$ where $\frac{(1-\delta)}{r-1}n \le |V_i| \le \frac{(1+\delta)}{r-1}n$ and
        there are at least $|V_i| - c_{\delta, \epsilon}$ vertices in any $V_i$ which are adjacent to more than $|V_i^c| - \epsilon n$ vertices of $V_i^c$.
\end{theorem}

First, we prove the generalization of a lemma from \cite{erdos1966new} used to prove Theorem 3 of the same paper.

\begin{lemma}
\label{lem: stability complete multipartite min degree}   
For any arbitrary positive constant $\mu$
    there exists a natural number $N_{\mu}$ such that for all $n \ge N_{\mu}$, if 
    $G'(n)$ is an $F$-free graph on $n$ vertices with
    $e(G'(n)) \ge \ex(n, F) - \mu n$, then the minimum degree of $G'(n)$ is at least $
    n\left(1 - \frac{1}{r-1} - 4\mu\right)$, where $\chi(F) = r$. 
\end{lemma}

\begin{proof}
It suffices to assume that $\mu \le \frac{1}{4}\left(1 - \frac{1}{r-1}\right)$, otherwise the lower bound on the minimum degree is negative. Let $F$ be a graph with $\chi(F) = r$. Assume to the contrary that there exists some positive constant $\mu
$ and graphs $G'(n)$ on $n$ vertices with $e(G'(n)) \ge \ex(n, F) - \mu n$ for infinitely many values of $n$, where $G'(n)$ does not contain $F$,  and $G'(n)$ has a vertex $y$ with degree less than 
$n\left(1 - \frac{1}{r-1} - 4\mu\right)$. Then by Remark~\ref{rem: ErdosSim on general structure of extremal graphs}.1.\ we have $e(G'(n)) \ge \ex(n, F) - \mu n \ge \frac{n^2}{2}(1 - \frac{1}{r-1}) - n^{1.9}$ for $n$ sufficiently large.
Next, since $G'(n)$ is $F$-free, $G'(n)$ must also be $K_{r}(t, \ldots, t)$-free where $t$ is the size of a largest color class of $F$.
Given any positive constant $\delta$, it follows then from Corollary~\ref{lem: large r-partite graphs in asymptotiacally extremal graphs} that for $n$ sufficiently large, $G'(n)$ differs from some $K_{r-1}(p_1, \ldots, p_{r-1})$ by fewer that $\delta n^2$ edges, 
where $\frac{1 - \delta}{r-1}n \le p_i \le \frac{1 + \delta}{r-1}n$ for all $i \in [r-1]$. For all $i \in [r-1]$, let $V_i$ be the set of vertices of $G'(n)$ associated to the partite set of $K_{r-1}(p_1, \ldots, p_{r-1})$ on $p_i$ vertices.
Let $k \ge v(F)$ be some fixed natural number.
Let $a:=\frac{2\delta n k}{\mu}$ and $b:=\frac{\mu n}{k}$.
Let $A \subset V_1$ be the collection of vertices with more than $b$ non-neighbors among the remaining $n - p_1$ vertices of $G'(n)$, that is $V_1^c$. Then $|A| < a$ since $|A|b < \delta n^2 <2\delta n^2 = ab$.  
Next, if we set $\delta < 
\frac{\mu}{4k(r-1)}$, 
then because $p_1 \ge \frac{1-\delta}{r-1}n$, we can find at least $k$ vertices $x_1, \ldots, x_k$ in $V_1 \setminus A$. Further, using $p_1 \le \frac{1+\delta}{r-1}n$, we show that their common neighborhood $N(x_1, \ldots, x_k):=\cap_{i=1}^k N(x_i)$ has size at least $n\left(1-\frac{1}{r-1} - 2\mu\right)$.

To see this, observe that \begin{equation}
    \begin{aligned}
        |N(x_1, \ldots, x_k)| &\ge \sum_{i=1}^k|N_{V_{
        1}^c}(x_i)| - (k-1)|V_1^c|\\
        &\ge k(n-p_1-b) - (k-1)(n-p_1)
        \ge n-p_1-kb \\ 
        &\ge n-\left(\frac{1+\delta}{r-1}\right)n - k\frac{\mu n}{k} \\
        &\ge n\left(1-\frac{1}{r-1} - 2\mu\right).
    \end{aligned}
\end{equation} 

Next to prove the lower bound on the minimum degree of $G'(n)$, we will modify the graph as follows. Delete all the edges adjacent to $y$ from $G'(n)$ and add edges between $y$ and every vertex of $N(x_1, \ldots, x_k) \setminus \{y\}$. 
The new graph we obtain has more than $\ex(n, F)$ edges since we have increased the edge count by more than $\mu n$.  However, this resultant graph must also be $F$-free since otherwise $G'(n)$ itself would have contained $F$. 
To see this, suppose to the contrary that a copy of $F$ is created after the modification. Then $y$ must be a vertex of it. 
However, each of the $k$  vertices, $x_1, \ldots, x_k$ in $G'(n)$ are adjacent to all the neighbors of $y$ in the new graph, and since $k \ge v(F)$, at least one of $x_i$ is not a part of the $F$ in the new graph, and could be used in place of $y$ to obtain a copy of $F$ in $G'(n)$. 
Since, $G'(n)$ was assumed to be $F$-free, we must have that the new graph is also $F$-free. 
This contradicts the fact that the number of edges in the new graph is more than $\ex(n, F)$. Hence, there does not exist any vertex $y$ whose degree is less than $n\left(1 - \frac{1}{r-1} - 4\mu\right)$.   
\end{proof}

We are now ready to prove Theorem~\ref{thm: generalization of a structural result for extremal graphs to almost extremal graphs}. 

\begin{proof}[Proof of Theorem~\ref{thm: generalization of a structural result for extremal graphs to almost extremal graphs}]
Assume, to the contrary, that for any natural number $k$, there exist positive $\delta$ and $\epsilon$ such that there are infinitely many values of $n$ for which there exists an $F$-free graph $G’(n)$ on $n$ vertices with at least $\ex(n, F) - \delta n$ edges. Furthermore, all possible partitions of the vertices of $G’(n)$ into $r-1$ partite sets $V_1, \ldots, V_{r-1}$ must have a partite set $V_i$ for some $i \in [r-1]$ with at least $k+1$ vertices such that each vertex in $V_i$ is not adjacent to more than $\epsilon n$ vertices of $V_i^c$.

Let $\cup_{j=1}^{r-1}P_j(n)$ be a partition of the vertices of $G'(n)$ which has the minimum possible number of edges where both vertices of any edge lie in the same part. For simplicity, we will suppress $n$ and use $P_j$ instead of $P_j(n)$ when the order of the graph is clear from context. Then for any fixed constant $k$, there exists $n$ sufficiently large such that for some $1 \le i \le r-1$ there exist at least $k$ vertices $x_1, \ldots, x_k$ in $P_i(n)$ such that for all $p \in [k]$, $x_p$ is not adjacent to at least $\epsilon n$ vertices that lie in $\cup_{j \ne i}P_j$. Now say $t$ is the size of the largest color class of $F$.

By an $(r-1)$-fold application of Lemma~\ref{lem: large common neighborhoods when large degrees} (with $N=n$, $y_1, \ldots, y_N$ representing the $n$ vertices of $G'(n)$, $A_i$ representing the neighborhood of vertex $x_i$ in $P_j$, $m = k = 2t^2\epsilon^t$, $w = \frac{1}{\epsilon}$ and $l=t$)
we can find $t$ vertices, 
$x_1, \ldots, x_t$ in $P_i$ along with $s \ge \frac{\epsilon^t}{2}n$ vertices from each of the partite sets $P_j$, $y_1^j, \ldots, y_s^j$,  $j=1, \ldots, r-1$, such that $x_i$ is adjacent to $y_l^j$ for all $1 \le i \le t$, $1\le l \le s$ and $1 \le j \le r-1$. 
It follows from Corollary~\ref{lem: large r-partite graphs in asymptotiacally extremal graphs} that for any positive constant $\xi$, if $n$ is sufficiently large then all but at most $\xi n^2$ edges of the form $y_{l_1}^{j_1} y_{l_2}^{j_2}$ occur in $G'(n)$ for any $1 \le l_1, l_2 \le s$ and $1 \le j_1, j_2 \le r-1$, with $j_1 \ne j_2$.
Let $Y = \{y_l^j \mid 1 \le l \le s, 1 \le j \le r-1\}$ be the collection of all $y$'s obtained above. Let $e(Y)$ denote the number of edges of $G'(n)[Y]$ (that is, the subgraph induced by $Y$ in $G'(n)$). Then there exist small positive constants $\epsilon'$ and $\xi$ with respect to which we may choose $n$ sufficiently large so that 
\begin{equation}
    \begin{aligned}
 e(Y) 
 &\ge \binom{r-1}{2}s^2 - \xi n^2\\
 &\ge \left(1-\frac{1}{r-2} + \epsilon'\right)\frac{s^2}{2} .       
    \end{aligned}
\end{equation}

Thus, by Remark~\ref{rem: ErdosSim on general structure of extremal graphs}.1.\ we have that there must exist a $K_{r-2}(t, \ldots, t)$ contained in the graph induced by the $y$'s and consequently there exists a $K_{r-1}(t, \ldots, t)$ contained in the graph induced by the $x$'s and $y's$ and therefore in  $G'(n)$, a contradiction. 
\end{proof}
The following corollary follows from the above result. 
\begin{corollary}
\label{cor: stability complete multipartite few large internal degree}
 
For any positive constants $\delta$ and $\epsilon$ there exist natural numbers $N_{\delta, \epsilon}$ and $c'_{\delta, \epsilon}$ such that if $n \ge N_{\delta, \epsilon}$
    then 
   any $F$-free graph $G'(n)$ on $n$ vertices 
        and at least $\ex(n, F) - \delta n$ edges can be partitioned into $r-1$ sets $V_1, \ldots, V_{r-1}$ where $\frac{(1-\delta)}{r-1}n \le |V_i| \le \frac{(1+\delta)}{r-1}n$ and
        there are at most $c_{\delta, \epsilon}'$ vertices in any $V_i$ which are adjacent to more than $\epsilon n$ vertices of $V_i$.
\end{corollary}

\subsection{Proof of Theorems~\ref{thm: turan numbers multipartite graphs} and \ref{thm: connected complement close to turan numbers multipartite graphs}}
\label{subsec: proof of Turan numbers for multipartite graphs}

For any non-negative integer $m$, let $\EX(n, K_{q+1}(r_1, \ldots, r_{q+1}), m)$  denote the collection of all $n$-vertex $K_{q+1}(r_1, \ldots, r_{q+1})$-free graphs with at least $\ex(n, K_{q+1}(r_1, \ldots, r_{q+1}))-   m$  edges. We informally refer to this collection of graphs as being ``almost" extremal if $m$ is either a constant or $m \le q \epsilon n$, where $\epsilon$ is some given arbitrarily small positive constant.

Note that under this new notation $\EX(n, K_{q+1}(r_1, \ldots, r_{q+1}), 0)= \EX(n, K_{q+1}(r_1, \ldots,$ $r_{q+1})
).$ In the following two remarks, we will summarize some of the properties of graphs in $\EX(n, K_{q+1}$ $(r_1, \ldots, r_{q+1}))$ and more generally $\EX(n, K_{q+1}(r_1, \ldots, r_{q+1}), m)$,
that we will be using to prove some of our main results.

We introduce a definition which we will be referring to in the coming sections.
\begin{definition}
For a natural number $k \ge 2$ and any graph $G$, we call a partition $\mathcal{P}:=\{V_1, \ldots V_k\}$ of the vertices of $G$ a \textit{$k$-partition of $G$} if $V(G) = \cup_{i=1}^k V_i$. For a given $k$-partition $\Pa$ of $G$ we say an  edge $e$ of $G$ is an \textit{internal edge of the partition} if it is contained entirely inside one of the parts, that is $e \in V_j$ for some $j \in [k]$, and we say $e$ is an \textit{external edge of the partition} if it is not an internal edge, that is the two end points of $e$ belong to two different parts $V_{j_1}$ and $V_{j_2}$ of $G$.   
We say a $k$-partition of $G$ is a \textit{$k$-good partition} of $G$ if it has the minimum number of internal edges among all $k$-partitions of $G$. Note that this is equivalent to saying that a $k$-partition of $G$ is a $k$-good partition of $G$ if it has the maximum number of external edges among all $k$-partitions of $G$. Further, if $\Pa = \{V_1, \ldots V_k\}$ is a $k$-good partition of $G$, and $v$ is an arbitrary vertex of $G$, if $v \in V_i$ for some $i \in [k]$, then $d_{V_i}(v) \le d_{V_j}(v)$ for all $j \in [k]$.  
\end{definition}

\begin{remark}
\label{rem: structural properties of extremal and stability graphs complete multipartie}    

Using Remark~\ref{rem: ErdosSim on general structure of extremal graphs}.2., observe that given any real and arbitrarily small $\epsilon > 0$, we can find an order $N_{\epsilon}$ such that for all $n \ge N_{\epsilon}$, any graph $K(n) \in \EX(n, K_{q+1}(r_1, \ldots,$ $ r_{q+1}))$ has the following structural properties:

If we take any $q$-good partition $\Pa = \{U_1, \ldots, U_q\}$ of $K(n)$ and define graphs $N_1, N_2, \ldots, N_q$ with respect to this partition, such that $N_i = G[U_i]$ for all $i \in [q]$, then the following hold.
\begin{enumerate}
    \item[$\alpha$)] \label{rem: structural properties of extremal and stability graphs complete multipartie alpha}  Let $n_i := |U_i|$ for all $i \in [q]$. Then $\frac{n}{q} - \epsilon n \le n_i \le \frac{n}{q} + \epsilon n$ for all $i = 1, \ldots, q$.
    \item[$\beta$)] \label{rem: structural properties of extremal and stability graphs complete multipartie beta} The minimum degree of $K(n)$ is at least $\frac{n}{q}(q-1) - \epsilon n$.
    \item[$\gamma$)] \label{rem: structural properties of extremal and stability graphs complete multipartie gamma} Let $U_i'$ denote the vertices of $U_i$ that are adjacent to less than $\epsilon n$ other vertices of $U_i$. Then there exists a constant $k_{\epsilon}$ that depends only on $\epsilon$ such that $|U_i \setminus U_i'| \le k_{\epsilon}$. Consequently, every vertex $v \in U_i'$ is adjacent to at most $\epsilon n$ vertices of $U_i$, and therefore $v$ is adjacent to at least $|U_j| -  3\epsilon n$ vertices of $U_j$ for each $j \neq i$ and thereby adjacent to at least $|U_j'| -  3\epsilon n$ vertices of $U_j'$ for each $j \neq i$. 
\end{enumerate}
\end{remark}

\begin{remark}
\label{rem: structure of stability extremal graphs}
Using Corollaries~\ref{lem: large r-partite graphs in asymptotiacally extremal graphs}, \ref{cor: stability complete multipartite few large internal degree}, Lemma~\ref{lem: stability complete multipartite min degree}, and Theorem~\ref{thm: generalization of a structural result for extremal graphs to almost extremal graphs} we can show that given any real and arbitrarily small $\epsilon > 0$, we can choose $\delta = q \epsilon$ for some natural number $q$ and find an order $N_{\epsilon}$ such that for all $n \ge N_{\epsilon}$,
any graph $K(n)$ in $\EX(n, K_{q+1}(r_1, \ldots, r_{q+1}), m)$ also satisfies the structural properties outlined in Remark~\ref{rem: structural properties of extremal and stability graphs complete multipartie} for extremal graphs, for any $m \le q\epsilon n$.
In particular this is also true if $K(n) \in \EX(n, K_{q+1}(r_1, \ldots, r_{q+1}), q-1)$  or $K(n) \in \EX(n, K_{q+1}(r_1, \ldots, r_{q+1}), q)$.
\end{remark}

Now, for any $K(n) \in \EX(n, K_{q+1}(r_1, \ldots, r_{q+1}), m)$ let $E = \sum_{1 \le i,j \le q}n_i n_j$  denote the number of pairs of vertices in $K(n)$ where the vertices come from different parts of a $q$-good partition.
Then it follows from Lemma~\ref{lem: product graphs not contain forbidden graph} that 
\begin{equation}
\label{eqn: lower bound on e(K(n))}
    e(K(n)) \ge E + \ex(n_1, K_{2}(r_1, r_2)) + \sum_{i=2}^q\ex(n_i, K_2(1, r_2)) - m.
\end{equation} 
To see this, note that when $G(n_1)$ is an extremal graph for $K_2(r_1, r_2)$, and $G(n_i)$ are extremal graphs for $\{K_{2}(1, r_2), K_2(2,2), K_3(1, 1, 1)\}$ where $i = 2, \ldots, q$, then $G(n) = \bigvee_{i=1}^q G(n_i)$ does not contain any $K_{q+1}(r_1, \ldots, r_{q+1})$. 
Thus $e(K(n)) \ge e(G(n)) - m$. 
Moreover, it can be checked that $\EX(n_i, \{K_{2}(1, r_2), K_2(2,2), K_3(1, 1, 1)\}) \subseteq \EX(n_i, K_2(1, r_2))$. 
The extremal graphs in $\EX(n_i, K_2(1, r_2))$ are $(r_2 -1)$-regular if either $n_i$ or $r_2 - 1$ are even, otherwise the extremal graphs are almost-$(r_2 - 1)$-regular graphs. Here, an almost-$k$-regular graph is a graph where all vertices have degree $k$, except for one vertex of degree $k-1$. 
It can then be verified that there always exists a $K_2(1, r_2-1)$ extremal graph that does not contain any $C_3$ or $C_4$. Thus, (\ref{eqn: lower bound on e(K(n))}) holds. 

In terms of $m$ we will be using Equation~(\ref{eqn: lower bound on e(K(n))}) particularly when $m = 0, q-1, q$ or for any $m \le q\epsilon n$.

Throughout the rest of this section, we will be assuming that given an arbitrarily small positive constant $\epsilon$, the graph $K(n) \in \EX(n, K_{q+1}(r_1, \ldots, r_{q+1}), m)$ for $m \le q \epsilon n$ and $n$ is taken to be sufficiently large to make our arguments go through and apply Remarks~\ref{rem: structural properties of extremal and stability graphs complete multipartie} and~\ref{rem: structure of stability extremal graphs}. 
We will also use $N_i$ for $i = 1, \ldots, q$ to denote the same induced graphs that appear in the two remarks.
To make the statements and proofs of subsequent lemmas cleaner we will not be repeating this assumption for the graphs $K(n)$. 

We now state a lemma which will be very useful in determining the finer structure of extremal (and ``almost" extremal) graphs $K(n)$. In what follows, we will use $U_i$ and $U_i'$ as defined above.
 
\begin{lemma}
\label{lem: size of common neighborhood}
    Let $1 \le j \le q$ be fixed. Let $t > 0$ be any fixed integer.  For $n$ sufficiently large if $u_1, \ldots, u_s$ are $s$ vertices of $K(n)$ contained in $\bigcup_{i \ne j }U_i'$, then the common neighborhood $N(u_1, \ldots, u_s)$ contained in $U_j'$ (and by a similar argument  the common neighborhood in $U_j$) has size at least $t$, that is, 
 \begin{equation}
        |N(u_1, \ldots, u_s) \cap U_j'| \ge t.
    \end{equation}

\end{lemma}
    \begin{proof}
    Given any fixed integer $t > 0$, we can choose a small enough $\epsilon > 0$ and $n$ sufficiently large so that by Remarks~\ref{rem: structural properties of extremal and stability graphs complete multipartie} and \ref{rem: structure of stability extremal graphs}
    \begin{equation}
        |N(u_1, \ldots, u_s) \cap U_j'| \ge \left(\sum_{i=1}^s |N_{U_j'}(u_i)|\right) - (s-1)|U_j'| \ge \left(\sum_{i=1}^s (|U_j'| -3\epsilon n)\right) - (s-1)|U_j'|
    \end{equation}
    which is at least $t$.    
    \end{proof}

\begin{lemma}
\label{lem: Forbidden complete bipaprtite graph in U_i'}
For $i\in [q]$, the graphs induced by $U_i'$ do not contain $K_2(r_1, r_2)$, whenever $n$ is sufficiently large.    
\end{lemma}

\begin{proof}
Assume to the contrary that there exists some $i$ such that the edges in $U_i'$ induce a $K_2(r_1, r_2)$. We can assume without loss of generality that $i = 1$. Say $u_1, \ldots, u_{r_1 + r_2}$ are the vertices of this $K_2(r_1, r_2)$. Then by Lemma~\ref{lem: size of common neighborhood} the size of their common neighborhood lying in $U_2'$ is at least
\[\left(\sum_{l = 1}^{r_1 + r_2} (|U_2'| - 3\epsilon n)\right) - (r_1 + r_2 - 1)|U_2'| \ge |U_2'| -3(r_1 + r_2)\epsilon n>  r_{q+1}\] 
as long as $\epsilon$ is sufficiently small and $n$ is sufficiently large. So we can find at least $r_3$ vertices in $U_2'$ that are adjacent to all the vertices of the $K_2(r_1, r_2)$  contained in $U_1'$. Hence the graph induced by $U_1' \cup U_2'$ contains a $K_3(r_1, r_2, r_3)$ . Similarly, we can find $r_4$ vertices in $U_3'$ that are all adjacent to the $K_3(r_1, r_2, r_3)$ already found in $K(n)\left[U_1' \cup U_2'\right]$. Proceeding similarly we find a $K_{q+1}(r_1, \ldots, r_{q+1})$  in $K(n)\left[\cup_{i=1}^q U_i'\right]$, a contradiction.   
\end{proof}

We have shown that the graphs induced by the $U_i'$ for all $i = 1, \ldots, q$ do not contain any $K_2(r_1, r_2)$.

Now, if the graph induced by every $U_i'$ for $i = 1 \ldots, q$ is $K_2(r_1-1, r_2)$-free, then $\sum_{i=1}^q e(U_i') \le \ex(n, K_2(r_1-1, r_2))$. Since $|U_i \setminus U_i'| \le k_{\epsilon}$ for each $i \in [q]$, we have that 
\begin{equation}
\label{eqn: upperbound if each U_i' is Kr_1-1,r_2 free}
    e(K(n)) \le E + O(n) + O(\ex(n, K_2(r_1-1, r_2))) = E + O(\ex(n, K_2(r_1-1, r_2))). 
\end{equation}

Then, using Remark~\ref{comparing turan numbers of complete bipartite graphs} to compare, we can observe that (\ref{eqn: upperbound if each U_i' is Kr_1-1,r_2 free}) contradicts (\ref{eqn: lower bound on e(K(n))}) when $m \le q \epsilon n$.
Hence, at least one of the $U_i'$ must induce a graph containing a $K_2(r_1 -1, r_2)$.

Now, without loss of generality assume that $K(n)[U_1']$ contains a $K_2(r_1 - 1, r_2)$. 
For all $j \in [q], j \ne 1$, let $B_j$ denote the subset of vertices of $U_j'$ that are in the common neighborhood of every vertex of this $K_2(r_1 - 1, r_2)$.

\begin{lemma}
    The maximum degree of the graph induced by the vertices of $B_j$ is at most $r_3 -1$ for all $j = 2, \ldots, q$. 
\end{lemma}
\begin{proof}
    Towards a contradiction, assume without loss of generality that $B_2$ has a vertex $u$ with $r_3$ neighbors $v_1, \ldots, v_{r_3}$ in $B_2$. Then this $K_2(1, r_3)$ along with the $K_2(r_1 - 1, r_2)$ contained in $K(n)[U_1']$ forms a $K_3(r_1, r_2, r_3)$ in $K(n)[U_1' \cup U_2']$. Using a recursive application of Lemma~\ref{lem: size of common neighborhood} we can find $r_4, \ldots, r_{q+1}$ vertices in $U_3', \ldots, U_q'$, respectively, that together with the $K_3(r_1, r_2, r_3)$ form a $K_{q+1}(r_1, \ldots, r_{q+1})$ in $K(n)$, a contradiction.
\end{proof}
Therefore, $e(B_j) \le \frac{(r_3 -1)}{2}\left(\frac{n}{q} + \epsilon n\right) = O(n)$ for all $j=2, \ldots q$. Moreover, let $m_j = |U_j' \setminus B_j| \le (r_1 +r_2 -1)3\epsilon n$ as a consequence of Remark~\ref{rem: structural properties of extremal and stability graphs complete multipartie gamma}($\gamma$). Then the number of edges with both end points in $U_j' \setminus B_j$ is at most $\ex(m_j, K_2(r_1, r_2))$. 
Further, if $m_j \ne 0$, we can partition the vertices of $U'_j$ into at most $\frac{n+qn\epsilon}{qm_j}$ sets of size at most $m_j$, so that between any such set and the $m_j$ vertices of $U_j' \setminus B_j$ there are no more that $\ex(2m_j, K_2(r_1, r_2))$ edges, and consequently, 
$e(U_j', U_j' \setminus B_j) = O(\frac{n+q\epsilon n}{q m_j}\ex(2 m_j, K_2(r_1, r_2)))$. 
Therefore, for $j =2, \ldots, q$, since $m_j \le (r_1 +r_2 -1)3\epsilon n$,  we get
\begin{align*}
   e(U_j') &= O(n) + O(\ex(m_j, K_2(r_1, r_2))) + O\left(\frac{n+q\epsilon n}{q m_j}\ex(2m_j, K_2(r_1, r_2))\right)\\ &= \epsilon^{1-\frac{1}{r_1}}O(\ex(n, K_2(r_1, r_2))) 
\end{align*}
  by Remark~\ref{comparing turan numbers of complete bipartite graphs}. 
Since there are at most $k_\epsilon$ vertices in $U_j \setminus U_j'$, the same bounds also hold for $e(U_j)$. Thus, 
   \begin{equation}
   \label{eqn: bound on edges in Uj'}
       e(U_j) < e(U_j') + k_\epsilon n \le \epsilon^{1-\frac{1}{r_1}}O(\ex(n, K_2(r_1, r_2))) \text{ for all $j = 2, \ldots, q$.}
   \end{equation}

We have so far obtained an upper bound for the maximum degree of graphs induced by $B_j$ where $j =2, \ldots, q$. Next we will obtain stronger upper bounds for the maximum degree in graphs induced by $U_j'$ for each $j=2, \ldots, q$. Since $B_j \subset U_j'$, this strengthens the upper bound on the maximum degree for the graphs induced by each $B_j$, for $j = 2, \ldots, q$.

\begin{lemma}
\label{lem: max degree r_2 -1 in U_j'}
    The maximum degree of the graph induced by $U_j'$ is at most $r_2 -1$, for all $j = 2, \ldots, q$. 
\end{lemma}
\begin{proof}
    Say there is a star $K_2(1, r_2)$ in $K(n)[U_j']$ for some $2 \le j \le q$. Assume without loss of generality that the  $K_2(1, r_2)$ is contained in $K(n)[U_2']$. Let $A_1$ be the subset of vertices of $U_1'$ that lie in the common neighborhood of this $K_2(1, r_2)$.

    \begin{claim}
The graph induced by the vertices of $A_1$ do not contain any $K_2(r_1 - 1, r_3)$.
\end{claim}
      \begin{claimproof}
    To the contrary, if $K(n)[A_1]$ contains a $K_2(r_1 - 1, r_3)$, then along with the  $K_2(1, r_2)$ contained in $K(n)[U_2']$, this forms a $K_3(r_1, r_2, r_3)$ in $K(n)[U_1' \cup U_2']$. Then repeating the application of Lemma~\ref{lem: size of common neighborhood} as we did above, one can show that this $K_3(r_1, r_2, r_3)$ along with $r_4, \ldots, r_{q+1}$ vertices from $U_3', \ldots, U_q'$, respectively, forms a $K_{q+1}(r_1, \ldots, r_{q+1})$, a contradiction. 
    \end{claimproof}
    
Using arguments similar to those leading to (\ref{eqn: bound on edges in Uj'}), we can show that $U_1'$ also contains at most $\epsilon^{1-\frac{1}{r_1}}O(\ex(n, K_2(r_1, r_2)))$ edges.
To see this, let $m_1 = |U_1' \setminus A_1|$. 
Since $A_1$ is the common neighborhood of a $K_2(1, r_2)$ in $U_2'$, we have by Remark~\ref{rem: structural properties of extremal and stability graphs complete multipartie gamma}($\gamma$) that $m_1 = |U_1' \setminus A_1| \le (1+r_2)3\epsilon n$. Then the number of edges with both endpoints in $U_1' \setminus A_1$ is at most $\ex(m_1, K_2(r_1, r_2))$. If $m_1 \ne 0$, we can similarly partition the vertices of $U_1'$ into at most $\frac{n+ q\epsilon n}{q m_1}$ sets of size at most $m_1$, so that between any such set and the $m_1$ vertices of $U_1' \setminus A_1$ there are no more than $\ex(2m_1, K_2(r_1, r_2))$ many edges, and consequently, 
$e(U_1', U_1' \setminus A_1) = O(\frac{n+qn\epsilon}{q m_1}\ex(2 m_1, K_2(r_1, r_2)))$. Thus, 
\begin{equation*}
    \begin{aligned}
        e(U_1') &\leq e(U_1' \setminus A_1)+ e(U_1', U_1' \setminus A_1) + e(A_1)\\ 
        &\le O\left(\ex(m_1, K_2(r_1, r_2))\right) + O\left(\frac{n+q\epsilon n}{q m_1}\ex(2 m_1, K_2(r_1, r_2))\right) \\ &+ \ex\left(\frac{n+q\epsilon n}{q}, K_2(r_1 -1, r_3)\right).
        \end{aligned}
\end{equation*}
Therefore, using Remark~\ref{comparing turan numbers of complete bipartite graphs} and Lemma~\ref{lem: KST}, we get
\begin{equation*}
    e(U_1') \le \epsilon^{1-\frac{1}{r_1}}O(\ex(n, K_2(r_1, r_2))), 
\end{equation*}
as claimed.

By Remark~\ref{rem: structural properties of extremal and stability graphs complete multipartie}, there are at most $k_{\epsilon}$ vertices in $U_1 \setminus U_1'$. So, 
\begin{equation}
\label{eqn: an upperbound for edges in U_1'}
 e(U_1) < e(U_1') + k_\epsilon n \le \epsilon^{1-\frac{1}{r_1}}O(\ex(n, K_2(r_1, r_2))).   
\end{equation}

Combining (\ref{eqn: bound on edges in Uj'}) and (\ref{eqn: an upperbound for edges in U_1'}) gives
    \begin{equation}
    \label{eqn1: for contradiction for lower bound on number of edges of K(n)}
        e(K(n)) \le E + \epsilon^{1-\frac{1}{r_1}}O(\ex(n, K_2(r_1, r_2))).
    \end{equation}
Consequently, for $\epsilon$ small enough we get that (\ref{eqn1: for contradiction for lower bound on number of edges of K(n)}) contradicts (\ref{eqn: lower bound on e(K(n))}) when $m \le q \epsilon n$.
Thus, $U_2'$ does not contain $K_2(1, r_2)$ and more generally, $U_j'$ for $j \ge 2$ must be $K_2(1, r_2)$-free. 
\end{proof}

Next we show that $U_i = U_i'$ for all $i=1,\ldots, q$.
\begin{lemma}
    For every $i \in [q]$, we have that $U_i = U_i'$, that is, every vertex in $U_i$ is adjacent to less that $\epsilon n$ other vertices in $U_i$.
\end{lemma}

\begin{proof}
Assume to the contrary that there exists some $i \in [q]$ for which there exists a vertex $v \in U_i$ satisfying $d_{U_i}(v) \ge \epsilon n$. Consequently, $d_{U_j}(v) \ge \epsilon n$ for all $j = 1, \ldots, q$, since $\{U_1, \ldots, U_q\}$ forms a $q$-good partition for $V(K(n))$. Remark~\ref{rem: structural properties of extremal and stability graphs complete multipartie beta}($\beta$) implies that $d(v) = \sum_{j=1}^q d_{U_j}(v) \ge  \frac{n}{q}(q-1) - \epsilon n$. 
Therefore, $d_{U_j}(v) \ge 3\epsilon n q r_{q+1} 
+ k_{\epsilon}
$ for all $j \in [q] \setminus \{i\}$. 
This is because otherwise there exists some $l \in [q] \setminus \{i\}$ such that $d_{U_i}(v) \le d_{U_{l}}(v) \le 3\epsilon n q r_{q+1} 
+ k_{\epsilon}
$, and consequently $d(v) \le 6\epsilon n q r_{q+1} 
+ 2k_{\epsilon}
+ (q-2)\left(\frac{n}{q} + \epsilon n\right) < \frac{n}{q}(q-1) - \epsilon n$, contradicting Remark~\ref{rem: structural properties of extremal and stability graphs complete multipartie beta}($\beta$). 
Thus, $d_{U_{j\ne i}}(v) \ge 3\epsilon nq r_{q+1} 
+ k_{\epsilon}
$, as claimed above. 
Since $|U_j \setminus U_j'| \le k_{\epsilon}$ (by Remark~\ref{rem: structural properties of extremal and stability graphs complete multipartie gamma}($\gamma$)), we have $d_{U_j'}(v) \ge 3\epsilon n q r_{q+1}$ for all $j \in [q]\setminus \{i\}$. Now we will proceed by considering two cases depending on whether or not $v \in U_1$.

{\bf Case 1:} $v \in U_1$.

Let $G^*$ be the graph induced by the vertices of $U_1'$ along with $v$. 
First assume that $G^*$ contains a $K_2(r_1, r_2)$. Then $v$ must be a vertex of this $K_2(r_1, r_2)$, since otherwise $U_1'$ contains a $K_2(r_1, r_2)$, which contradicts Lemma~\ref{lem: Forbidden complete bipaprtite graph in U_i'}. 
Now, $d_{U_{j \ne 1}'}(v) \ge 3\epsilon n q r_{q+1}$ and every vertex in $U_l'$ is adjacent to at least $|U_j'| - 3\epsilon n$ vertices of $U_j'$ for all $j \ne l$ (by Remark~\ref{rem: structural properties of extremal and stability graphs complete multipartie gamma}($\gamma$)). Therefore, the common neighborhood in $U_2'$ of the $K_2(r_1, r_2)$ in $G^*$ has size at least

\begin{equation}
    \begin{aligned}
   3\epsilon n q r_{d+1} + (r_1 + r_2 - 1)(|U_2'| - 3\epsilon n) - (r_1 + r_2 - 1)|U_2'|
  &\ge 3\epsilon n q r_{d+1} - (3\epsilon n (q r_{d+1}-1))\\ 
  &\ge 3\epsilon n \ge r_{d+1} \ge r_3.      
    \end{aligned}
\end{equation}
So we can find a $K_3(r_1, r_2, r_3)$ in $\{v\} \cup U_1' \cup U_2'$ as long as $n$ is sufficiently large. Likewise,
we can apply Lemma~\ref{lem: size of common neighborhood} repeatedly to obtain a $K_{q+1}(r_1, \ldots, r_{q+1})$ in $\{v\} \cup U_1' \cup \ldots \cup U_q' \subseteq K(n)$, a contradiction.

Thus, $G^*$ is $K_2(r_1, r_2)$-free. By Lemma~\ref{lem: extremal graphs do not have large internal degree}, there exists a constant $c>0$ such that  $e(G^*) \le (1-c)\ex(n_1 +1, K_2(r_1, r_2))$,
since $d_{U_1'}(v) \ge \epsilon n -k_{\epsilon} \ge \epsilon(n_1+1)$. Moreover, $e(U_i) = O(n)$  for all $i \ge 2$ since $U_i'$ is $K_2(1, r_2)$-free and $e(U_i) - e(U_i') \le k_{\epsilon} n_i$,  because $|U_i \setminus U_i'| \le k_{\epsilon}$.
As a consequence,
\begin{equation}
    e(K(n)) \le e(G^*) + E + O(n) \le E + (1-c')\ex(n_1, K_2(r_1, r_2)) 
\end{equation}
for some positive constant $c'$.
This contradicts (\ref{eqn: lower bound on e(K(n))}) for
$m \le q\epsilon n$.
Thus, $U_1 = U_1'$.

{\bf Case 2:} $v \in U_i$ for some $i \ne 1$.

In case $d_{U_i}(v) \ge 3\epsilon n q r_{q+1} + k_{\epsilon}$, then the proof flows similar to the previous case:
Let $G^*$ be the graph induced by the vertices of $U_1$ along with $v$. 
Note that we have already shown that $U_1 = U_1'$.
If $G^*$ contains a $K_2(r_1, r_2)$. Then $v$ must be a vertex of this $K_2(r_1, r_2)$ and one can find a 
$K_3(r_1, r_2, r_3)$ in $\{v\} \cup U_1' \cup U_2'$ as long as $n$ is sufficiently large. Likewise one can repeatedly apply Lemma~\ref{lem: size of common neighborhood} to obtain a $K_{q+1}(r_1, \ldots, r_{q+1})$ in $\{v\} \cup U_1' \cup \ldots \cup U_q' \subseteq K(n)$, a contradiction.
Thus, $G^*$ is $K_2(r_1, r_2)$-free, and by Lemma~\ref{lem: extremal graphs do not have large internal degree}, there exists a constant $c>0$ such that  $e(G^*) \le (1-c)\ex(n_1 +1, K_2(r_1, r_2))$, 
since $d_{U_1'}(v) \ge \epsilon n$. Moreover, $e(U_i) = O(n)$  for all $i \ge 2$ since $U_i'$ is $K_2(1, r_2)$-free and $e(U_i) - e(U_i') \le k_{\epsilon} n_i$,  due to $|U_i \setminus U_i'| \le k_{\epsilon}$.
As a consequence,
\begin{equation}
    e(K(n)) \le e(G^*) + E + O(n) \le E + (1-c')\ex(n_1, K_2(r_1, r_2))
\end{equation}
for some positive constant $c'$.
This contradicts (\ref{eqn: lower bound on e(K(n))}) for $m \le q \epsilon n$.

Thus, it remains to consider the case $\epsilon n \le d_{U_i}(v) \le 3\epsilon n q r_{q+1} 
+ k_\epsilon
$. Since $|U_i \setminus U_i'| \le k_{\epsilon}$, $v$ must have at least $r_{2}$ neighbors in $U_i'$ in this case. Now let
 $A_1$ be the subset of vertices of $U_1$ that lie in the common neighborhood of this $K_2(1, r_2)$.
Then the graph induced by the vertices of $A_1$ does not contain any $K_2(r_1 - 1, r_3)$ as otherwise the $K_2(r_1 - 1, r_3)$ in $K(n)[A_1]$ along with the  $K_2(1, r_2)$ contained in $K(n)[U_{i}' \cup \{v\}]$ forms a $K_3(r_1, r_2, r_3)$ in $K(n)[U_1 \cup U_{i}' \cup \{v\}]$. 
Next, it follows from Remark~\ref{rem: structural properties of extremal and stability graphs complete multipartie beta}($\beta$) that for any fixed $j \in [q]\setminus\{i\}$ we have $d_{U_{j \neq i}}(v) = d(v) - d_{U_i}(v) - \sum_{l\in[q],l\not\in\{i, j\}}d_{U_l}(v) \ge \frac{n}{q}(q-1) - \epsilon n - 3\epsilon n q r_{d+1} 
- k_\epsilon
- (q-2)(\frac{n}{q}+\epsilon n) \ge \frac{n}{q} - \epsilon n q(3r_{d+1}+1) 
- k_\epsilon
> 3\epsilon n q r_{d+1} 
+ k_\epsilon
$ for $n$ sufficiently large. 
Thus, $d_{U_{j \neq i}'}(v) \ge 3\epsilon n q r_{d+1}$.  
Then repeating the application of Lemma~\ref{lem: size of common neighborhood} as we have done previously, one can show that this $K_3(r_1, r_2, r_3)$ along with $r_4, \ldots, r_{q+1}$ vertices from $U_2', \ldots, U_q'$ (skipping $U_i'$), respectively, form a $K_{q+1}(r_1, \ldots, r_{q+1})$, a contradiction. 
    
Next we will see that $U_1$ contains at most $(3+3r_2+r_3)\epsilon^{1-\frac{1}{r_1}}(\ex(n, K_2(r_1, r_2)))$ edges, by Lemma~\ref{lem: KST} and Remark~\ref{comparing turan numbers of complete bipartite graphs}.
To see this, let $m_1 = |U_1 \setminus A_1|$. 
Since $A_1$ is the common neighborhood of a $K_2(1, r_2)$ in $U_{{i}}' \cup \{v\}$, we have by Remark~\ref{rem: structural properties of extremal and stability graphs complete multipartie gamma}($\gamma$) that $m_1 = |U_1 \setminus A_1| \le(1+r_2)3\epsilon n$. Now $U_1 = U_1'$, so $U_1$ is $K_{r_1, r_2}$-free and therefore the number of edges with both endpoints in $U_1 \setminus A_1$ is at most $\ex(m_1, K_2(r_1, r_2))$. If $m_1 \ne 0$, we can similarly partition the vertices of $U_1 
$ into at most $\frac{n+ q\epsilon n}{q m_1}$ sets of size at most $m_1$, so that between any such set and the $m_1$ vertices of $U_1 \setminus A_1$ there are no more that $\ex(2m_1, K_2(r_1, r_2))$ edges, and consequently, 
$e(U_1 \setminus A_1) + e(A_1, U_1 \setminus A_1) \le \ex(m_1, K_2(r_1, r_2))+  \frac{n+q\epsilon n}{q m_1}\ex(2 m_1, K_2(r_1, r_2)))$. Thus, using Lemma~\ref{lem: KST} and  Remark~\ref{comparing turan numbers of complete bipartite graphs}, 
\begin{equation}
\label{eqn: an upperbound for edges in U_1}
    \begin{aligned}
        e(U_1) &= e(U_1 \setminus A_1)+ e(A_1, U_1 \setminus A_1) + e(A_1)\\ 
        &\le \left(1+  \frac{n+q\epsilon n}{q m_1}\right)\ex(2 m_1, K_2(r_1, r_2)) + \ex\left(\frac{n+q\epsilon n}{q}, K_2(r_1 -1, r_3)\right)\\
        &\le n m_1^{1-\frac{1}{r_1}} + r_3 n^{2-\frac{1}{r_1 - 1}}\\
        &\le(3+3r_2+r_3)\epsilon^{1-\frac{1}{r_1}}\ex(n, K_2(r_1, r_2)),
        \end{aligned}
\end{equation}

as claimed.

Combining  Lemma~\ref{lem: max degree r_2 -1 in U_j'} and (\ref{eqn: an upperbound for edges in U_1}) gives
\begin{equation}
\label{eqn: for contradiction for lower bound on number of edges of K(n)}
\begin{aligned} 
e(K(n)) &\le E + (3+3r_2+r_3)\epsilon^{1-\frac{1}{r_1}}\ex(n, K_2(r_1, r_2)) + O(n)\\
&\le E + 8r_3 \epsilon^{1-\frac{1}{r_1}}\ex(n, K_2(r_1, r_2)).
\end{aligned}    
\end{equation}
    Thus, for $\epsilon$ small enough we get that (\ref{eqn: for contradiction for lower bound on number of edges of K(n)}) contradicts (\ref{eqn: lower bound on e(K(n))}) for any $m \le q\epsilon n$.
    Thus, $U_i = U_i'$  for all $i \ge 2$. 
\end{proof}

\begin{proof}[Proof of Theorems~\ref{thm: turan numbers multipartite graphs} and \ref{thm: connected complement close to turan numbers multipartite graphs}]

It follows from the preceding lemmas that if $n$ is sufficiently large, and we know that $K(n) \in \EX(n, K_{q+1}(r_1, \ldots, r_{q+1}), m)$ for any $m \le q\epsilon n$,
then $K(n)$ is a subgraph of $\vee_{i=1}^q N_i$, where  $N_1$ is a $K_{r_1, r_2}$-free graph and $N_i$ is a $K_2(1, r_2)$-free graph for all $2 \le i \le q$. 
Therefore, $e(K(n)) \le E + \ex(n_1, K_2(r_1, r_2)) + \sum_{i=2}^q \ex(n_i, K_2(1, r_2))$.

In particular, if $m = 0$, that is, $K(n) \in \EX(n, K_{q+1}(r_1, \dots, r_{q+1}))$, then by (\ref{eqn: lower bound on e(K(n))}) we must have equality. This is possible if and only if $N_1$ is an extremal graph for $K_2(r_1, r_2)$ and $N_i$ are extremal graphs for $K_2(1, r_2)$ for all $i \ge 2$. 

This proves that if $K(n) \in \EX(n, K_{q+1}(r_1, \ldots, r_{q+1}))$, then $e(K(n)) = E + \ex(n_1, K_2(r_1, r_2)) + \sum_{i=2}^q \ex(n_i, K_2(1, r_2)).$

The second half of Theorem~\ref{thm: turan numbers multipartite graphs} now follows immediately since any $\hat{N_1}, \ldots, \hat{N_q}$ that satisfy items 4., 5., and 6. of Theorem~\ref{thm: turan numbers multipartite graphs} must be $K_{q+1}(r_1, \ldots, r_{q+1})$-free by Lemma~\ref{lem: product graphs not contain forbidden graph} and have $E + \ex(n_1, K_2(r_1, r_2)) + \sum_{i=2}^q \ex(n_i, K_2(1, r_2))$ edges, therefore belonging to $\EX(n, K_{q+1}(r_1, \ldots, r_{q+1}))$. 

Further, for general $m \le q \epsilon n$
 
there must exist non-negative integers $m_i$ for all $i =0, 1, \ldots, q$ satisfying $\sum_{i = 0}^q m_i \le m$ so that
the number of edges $xy$ with both endpoints coming from different parts (that is $x \in U_i$, $y \in U_j$ and $i\ne j$) are at least $E - m_0$ in number, $e(N_1) \ge \ex(n_{1}, K_2(r_1, r_2)) - m_1$, and $e(N_i) \ge \ex(n_{i}, K_2(1, r_2)) - m_i$ for all $2 \le i \le q$. 

Now, in Theorem~\ref{thm: connected complement close to turan numbers multipartite graphs} we are interested in $K_{q+1}(r_1, \ldots, r_{q+1})$-free graphs that have maximum number of edges among all $n$-vertex graphs with connected complements.  Let $G'(n)$ be one such graph. 

For $K(n) \in \EX(n, K_{q+1}(r_1, \ldots, r_{q+1}))$, note that the complement $\overline{K(n)}$ is disconnected and has connected components $\overline{N_i}$, $i = 1, \dots, q$. Here we know that each $\overline{N_i}$ is connected since $U_i' = U_i$ and therefore the minimum degree of each $\overline{N_i}$ is at least $n_i - \epsilon n > \frac{n_i}{2}$ implying that any two vertices of $\overline{N_i}$ have a common neighbor and so $\overline{N_i}$ is connected.

Observe that we must delete at least $q-1$ edges of $K(n)$ from between the parts (that is, from $E$) so that the complement is connected. And, deleting $q-1$ carefully chosen edges from between the parts suffices in making the complement connected. Thus, for any $G'(n) \in \EX_{cc}(n, K_{q+1}(r_1, \ldots, r_{q+1}))$, we have $e(G'(n)) \geq e(K(n)) - (q-1)$. 
Therefore,
$G'(n)$ is in $\EX(n, K_{q+1}(r_1, \ldots, r_{q+1}), q-1)$. 
Next we will observe that
$$e(G'(n)) = \ex(n, K_{q+1}(r_1, \ldots, r_{q+1})) - (q-1).$$ We know that since $G'(n) \in \EX(n, K_{q+1}(r_1, \ldots, r_{q+1}), m)$ where $m= q-1$, that there exist some non-negative $m_i$ for all $0\leq i\leq q$ such that $\sum_{i = 0}^q m_i \le q-1$.  Since $\overline{G'(n)}$ is connected, we must have that vertices lying in separate partite sets $U_i$ and $U_j$ must be connected in $\overline{G'(n)}$ by edges. Therefore, $\overline{G'(n)}$ must have at least $q-1$ edges connecting the partite sets. Consequently, $m_0 = q-1$ and $m_i = 0$ for all $1 \le i \le q$. Thus, $e(G'(n)) = e(K(n)) - (q-1)$. Further, adding the $q-1$ missing edges of $E$, we obtain a graph $G(n)$ in $\EX(n, K_{q+1}(r_1, \ldots, r_{q+1}))$, thus we can see that every $K_{q+1}(r_1, \ldots, r_{q+1}))$-free graph with a connected complement and maximum number of edges can be obtained by deleting some carefully chosen $q-1$ edges of a graph $G(n)$ in $\EX(n, K_{q+1}(r_1, \ldots, r_{q+1}))$. And conversely, we can take any $K(n)$ in $\EX(n, K_{q+1}(r_1, \ldots, r_{q+1}))$ and deleting some carefully chosen $q-1$ edges from it that go between its parts we can obtain a $K_{q+1}(r_1, \ldots, r_{q+1}))$-free graph with maximum number of edges among all graphs with connected complements. 
\end{proof}

\begin{remark}
\label{rem: stability remark}
    Note that it follows from the proof above that for any arbitrarily small positive constant $\epsilon$, as $n \to \infty$, if $H$ is some graph in $\EX(n, K_{q+1}(r_1, \ldots, r_{q+1}), m)$ for $m \le q \epsilon n$,
then there must exist non-negative integers $m_i$ for all $i =0, 1, \ldots, q$ satisfying $\sum_{i = 0}^q m_i \le m$ so that
the number of edges $xy$ with both endpoints coming from different parts (that is $x \in U_i$, $y \in U_j$ and $i\ne j$) are at least $E - m_0$ in number (where $E$ is the number of all pairs of vertices which come from different parts), $N_1$ is a $K_2(r_1, r_2)$-free graph such that $e(N_1) \ge \ex(n_1, K_2(r_1, r_2)) - m_1$, and $N_i$ are $K_2(1, r_2)$-free graphs with $e(N_i) \ge \ex(n_i, K_2(1, r_2)) - m_i$ for all $2 \le i \le d$.
\end{remark}

We end this section with a Corollary of Theorem~\ref{thm: connected complement close to turan numbers multipartite graphs}.

\begin{corollary}
    Let $\tau=2k$. Let $n$ be sufficiently large and $\tau \mid n$. Then the set $\mathcal{T}_{n,2k}$ is the complete set of edge minimizers in $\mathcal{D}_{n,2k}$.
\end{corollary}

\begin{proof}
    We saw in Theorem~\ref{edge minimizer} that the graphs in $\mathcal{T}_{n, 2k}$ are edge minimizers in $\mathcal{D}_{n,2k}$. It remains to show that these are the only edge minimizers in $\mathcal{D}_{n,2k}$ when $n$ is sufficiently large and $\tau \mid n$. We have observed that the edge minimizers in $\mathcal{D}_{n, 2k}$ are the compliments of graphs in  $\EX_{cc}(n, L_{2k+1})$.
    Applying Theorems~\ref{thm: Erdos-Sim comp. multipartite} and \ref{thm: connected complement close to turan numbers multipartite graphs} to $L_{2k+1} = K_{k+1}(1, 2, \ldots, 2)$, we see that when $\tau \mid 2k$ and $n$ is sufficiently large, then all the edge maximizers in $\EX_{cc}(n, L_{2k+1})$ are obtained by deleting $k-1$ edges from $G_{n, 2k}$ (see Figure~\ref{graph in Gn2k}), so as to make the compliments connected.
    It can now be observed that the complement of any graph in $\EX_{cc}(n, L_{2k+1})$ is in $\mathcal{T}_{n, 2k}$. This completes the proof.
\end{proof}

\section{Lower bounds for edges and spectral radius of graphs with given $d$-independence number}
\label{sec: lower bounds on spectral minimizers for d-independnce number}
Generalizing the concepts of independence sets and dissociation sets for higher degrees $d$ leads to the concept of $d$-independent sets. 
 A subset of vertices of a graph $G$ is said to be $d$-independent if it induces a subgraph with maximum degree at most $d$. The cardinality of a largest  $d$-independent set in a graph $G$ is called the \textit{$d$-independence number} of $G$ and is denoted by $i_d(G)$. 
 In the following, we explore the minimum number of edges among all graphs on $n$ vertices with given $i_d(G)$, for any $d \ge 2$. Note that when $d = 0, 1$ this involves fixing the independence numbers and dissociation numbers, respectively. We will apply Theorem~\ref{thm: connected complement close to turan numbers multipartite graphs} to obtain lower bounds for the number of edges and spectral radius for graphs with $d$-independence number $s$, when the order $n$ is sufficiently large.

\subsection{Connecting edge minimizers for given $d$-independence number to Tur\'an problems for complete multipartite graphs}
Consider the following sets of graphs on $t$ vertices:
Let $\calL_{t, d} = \{L \mid v(L) = t, \Delta(L) \le d\}$ and $\calH_{t,d} = \{K_t - L \mid L \in \calL_{t,d}\}$. That is, $\calL_{t, d}$ is the collection of all graphs on $t$ vertices with maximum degree at most $d$, and $\mathcal{H}_{t, d}$ is the collection of their complementary graphs. In Section~\ref{sec: Edge minimizers and connection to Turan problems}, we connected the problem of minimizing the number of edges in an $n$-vertex graph with fixed dissociation number $\tau$ for $\tau =2s$ and $2s+1$ to the Tur\'an problems determining $\ex(n, K_{s+1}(1,2,2,\ldots, 2))$ and $\ex(n, K_{s+1}(2,2,2,\ldots, 2))$, respectively. The connection in Section~\ref{sec: Edge minimizers and connection to Turan problems} arose because \begin{itemize}
        \item $\calL_{\tau+1,1}$ consisted of all possible induced graphs which could be the evidence for the dissociation number of the graph being at least $\tau + 1$ and $\calH_{\tau+1,1}$ consisted of the complements of the graphs in $\calL_{\tau + 1, 1}$. Further, we established in Proposition~\ref{unique L5 free} that $\{K_n - H_n \mid H_n \in \EX(n, \calH_{\tau + 1, 1})\}$ is the collection of graphs with the minimum number of edges among all graphs with dissociation number $\tau$.
        \item Further, if $M_{n}$ denotes the maximal matching on $n$ vertices, then any graph in $\mathcal{L}_{\tau + 1, 1}$ is a subgraph of $M_{\tau + 1}$. Thus, for any $H \in \mathcal{H}_{\tau + 1, 1}$, we have that $K_{\tau + 1, 1} - M_{\tau + 1} \subset H$. Therefore, $\ex(n, \calH_{\tau+1, 1}) = \ex(n, K_{\tau+1} - 
        M_{\tau + 1})$.
    \end{itemize}
    
The problem of determining the minimum number of edges in an $n$-vertex graph with $d$-independence number fixed to be $s$ is similarly connected to the Tur\'an problem $\ex(n,\calH_{s+1, d})$. The following is the connection:
    \begin{itemize}
        
    \item  For any graph $H_n$ in $\EX(n, \calH_{s+1, d})$, its complement graph $K_n - H_n$ has $i_d(K_n-H_n)\leq s$. In fact, $i_d(K_n-H_n)=s$ (similar to Theorem~\ref{thm: turan connection}).  Consequently, $K_n-H_n$ has the minimum number of edges among all graphs on $n$ vertices with the $d$-independence number equal to $s$.

    \item Among all graphs in $\calH_{s+1, d}$, the graph $K_{\lceil\frac{s+1}{d+1}\rceil}(a,d+1,\ldots,d+1)$, where $a \equiv s+1 \pmod {d+1}$ and $1 \le a \le d+1$, has the least chromatic number. We will prove this fact in Lemma~\ref{lem: min chromatic number graph in H(s+1,d)}.
    
    \item Thus, by Remark~\ref{rem: ErdosSim on general structure of extremal graphs}, $\ex(n,\calH_{s+1, d}) = \Theta(\ex(n,K_{\lceil\frac{s+1}{d+1}\rceil}(a, d+1,\ldots,d+1))) = \binom{n}{2} (1 - \frac{1}{\lceil\frac{s+1}{d+1}\rceil-1} + o(1))$.  We determined  $\EX(n, K_{\lceil\frac{s+1}{d+1}\rceil}(a,d+1,\ldots,d+1))$ when $d+1 \ge (a - 1)! + 1$ (see Theorem~\ref{thm: turan numbers multipartite graphs}), which may be used to obtain an upper bound for $\ex(n,\calH_{s+1, d})$.
    \end{itemize}

Recall that for a given family of graphs $\mathcal{F}$, $\ex_{cc}(n, \mathcal{F})$ denotes the maximum number of edges in any $\mathcal{F}$-free graph on $n$ vertices with a connected complement, and $\EX_{cc}(n \mathcal{F})$ denotes the collection of all $\mathcal{F}$-free graphs on $n$ vertices with $\ex_{cc}(n, \mathcal{F})$ edges that have a connected complement.

\begin{lemma}
\label{lem: d-independence number for complements of extremal graphs}
    For any two integers $s, d$ with $s > d \ge 0$, and any positive integer $n$, let $G$ be any graph in $\EX(n, \mathcal{H}_{s+1, d})$. Then $\tau(\overline{G}) = s$. Further, if $G'$ is any graph in $\EX_{cc}(n, \mathcal{H}_{s+1, d})$, we also have that $\tau(\overline{G'}) = s$  
\end{lemma}
\begin{proof}
   Let $G$ and $G'$ be arbitrary graphs in $\EX(n, \mathcal{H}_{s+1, d})$ and $\EX_{cc}(n, \mathcal{H}_{s+1, d})$, respectively. Then for any $G^* \in \{G, G'\}$, $G^*$ is $\mathcal{H}_{s+1, d}$-free and therefore any induced subgraph on $s+1$ vertices in $\overline{G^*}$ has maximum degree at least $d+1$. Thus, $\tau(\overline{G^*}) \le s$. 
   Further, adding any edge $e$ to $G$ creates a copy of some $H \in \mathcal{H}_{s+1, d}$, and adding any edge $e'$ to $G'$ such that the complement of the resultant graph remains connected, must also create a copy of some graph $H'$ in $\mathcal{H}_{s+1, d}$. Then $G$ and $G'$ have a copy of $H-e$ and $H'-e'$ as subgraphs.
   
   Thus, there are subsets $S$ and $S'$ of size $s+1$ in $\overline{G}-e$ and $\overline{G'} - e'$, respectively, which induce a graph with maximum degree at most $d$, but $\overline{G}$ and $\overline{G'}$ do not contain any set of $s+1$ vertices that induce a graph with maximum degree at most $d$.
   However, we show next that there are sets of size $s$ which induce a graph with maximum degree $d$ in both $G$ and $G'$.
   Say $e = uv$ and $e'= u'v'$. Then
   $G[S \setminus \{u\}]$ and $G[S' \setminus \{u'\}]$ have maximum degree $d$.  
 Thus, $\tau(\overline{G}) = \tau(\overline{G'}) = s$. 
\end{proof}
\begin{lemma}
\label{lem: min chromatic number graph in H(s+1,d)}
    For any two integers $s, d$ with $s > d \ge 0$ and $a \equiv s+1 \pmod {d+1}$ with $1 \le a \le d+1$, let $H$ be any graph in $\mathcal{H}_{s+1, d}$. Then $\chi(H) \ge \chi(K_{\lceil\frac{s+1}{d+1}\rceil}(a,d+1,\ldots,d+1))$.
\end{lemma}

\begin{proof}
    Let $H$ be an arbitrary member of $\mathcal{H}_{s+1, d}$. Then, by the definition of $\mathcal{H}_{s+1, d}$, the maximum degree of its complementary graph $\Delta(\overline{H}) \le d$. Consequently, the independence number $\alpha(H) \le d+1$, since otherwise there is a vertex with degree more than $d$ in $\overline{H}$. Since any color class of $H$ is an independent set, we have that the size of any color class of $H$ is at most $d+1$. Thus, $\chi(H) \ge  \lceil\frac{s+1}{d+1}\rceil= \chi(K_{\lceil\frac{s+1}{d+1}\rceil}(a,d+1,\ldots,d+1))$  as desired.   
\end{proof}

\begin{proposition}
\label{thm: lower bound for spectral radius given d-independence number}
For any positive integer $s$, if $G$ is a connected graph on $n$ vertices with $d$-independence number $s$,
then \[\lambda(G) \ge \dfrac{2(\binom{n}{2} - \ex_{cc}(n, \mathcal{H}_{s+1, d}))}{n} \ge \dfrac{2(\binom{n}{2} - \ex_{cc}(n, K_{\lceil\frac{s+1}{d+1}\rceil}(a, d+1, \ldots, d+1)))}{n},\]  where $a \equiv s+1 \pmod {d+1}$ and $1 \le a \le d+1$.

\end{proposition}

\begin{proof}
Let $s$ be any positive integer. Let $G$ be a connected graph on $n$ vertices that has $d$-independence number $s$. 
Then, 
$G$ cannot have any induced subgraph on $s+1$ vertices that belongs to the collection $\calL_{s+1, d}$.
Further, 
$\overline{G}$ is $\calH_{s+1, d}$-free. Moreover, by Lemma~\ref{lem: d-independence number for complements of extremal graphs} if $H$ is any graph in $\EX(n, \calH_{s+1, d})$ or $\EX_{cc}(n, \calH_{s+1, d})$, then $\overline{H}$ must have $d$-independence number $s$. If $\overline{H}$ is connected, that is, $H \in \EX_{cc}(n, \calH_{s+1, d})$, then $e(\overline{H})$ is the minimum number of edges in any connected $n$-vertex graph with $d$-independence number $s$. 
Therefore, $e(G) \ge \binom{n}{2} - \ex_{cc}(n, \calH_{s+1, d})$, with equality if and only if $\overline{G} \in \EX_{cc}(n, \calH_{s+1, d})$. Since the average degree of a graph is a lower bound for its spectral radius, we get that 
   \[\lambda(G) \ge \frac{2e(G)}{n} \ge \frac{2(\binom{n}{2} - \ex_{cc}(n, \calH_{s+1, d}))}{n}.\]

Further, for $a \equiv s+1 \pmod {d+1}$ and $1 \le a \le d+1$, we have $K_{\lceil\frac{s+1}{d+1}\rceil}(a, d+1, \ldots, d+1) \in \mathcal{H}_{s+1, d}$. Therefore, $$\ex_{cc}(n, K_{\lceil\frac{s+1}{d+1}\rceil}(a, d+1, \ldots, d+1)) \ge  \ex_{cc}(n, \calH_{s+1, d})$$ giving 
\[\lambda(G) \ge \frac{2e(G)}{n} \ge \frac{2(\binom{n}{2} - \ex_{cc}(n, \calH_{s+1, d}))}{n} \ge  \frac{2(\binom{n}{2} - \ex_{cc}(n, K_{\lceil\frac{s+1}{d+1}\rceil}(a, d+1, \ldots, d+1)))}{n}.\]
\end{proof}

From Remark~\ref{rem: ErdosSim on general structure of extremal graphs}(1) we know that the asymptotics of $\ex(n, \calH_{s+1, d})$ is controlled by the graphs in $\calH_{s+1, d}$ with the least chromatic number. 
By Lemma~\ref{lem: min chromatic number graph in H(s+1,d)}, we know that $\lceil\frac{s+1}{d+1}\rceil$ is the minimum chromatic number among all graphs in $\calH_{s+1, d}$. Now denote the subcollection of all graphs in $\calH_{s+1, d}$ with chromatic number $\lceil\frac{s+1}{d+1}\rceil$ by $X_{s+1, d}$. Then $X_{s+1, d}$ contains 
\[\left\{K_{\lceil\frac{s+1}{d+1}\rceil}(r_1, r_2, \ldots, r_{\lceil\frac{s+1}{d+1}\rceil}) \; \bigg| \; \sum_{i=1}^{{\lceil\frac{s+1}{d+1}\rceil}}r_i  = s+1, \textbf{ and } 1 \le r_i \le d+1 \text{ for all } 1 \le i \le \left\lceil\frac{s+1}{d+1}\right\rceil \right\}.\]
Here, we have ordered the $r_i$ such that $r_1 \le r_2 \le \ldots \leq r_{\left\lceil\frac{s+1}{d+1}\right\rceil}$. Then $r_1 \ge a$, where $a$ satisfies $a \equiv s+1 \pmod {d+1}$, $1 \le a \le d+1$. Moreover, it seems plausible that the member of $X_{s+1, d}$ that may best determine $\ex(n, X_{s+1, d})$ and consequently $\ex(n, \calH_{s+1, d})$, is $K_{\lceil\frac{s+1}{d+1}\rceil}(a, d+1, \ldots, d+1)$.  Similarly, if we require connected complements,  $\ex_{cc}(n, X_{s+1, d})$ and $\ex_{cc}(n, \calH_{s+1, d})$ are best determined by $K_{\lceil\frac{s+1}{d+1}\rceil}(a, d+1, \ldots, d+1)$.
\begin{remark}
    If we know that $d+1 \ge (a - 1)! +1$, then Theorem~\ref{thm: connected complement close to turan numbers multipartite graphs} gives us the structure of any graph in $\EX_{cc}(n, K_{\lceil\frac{s+1}{d+1}\rceil}(a, d+1, \ldots, d+1)$. Then if one obtains strong bounds on $\ex(n, K_2(a, d+1))$, we can use that to calculate $\ex(n, K_{\lceil\frac{s+1}{d+1}\rceil}(a, d+1, \ldots, d+1))$ and get a numerical lower bound for $\lambda(G)$ in Proposition~\ref{thm: lower bound for spectral radius given d-independence number}.
\end{remark}
\section{Spectral minimizers for $\tau =2k, k\geq 3$}
\label{sec: spectral minimizers for even dissociation numbers}

In this section, we apply Theorem~\ref{thm: connected complement close to turan numbers multipartite graphs} to obtain a few results on the spectral minimizers for a given even dissociation number, for sufficiently large order $n$ as long as $n$ satisfies some parity conditions.
For any fixed integer $l > 0$ and even integer $m > 0$, let $C_1, \ldots, C_l$ be $l$ copies of $CP_{m}$. For all $i \in [l]$, let $\{u_i, v_i\} \subset C_i$ be a pair of vertices. Let $S = \cup_{i=1}^l\{u_i, v_i\}$. 
\begin{definition}
  Let $\mathcal{P}_l(m, S)$ denote a graph isomorphic to the graph with vertex set $\cup_{i=1}^l V(C_i)$ and edge set $\cup_{i=1}^l E(C_i) \cup_{i=1}^{l-1} v_i u_{i+1}$. We will refer to graphs of this forms as \emph{CP-paths}. If $u_i$ and $v_i$ form a non-adjacent pair of vertices for each $i \in [l]$, we  will refer to these graphs as \emph{aligned} \emph{CP-paths} (see Figure~\ref{fig: CP-path} (left)).     
\end{definition}
\begin{definition}\label{def: cp cycle}
  Let $\mathcal{C}_l(m, S)$ be a graph isomorphic to one obtained by additionally also including the edge $v_l u_1$ to $\mathcal{P}_l(m, S)$.  We will refer to  graphs of this form as \emph{CP-cycles}. And if $u_i$ and $v_i$ form a non-adjacent pair of vertices for each $i \in [l]$, we  will refer to these graphs as \emph{aligned CP-cycles} (see Figure~\ref{fig: CP-path} (right)).
\end{definition}
When referring to aligned CP-paths and cycles we will drop $S$ from the notation and simply use $\mathcal{P}_l(m)$ and $\mathcal{C}_l(m)$.

Let $\tau  = 2k$ and $\tau \mid n$. Recall that $\mathcal{T}_{n,\tau}$ is the set of all graphs of order $n$ and dissociation number $\tau$ obtained by replacing each vertex of a tree on $k$ vertices with a cocktail party graph $CP_{\frac{n}{k}}$.
Then $\mathcal{P}_k(\frac{n}{k}) \in \mathcal{T}_{n, \tau}$ and both $\mathcal{P}_k(\frac{n}{k})$ and $\mathcal{C}_k(\frac{n}{k})$ are connected graphs of order $n$ with dissociation number $2k$. Hence, they are in the set $\mathcal{D}_{n,\tau}$. For simplicity, in this section when $m = \frac{n}{k}$, we will use $\mathcal{P}_l$ and $\mathcal{C}_l$ to denote $\mathcal{P}_l(\frac{n}{k})$
 and  $\mathcal{C}_l(\frac{n}{k})$, respectively and $\mathcal{P}_l(S)$ and $\mathcal{C}_l(S)$ to denote $\mathcal{P}_l(\frac{n}{k}, S)$ and $\mathcal{C}_l(\frac{n}{k}, S)$, respectively.

\begin{figure}[h!]
\begin{center}
     \begin{tikzpicture}[every path/.append style={thick}, scale =1.5]
       \draw (0,0) ellipse (0.5cm and 1cm);
     \draw (1.5,0) ellipse (0.5cm and 1cm); 
     \draw (3,0) ellipse (0.5cm and 1cm);
     \foreach \x in {(-0.2,0.3), (0.2,0.3), (1.3,0.3), (1.7,0.3), (2.8, 0.3), (3.2, 0.3)}
     {\draw[fill] \x circle[radius = 0.05cm];}
     \draw (0.2,0.3)--(1.3,0.3);
     \draw (1.7,0.3)--(2.8, 0.3);
  
    \node at (-0.2,0) {\small $u_1$};
    \node at (0.2,0) {\small$v_1$};
    \node at (1.3,0) {\small $u_2$};
    \node at (1.7,0) {\small $v_2$};
    \node at (2.8,0) {\small $u_3$};
    \node at (3.2,0) {\small $v_3$};
    \node at (0, -1.5) {\small $CP_{m}$};
    \node at (1.5, -1.5) {\small $CP_{m}$};
    \node at (3, -1.5) {\small $CP_{m}$};
    \end{tikzpicture}
    \hspace{0.5cm}
    \begin{tikzpicture}[every path/.append style={thick}, scale =1.5]
       \draw (0,0) ellipse (0.5cm and 1cm);
     \draw (1.5,0) ellipse (0.5cm and 1cm); 
     \draw (3,0) ellipse (0.5cm and 1cm);
     \foreach \x in {(-0.2,0.3), (0.2,0.3), (1.3,0.3), (1.7,0.3), (2.8, 0.3), (3.2, 0.3)}
     {\draw[fill] \x circle[radius = 0.05cm];}
     \draw (0.2,0.3)--(1.3,0.3);
     \draw (1.7,0.3)--(2.8, 0.3);
    \draw (3.2,0.3) to [out= 30, in=150](-0.2,0.3);
    \node at (-0.2,0) {\small $u_1$};
    \node at (0.2,0) {\small$v_1$};
    \node at (1.3,0) {\small $u_2$};
    \node at (1.7,0) {\small $v_2$};
    \node at (2.8,0) {\small $u_3$};
    \node at (3.2,0) {\small $v_3$};
    \node at (0, -1.5) {\small $CP_{m}$};
    \node at (1.5, -1.5) {\small $CP_{m}$};
    \node at (3, -1.5) {\small $CP_{m}$};
    \end{tikzpicture}
    \end{center}
    \caption{The CP-path $\mathcal{P}_{3}(m)$ and the CP-cycle $\mathcal{C}_{3}(m)$.}
    \label{fig: CP-path}
\end{figure}
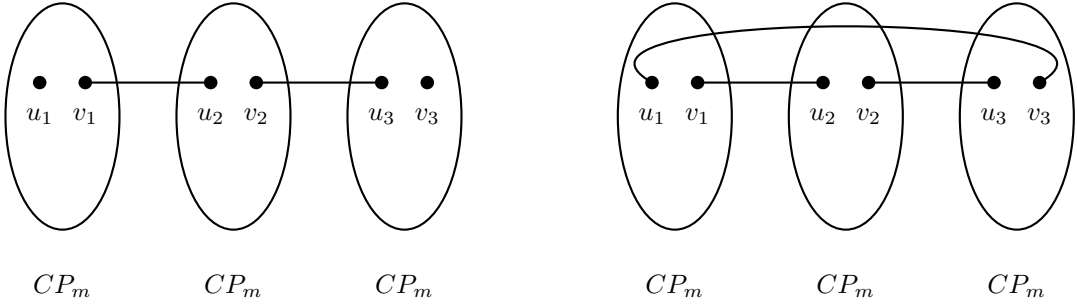

We believe that among all graphs in $\mathcal{D}_{n,2k}$, the aligned CP-path $\mathcal{P}_{k}$ is the unique spectral minimizer. Progressing in that direction, we first prove that if $G \in \mathcal{D}_{n, 2k}$ is a spectral minimizer, then $G\in\mathcal{T}_{n,2k}$.

\begin{theorem}\label{spec is edge minimizer}
  For given dissociation number $\tau =2k$, $k \in \mathbb{N}$, if $n$ is sufficiently large and $\tau \mid n$, then any $G$ which is a spectral minimizer in $\mathcal{D}_{n,\tau}$ is also an edge minimizer in $\mathcal{D}_{n,\tau}.$
\end{theorem}
\begin{proof}
    We have seen that any edge minimizer in $\mathcal{D}_{n, \tau}$ is obtained as the complement of some graph $G'(n) \in \EX(n, K_{k+1}(1, 2 \ldots, 2), k-1)$ with $\ex(n, K_{k+1}(1, 2, \ldots, 2)) - (k-1)$ edges, obtained itself from some graph $G(n) \in \EX(n, K_{k+1}(1, 2, \ldots, 2))$ by deleting some carefully chosen $k-1$ edges which make the complement graph connected. Thus, edge minimizers in $\mathcal{D}_{n, \tau}$ have $\binom{n}{2} - \ex(n, K_{k+1}(1, 2, \ldots, 2))+ k-1$ many edges. Next we will show that for $\tau = 2k$ and $\tau \mid n$, if $G \in \mathcal{D}_{n, \tau}$ is a spectral minimizer, then $G$ is also an edge minimizer of $\mathcal{D}_{n, \tau}$. For this, consider any $K(n) \in \EX(n, K_{k+1}(1, 2 \ldots, 2)) = \vee_{i=1}^k N_i$ where $N_i \in \EX(n_i, K_2(1,2)) = M_{n_i}$, a perfect matching on $n_i$ vertices, for all $i \in [k]$. It can be observed that each $n_i = \frac{n}{k}$ since $K(n)$ has the maximum number of edges. 
    Consequently, $e(G) \ge \frac{n}{2}\left(\frac{n}{k} - 2\right) + k - 1$.

Let $\Tilde{G}$ be the aligned CP-path $\mathcal{P}_{k}$ and let $\hat{G}$ be the aligned CP-cycle $\mathcal{C}_{k}$.  Since $G$ is a spectral minimizer in $\mathcal{D}_{n, \tau}$, $\rho(G) \le \rho(\Tilde{G}) < \rho(\hat{G})$. To calculate the spectral radius $\rho(\hat{G})$ consider the following $2 \times 2$ quotient matrix $Q$ of the adjacency matrix of $\hat{G}$ where the quotient corresponds to the vertex partition $A = \{u_1, v_1, \ldots, u_k, v_k\}$ and $B = V(G) \setminus A$  given as follows:
    \[Q = 
    \begin{bmatrix}
    1&\frac{n}{k}-2\\
    2&\frac{n}{k}-4\\
    \end{bmatrix}.
    \]

Since the partition is an equitable partition,
\begin{equation}\label{eq: spectral radius of cp cycle}
  \rho(\hat{G}) = \rho(Q) =  \frac{\frac{n}{k}-3 + \sqrt{(\frac{n}{k} - 1)^2 + 8}}{2} < \frac{n}{k} - 2 + \frac{2k}{n-k}. 
\end{equation} 
Therefore, $$e(G) \le \frac{n\rho(G)}{2} < \frac{n^2}{2k} - n + \frac{nk}{n-k} < \frac{n}{2}\left(\frac{n}{k} - 2\right) + (k -1) + 2,$$
where the last inequality holds for any $n > k^2 + k$.
This implies that $$\frac{n}{2}\left(\frac{n}{k} - 2\right) + (k -1) \le e(G) \le \frac{n}{2}\left(\frac{n}{k} - 2\right) + k,$$ for $n$ sufficiently large. This tells us that $\overline{G} \in \EX(n, K_{k+1}(1, 2, \ldots, 2), k)$.
Now, if $G$ is not an edge minimizer in $\mathcal{D}_{n, \tau}$, then $e(G) =  \frac{n}{2}(\frac{n}{k} - 2) + k$ and $\overline{G} \not\in \EX(n, K_{k+1}(1, 2, \ldots, 2), k-1)$. However, since the spectral radius increases on adding edges to a connected graph, it must be that $G$ does not contain any connected subgraph in $\mathcal{D}_{n, \tau}$. By Remark~\ref{rem: stability remark}, we have that $n_i = \frac{n}{k}+1$ for some $i \in [k]$ and the induced subgraph $G[U_i] = L_{\frac{n}{k}+ 1}$ has spectral radius at least $\frac{n}{k} - 1$. Then the minimum degree of $G[U_i] = \frac{n}{k} - 1$, which contradicts a previous observation that $\rho(G) < \frac{n}{k} - 2 + \frac{2k}{n-k}$. Thus, $G$ must also be an edge minimizer.
\end{proof}

Now, it is a natural question to ask here that among the graphs in the set $\mathcal{D}_{n, 2k}\setminus \mathcal{T}_{n,2k}$ which graphs have the smallest spectral radius. We will prove that the aligned CP-cycle $\mathcal{C}_k$ is the unique spectral minimizer in $\mathcal{D}_{n, 2k}\setminus \mathcal{T}_{n,2k}$. This will resolve Problem~\ref{prob: cycle minimizer}. We will need the following lemmas.

\begin{lemma}\label{lem: nonadjacent connectors are better}
Consider the graph $G$ in Figure~\ref{fig: nonadjacent connectors are better}. Let $G_1, G_2$ be any two graphs. Let $u_1, v_1$ be a non-adjacent pair in $CP_m$ and let $u_1\sim a$ and $v_1\sim b$, where $a\in V(G_1)$ and $b\in V(G_2)$.  Let $u_2, v_2$ be another non-adjacent pair in $CP_m$. Let $G' = G +u_1v_1 + u_2v_2 - u_1v_2 - u_2v_1$. Then, $\rho(G')>\rho(G)$.
\end{lemma}
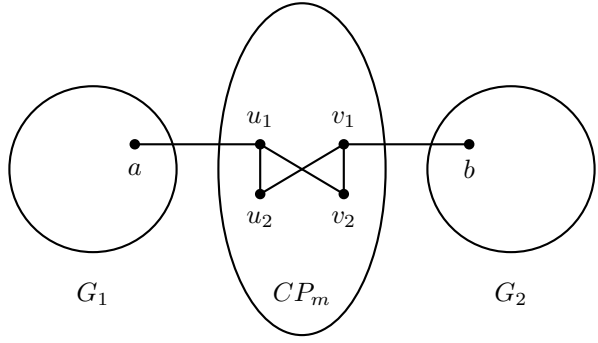
\begin{figure}
\centering
\begin{tikzpicture}[every path/.append style={thick}, scale =1.1]
    \draw (-1,0) circle (1cm);
     \draw (1.5,0) ellipse (1cm and 2cm); 
     \draw (4,0) circle (1cm);
     \foreach \x in {(-0.5,0.3), (1,0.3), (2,0.3), (1,-0.3), (2, -0.3), (3.5, 0.3)}
     {\draw[fill] \x circle[radius = 0.05cm];}
     \draw (-0.5,0.3)--(1,0.3);
     \draw (2,0.3)--(3.5, 0.3);
     \draw (1,0.3)--(1,-0.3)--(2,0.3)--(2,-0.3)--(1,0.3);

    \node at (-0.5,0.0){\small$a$};
    \node at (1,0.6) {\small $u_1$};
    \node at (2,0.6) {\small $v_1$};
     \node at (1,-0.6) {\small $u_2$};
    \node at (2,-0.6) {\small $v_2$};
    \node at (3.5, 0.0) {\small $b$};
    \node at (-1, -1.5) {\small $G_1$};
    \node at (1.5, -1.5) {\small $CP_m$};
    \node at (4, -1.5) {\small $G_2$};
    \end{tikzpicture}
    \caption{The graph $G$ in Lemma~\ref{lem: nonadjacent connectors are better}.}
 \label{fig: nonadjacent connectors are better}
\end{figure}
\begin{proof}
    Let $x$ be the principal eigenvector of $G$. Then by the Rayleigh quotients, we have that
    \begin{align*}
     \rho(G')&\geq x^TA(G')x\\
     &= x^TA(G)x + 2x(u_1)x(v_1) +2x(u_2)x(v_2) -2x(u_1)x(v_2)-2x(u_2)x(v_1)\\
     &=\rho(G) +2(x(u_1)-x(u_2))(x(v_1)-x(v_2)).
    \end{align*}
 If $x(u_1)\geq x(u_2)$ and $x(v_1)\geq x(v_2)$, then we are done. So, suppose $x(u_1)<x(u_2)$ or $x(v_1)<x(v_2)$. WLOG, suppose $x(u_1)<x(u_2)$. Consider the graph $\hat{G} = G_1 - u_1a + u_2a$. Note that $\hat{G}$ is isomorphic to graph $G'$ and 
 \begin{align*}
     \rho(\hat{G})&\geq x^TA(\hat{G})x\\
     &=x^TA(G)x + 2x(a)(x(u_2)-x(u_1))\\
     &>\rho(G).
 \end{align*}
 This completes the proof.
\end{proof}
\begin{lemma}
\label{lem: CP-cycle with minimum spectral radius}
Let $l, m$ be fixed positive integers and $m$ be an even number. If $S$ is any set of the form $\cup_{i=1}^l \{u_i, v_i\}$ where $\{u_i, v_i\} \subset C_i$ for all $i \in [l]$. 
Then among all CP-cycles of the form $\mathcal{C}_l(m, S)$, the aligned CP-cycle $\mathcal{C}_l(m)$ is the unique CP-cycle with minimum spectral radius.
\end{lemma}
\begin{proof}
We will apply Lemma~\ref{lem: nonadjacent connectors are better} to show that any graph of the form $\mathcal{C}_k(m, S)$ has spectral radius no less than $\rho(\mathcal{C}_l)$, with $\mathcal{C}_k(m, S) = \rho(\mathcal{C}_l)$ only if $\mathcal{C}_k(m, S) \cong \mathcal{C}_l$. Say $S$ contains an adjacent pair of vertices $\{u_i, v_i\} \subset C_i$ for some $i \in [l]$. Then by Lemma~\ref{lem: nonadjacent connectors are better} we have that $\mathcal{C}_k(m, S') < \mathcal{C}_k(m, S)$, where $S' = (S\setminus\{v_i\})\cup\{u_i'\}$ and $\{u_i, u_i'\}$ form a non-adjacent pair of vertices in $C_i$.  Thus, the CP-cycle with the minimum spectral radius must be an aligned CP-cycle.  
\end{proof}
\begin{theorem}
For given dissociation number $\tau =2k$, $k \in \mathbb{N}$, let $n$ be sufficiently large and $\tau \mid n$. If $G$ is a spectral minimizer in the set $\mathcal{D}_{n, 2k}\setminus \mathcal{T}_{n,2k}$, then $G\cong \mathcal{C}_k$.
\end{theorem}
\begin{proof}
Since $G$ is a spectral minimizer of $\mathcal{D}_{n, 2k}\setminus \mathcal{T}_{n,2k}$, and $\mathcal{C}_k \in \mathcal{D}_{n, 2k}\setminus \mathcal{T}_{n,2k}$, it is true that $\rho(G) \le \rho(\mathcal{C}_k)$.
   Consequently. $e(G) \le \rho(\mathcal{C}_k)\frac{n}{2} < \frac{n}{2}\left(\frac{n}{k} - 2\right) + k+1$. 
   Since the only graphs in $\mathcal{D}_{n, 2k}$ with $\frac{n}{2}\left(\frac{n}{k} - 2\right) + k-1$ edges are those in $\mathcal{T}_{n, 2k}$, we have that $e(G) = \frac{n}{2}\left(\frac{n}{k} - 2\right) + k$.
   Thus, $e(\overline{G}) = \ex(n, K_{k+1}(1, 2, \ldots, 2) + k$ and $\overline{G} \in \EX(n, K_{k+1}(1, 2, \ldots, 2), k)$.
   Using the notation of Section~\ref{sec: edge minimizers d-independence} and results there in, we know that $V(\overline{G})$ has a $k$-partition $V(G) = \cup_{i=1}^k U_i$, and $N_i = G[U_i]$ for $i \in [k]$, for which by Remark~\ref{rem: stability remark}, there exist non-negative integers $m_i$ for all $i = 0, 1, \ldots, k$, satisfying $\sum_{i=0}^k m_i = k$. Now, $k-1 \le m_0 \le k$ since $G$ is connected, and consequently $0 \le m_i \le 1$ for all $i \in [k]$, with $m_i = 1$, for at most one $i \in [k]$. 
   If $m_i = 1$ for any particular $i \in [k]$, we know that $\rho(N_i) = \rho(\mathcal{C}_k)$. Since $G$ is connected, $\rho(G) > \rho(N_i)$, and hence we have a contradiction to the minimality of the spectral radius of $G$. Thus, $m_i = 0$ for all $i \in [k]$ and $m_0 = k$. Consequently, for some integer $l \in [2, k]$, there must be a CP-cycle $\mathcal{C}_l(S)$ contained in $G$ with respect to pairs $\{u_{i_j}, v_{i_j}\} \subset C_{i_j}$ as $j \in [l]$ and $S=\cup_{j=1}^l\{u_{i_j}, v_{i_j}\}$. 
   Thus, 
   $$\rho(G) \ge \rho(\mathcal{C}_l(S)) \ge \rho(\mathcal{C}_l) = \rho(\mathcal{C}_k),$$
where the first inequality is an equality if and only if $l = k$, and the second inequality is an equality if and only if $\mathcal{C}_l(S) = \mathcal{C}_l$. Thus, $G \cong \mathcal{C}_k$.
\end{proof}

We end this section with the following results which describe a few structural properties of a spectral minimizer in $\mathcal{D}_{n,2k}$. Let $G\in \mathcal{T}_{n,2k}$. Let $C_i$ be the $i^{th}$ copy of $CP_{\frac{n}{k}}$ in $G$. If a vertex $u\in C_i$ is connected to a vertex $v \in C_j$ ($i\not =j$), we call $u$ a \emph{connector} in $C_i$.
\begin{lemma}\label{lem: four connectors are not allowed}
Consider the graph $G$ in Figure~\ref{fig: four connectors are not allowed}. Each $C_i \cong CP_{\frac{n}{k}}$, $i\in\{1,2,3,4,5\}$. Let $u_1, v_1$ and $u_2, v_2$ be non-adjacent pairs in $C_1$. Let $a, a'$, $b,b'$, $c, c'$, and $d,d'$ be non-adjacent pairs in $C_2, C_3$, $C_4$, and $C_5$, respectively. Then for $n>4k$, we have that 
$$\rho(G)> \rho(\mathcal{C}_k),$$ where $\mathcal{C}_k$ is an aligned CP-cycle.  
\end{lemma}
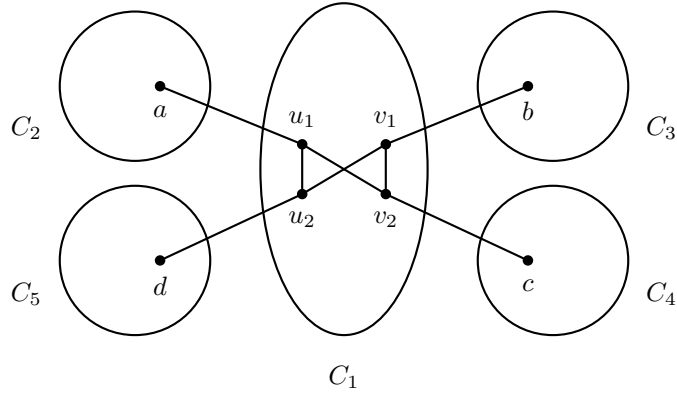
\begin{figure}
\centering
\begin{tikzpicture}[every path/.append style={thick}, scale =1.1]
\draw (-1,1) circle (0.9cm);
\draw (1.5,0) ellipse (1cm and 2cm);
\draw (4, -1.1) circle (0.9cm);
\draw (-1,-1.1) circle (0.9cm);
\draw (4,1) circle (0.9cm);
\foreach \x in {(-0.7,1), (1,0.3), (2,0.3), (1,-0.3), (2, -0.3), (3.7, 1), (3.7,-1.1), (-0.7, -1.1)}
     {\draw[fill] \x circle[radius = 0.05cm];}
     \draw (-0.7,1)--(1,0.3);
     \draw (2,0.3)--(3.7, 1);
     \draw (-0.7,-1.1)--(1,-0.3);
     \draw (1,0.3)--(1,-0.3)--(2,0.3)--(2,-0.3)--(1,0.3);
     \draw (2,-0.3)--(3.7,-1.1);

    \node at (-0.7,0.7){\small$a$};
    \node at (1,0.6) {\small $u_1$};
    \node at (2,0.6) {\small $v_1$};
    \node at (1,-0.6) {\small $u_2$};
    \node at (2,-0.6) {\small $v_2$};
    \node at (3.7, 0.7) {\small $b$};
    \node at (3.7, -1.4) {\small $c$};
    \node at (-0.7, -1.4) {\small $d$};
    \node at (-2.3, 0.5) {\small $C_2$};
    \node at (1.5, -2.5) {\small $C_1$};
    \node at (5.3, 0.5) {\small $C_3$};
    \node at (5.3, -1.5) {\small $C_4$};
    \node at (-2.3, -1.5) {\small $C_5$};
    \end{tikzpicture}
    \caption{The graph $G$ in Lemma~\ref{lem: four connectors are not allowed}.}
 \label{fig: four connectors are not allowed}
\end{figure}
\begin{proof}
    Consider the partition of $V(G)$ into the following five parts: $\{a,b,c,d\}, \{a',b',c',d'\}, V(C_2)$ $\cup V(C_3)$ $\cup V(C_4) \cup V(C_5) \setminus \{a,a',b,b',c,c', d, d'\}, \{u_1, u_2, v_1, v_2\}, V(C_1)\setminus\{u_1, u_2, v_1, v_2\}$. The corresponding quotient matrix is
    $$Q = \begin{bmatrix}
        0&0&\frac{n}{k}-2&1&0\\
        0&0&\frac{n}{k}-2&0&0\\
        1&1&\frac{n}{k}-4&0&0\\
        1&0&0&2&\frac{n}{k}-4\\
        0&0&0&4&\frac{n}{k}-6
    \end{bmatrix}.$$
The characteristic polynomial of $Q$ is 
\begin{align*}
  p(x) =& \;   x^5 + \left(8-\frac{2n}{k}\right)x^4 + \left(23-\frac{12n}{k} + \frac{n^2}{k^2}\right)x^3 + \left(22-\frac{22n}{k}+\frac{4n^2}{k^2}\right)x^2\\
  &+ \left(-10-\frac{5n}{k}+\frac{3n^2}{k^2}\right)x
  -12+\frac{8n}{k}-\frac{n^2}{k^2} .  
\end{align*}
When $x = \rho(\mathcal{C}_k) = \frac{\frac{n}{k} -3 + \sqrt{(\frac{n}{k}-1)^2 + 8}}{2}$ (from Equation~\ref{eq: spectral radius of cp cycle}), we have 
\begin{align*}
    p(\rho(\mathcal{C}_k)) =& \frac{-1}{2k^2}\left(n^2 -kn\sqrt{\left(\frac{n}{k}-1\right)^2 +8} +4k^2\sqrt{\left(\frac{n}{k}-1\right)^2 +8} -5kn+12k^2\right).
\end{align*}
We observe that for $n>4k$, we have $p(\rho(\mathcal{C}_k))<0$ if and only if 
\begin{align*}
   &\left(n^2 -kn\sqrt{\left(\frac{n}{k}-1\right)^2 +8} +4k^2\sqrt{\left(\frac{n}{k}-1\right)^2 +8} -5kn+12k^2\right)>0\\
  \iff & n^2 -5kn+12k^2> kn\sqrt{\left(\frac{n}{k}-1\right)^2 +8} -4k^2\sqrt{\left(\frac{n}{k}-1\right)^2 +8}\\
  \iff & (n^2 -5kn+12k^2)^2 - \left(kn\sqrt{\left(\frac{n}{k}-1\right)^2 +8} -4k^2\sqrt{\left(\frac{n}{k}-1\right)^2 +8}\right)^2>0\\
 \iff & 8k^2n(n-2k)>0,
\end{align*}
which is true for our choice of $n$.
This proves that $\rho(Q)>\rho(\mathcal{C}_k)$. Since the partition is equitable, we have that $\rho(G)= \rho(Q)$. Therefore, $\rho(G)>\rho(\mathcal{C}_k)$, as desired. 
\end{proof}

\begin{proposition}\label{prop: structure of spec minimizer}
     For given dissociation number $\tau =2k$, $k \in \mathbb{N}$, let $n$ be sufficiently large and $\tau \mid n$. If $G_0$ is a spectral minimizer in $\mathcal{D}_{n,\tau}$, then $G_0$ has the following properties.
\begin{enumerate}
\item If $u,v$ are two non-adjacent vertices in a copy of $CP_{\frac{n}{k}}\in G_0$ and $d_{G_0}(u)> \frac{n}{k}-1$, then $d_{G_0}(v)= \frac{n}{k}-1$.
\item If there are exactly two connectors in a copy of $CP_{\frac{n}{k}}\in G_0$, then they must be non-adjacent.
\item If there are three connectors (say $u,v,w$) in a copy of $CP_{\frac{n}{k}}\in G_0$, then $u\not\sim v$ and $d_{G_0}(w)=\frac{n}{k}-1$.
\item Any copy of $CP_{\frac{n}{k}}\in G_0$ can have at most three connectors.
\item  $\delta(G_0) = \frac{n}{k}-2$ and $\Delta(G_0)\leq \frac{n}{k}$.
\end{enumerate}   
\end{proposition}
\begin{proof}
Let $u_i$, $v_i$ be two non adjacent vertices which belong to the $i^{th}$ copy of $CP_{\frac{n}{k}}$ in $G_0$ for some $i\in[k]$. Let $u_i$ be adjacent to $w_1, w_2, \cdots, w_t$ which belong to $t>1$ vertex disjoint copies of $CP_{\frac{n}{k}}$ in $G_0$ and $d_{G_0}(v_i) = \frac{n}{k}-2$. Consider the graph $G' = G_0 -u_iw_j + v_iw_j$ for some $j\in[t]$. Note that $G'\in \mathcal{D}_{n,2k}$ and that we can obtain graph $G_0$ from $G'$ by doing the Kelmans operation on the vertices $u_i$ and $v_i$. Therefore, $\rho(G')<\rho(G_0)$, a contradiction. This proves that $d_{G_0}(v_i)\geq \frac{n}{k}-1$. To prove the equality, consider the graph $G$ in Figure~\ref{fig: outdegree of nonadjacent vertex}. Each $C_i \cong CP_{\frac{n}{k}}$, $i\in\{1,2,3,4,5\}$. Let $u_1, v_1$ be a non-adjacent pair in $C_1$. Let $a, a'$, $b,b'$, $c, c'$, and $d,d'$ be non-adjacent pairs in $C_2, C_3$, $C_4$, and $C_5$, respectively. Then if $d_{G_0}(u_i)> \frac{n}{k}-1$ and $d_{G_0}(v_i)> \frac{n}{k}-1$, we have that $G$ must be a subgraph of $G_0$. Therefore, $\rho(G_0)\geq \rho(G)$. We will now prove that $\rho(G)>\rho(\mathcal{C}_k)$ to get a contradiction. 
\begin{figure}
\centering
\begin{tikzpicture}[every path/.append style={thick}, scale =1.1]
\draw (-1,1) circle (0.9cm);
\draw (1.5,0) ellipse (1cm and 2cm);
\draw (4, -1.1) circle (0.9cm);
\draw (-1,-1.1) circle (0.9cm);
\draw (4,1) circle (0.9cm);
\foreach \x in {(-0.7,1), (1,0), (2,0), (3.7, 1), (3.7,-1.1), (-0.7, -1.1)}
     {\draw[fill] \x circle[radius = 0.05cm];}
     \draw (-0.7,1)--(1,0);
     \draw (2,0)--(3.7, 1);
     \draw (-0.7,-1.1)--(1,0);
     \draw (2,0)--(3.7,-1.1);

    \node at (-0.7,0.7){\small$a$};
    \node at (1,0.3) {\small $u_1$};
    \node at (2,0.3) {\small $v_1$};
    \node at (3.7, 0.7) {\small $b$};
    \node at (3.7, -1.4) {\small $c$};
    \node at (-0.7, -1.4) {\small $d$};
    \node at (-2.3, 0.5) {\small $C_2$};
    \node at (1.5, -2.5) {\small $C_1$};
    \node at (5.3, 0.5) {\small $C_3$};
    \node at (5.3, -1.5) {\small $C_4$};
    \node at (-2.3, -1.5) {\small $C_5$};
    \end{tikzpicture}
    \caption{The graph $G$ in Lemma~\ref{prop: structure of spec minimizer}, part 1.}
 \label{fig: outdegree of nonadjacent vertex}
\end{figure}
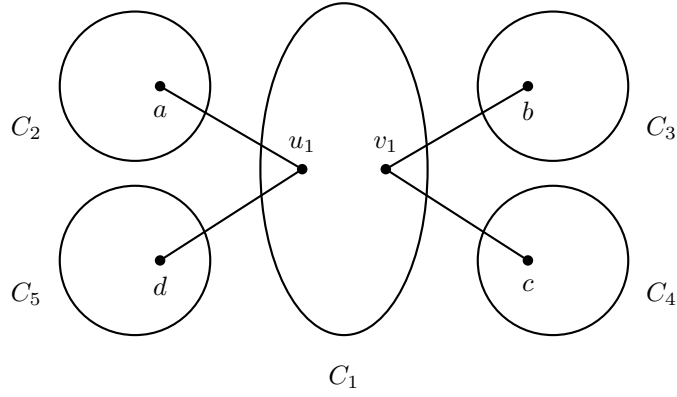
Consider the partition of $V(G)$ into the following five parts: $\{a,b,c,d\}, \{a',b',c',d'\}, \bigcup_{i=2}^5V(C_i)$ $ \setminus \{a,a',b,b',c,c', d, d'\}, \{u_1, v_1\}, V(C_1)\setminus\{u_1, v_1\}$. The corresponding quotient matrix is
\begin{equation}\label{eq: same quotient matrix}
    Q = \begin{bmatrix}
        0&0&\frac{n}{k}-2&1&0\\
        0&0&\frac{n}{k}-2&0&0\\
        1&1&\frac{n}{k}-4&0&0\\
        2&0&0&0&\frac{n}{k}-2\\
        0&0&0&2&\frac{n}{k}-4
    \end{bmatrix}.
\end{equation}
The characteristic polynomial of $Q$ is 
\begin{align*}
  p(x) =& \; \frac{1}{k^5}\bigg[k^5x^5 - (2k^4n-8k^5)x^4 -(12k^4n-22k^5-k^3n^2)x^3
  -(20k^4n-16k^5-4k^3n^2)x^2\\ & -(20k^5-2k^4n-2k^3n^2)x-(16k^5-12k^4n+2k^3n^2)\bigg].  
\end{align*}
When $x = \rho(\mathcal{C}_k) = \frac{\frac{n}{k} -3 + \sqrt{(\frac{n}{k}-1)^2 + 8}}{2}$ (from Equation~\ref{eq: spectral radius of cp cycle}), we have 
\begin{align*}
    p(\rho(\mathcal{C}_k)) =& \frac{-1}{2k^2}\left(n^2 -kn\sqrt{\left(\frac{n}{k}-1\right)^2 +8} +5k^2\sqrt{\left(\frac{n}{k}-1\right)^2 +8} -6kn+17k^2\right).
\end{align*}
We observe that for $n>5k$, we have $p(\rho(\mathcal{C}_k))<0$ if and only if 
\begin{align*}
   &\left(n^2 -kn\sqrt{\left(\frac{n}{k}-1\right)^2 +8} +5k^2\sqrt{\left(\frac{n}{k}-1\right)^2 +8} -6kn+17k^2\right)>0\\
  \iff & n^2 -6kn+17k^2> kn\sqrt{\left(\frac{n}{k}-1\right)^2 +8} -5k^2\sqrt{\left(\frac{n}{k}-1\right)^2 +8}\\
  \iff & (n^2 -6kn+17k^2)^2 - \left(kn\sqrt{\left(\frac{n}{k}-1\right)^2 +8} -5k^2\sqrt{\left(\frac{n}{k}-1\right)^2 +8}\right)^2>0\\
 \iff & 16k^2(n-2k)^2>0
 ,
\end{align*}
which is true for our choice of $n$.
This proves that $\rho(Q)>\rho(\mathcal{C}_k)$. Since the partition is equitable, we have that $\rho(G)= \rho(Q)$. Therefore, $\rho(G)>\rho(\mathcal{C}_k)$, as desired.

Part $2$ follows from Lemma~\ref{lem: nonadjacent connectors are better}. Part $3$ follows from Lemma~\ref{lem: nonadjacent connectors are better} and part $1$. 
Let $H\in \mathcal{D}_{n,2k}$ be a graph that contains a copy of $CP_{\frac{n}{k}}$ with four or more connectors. Using Lemma~\ref{lem: nonadjacent connectors are better}, we note that $H$ contains the graph $G$ from Lemma~\ref{lem: four connectors are not allowed}. Therefore, $\rho(H)\geq \rho(G)>\rho(\mathcal{C}_k)> \rho_{min}(n, 2k)$. This proves part $4$.

In part 5, $\delta(G_0) = \frac{n}{k}-2$ follows from Theorem~\ref{spec is edge minimizer}. We will prove the upper bound on the maximum degree of $G_0$ by contradiction. Let $G_0\in \mathcal{D}_{n,2k}$ be a spectral minimizer with $\Delta(G_0)\geq \frac{n}{k}+1$. Consider the graph $G$ in Figure~\ref{fig: 3 external edges not allowed}. Each $C_i \cong CP_{\frac{n}{k}}$, $i\in\{1,2,3,4,5\}$. Let $u_1, v_1$ be a non-adjacent pair in $C_1$. Let $a, a'$, $b,b'$, $c, c'$, and $d,d'$ be non-adjacent pairs in $C_2, C_3$, $C_4$, and $C_5$, respectively. Then by part 2 above, we have that $G$ must be a subgraph of $G_0$. Therefore, $\rho(G_0)\geq \rho(G)$. We will now prove that $\rho(G)>\rho(\mathcal{C}_k)$ to get a contradiction. 
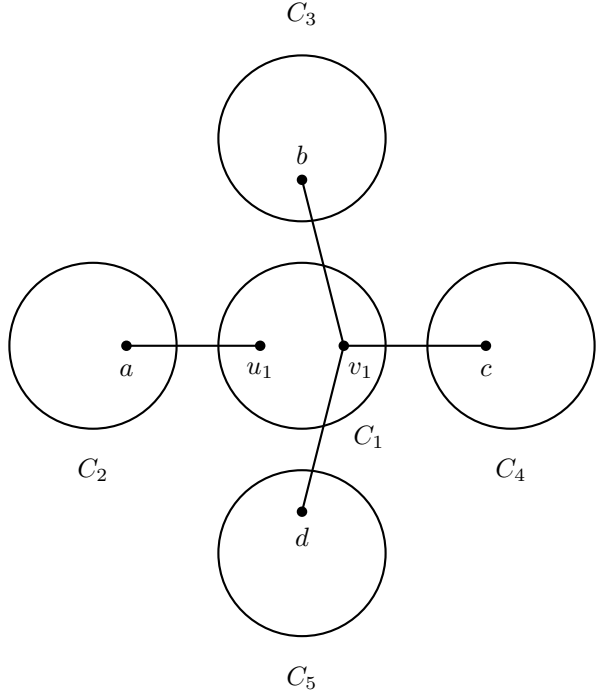
\begin{figure}
\centering
\begin{tikzpicture}[every path/.append style={thick}, scale =1.1]
\draw (-1,0) circle (1cm);
\draw (1.5,0) circle (1cm);
\draw (1.5, 2.5) circle (1cm);
\draw (4, 0) circle (1cm);
\draw (1.5, -2.5) circle (1cm);
     \foreach \x in {(-0.6,0), (1,0), (2,0), (3.7, 0), (1.5, -2), (1.5,2)}
     {\draw[fill] \x circle[radius = 0.05cm];}
     \draw (-0.6,0)--(1,0);
     \draw (2,0)--(3.7,0);
     \draw (1.5,2)--(2,0)--(1.5,-2);

    \node at (-0.6,-0.3){\small$a$};
    \node at (1,-0.3) {\small $u_1$};
    \node at (2.2,-0.3) {\small $v_1$};
    \node at (1.5, 2.3) {\small $b$};
    \node at (3.7,-0.3) {\small $c$};
    \node at (1.5, -2.3) {\small $d$};
    \node at (-1, -1.5) {\small $C_2$};
    \node at (2.3, -1.1) {\small $C_1$};
    \node at (1.5, 4) {\small $C_3$};
    \node at (4, -1.5) {\small $C_4$};
    \node at (1.5, -4) {\small $C_5$};
    \end{tikzpicture}
    \caption{The graph $G$ in Proposition~\ref{prop: structure of spec minimizer}, part 5.}
 \label{fig: 3 external edges not allowed}
\end{figure}
Consider the partition of $V(G)$ into the following five parts: $\{a,b,c,d\}, \{a',b',c',d'\}, \bigcup_{i=2}^5V(C_i) \setminus \{a,a',b,b',c,c', d, d'\}, \{u_1, v_1\}, V(C_1)\setminus\{u_1, v_1\}$. The corresponding quotient matrix is the same as in Equation~\ref{eq: same quotient matrix}.
Therefore, by part 1 above, we obtain that $\rho(Q)>\rho(\mathcal{C}_k)$. By eigenvalue interlacing, we have that $\rho(G)\geq \rho(Q)$. Therefore, $\rho(G)>\rho(\mathcal{C}_k)$, as desired. 
\end{proof}
\begin{remark}
 The final hurdle we face in proving that the aligned CP-path $\mathcal{P}_k$ is the unique spectral minimizer in $\mathcal{D}_{n,2k}$ is proving that if $G$ is a spectral minimizer in $\mathcal{D}_{n,2k}$, then any copy of $CP_{\frac{n}{k}}\in G$ cannot have three connectors. Once we have this, proving that a connector must have degree $\frac{n}{k}-1$ is not difficult.
\end{remark}

\section{Final remarks}
We conclude this paper with the following result on the spectral maximizers among all simple connected graphs of order $n$ and $d$-independence number $i_d$ and two bounds on the dissociation number of a graph.

Let $\mathcal{D}_{n, i_d}$ be the set of all simple connected graphs on $n$ vertices with $d$-independence number $i_d$. Let $G$ be a graph on $n$ vertices. We know that $\rho(G)\leq \Delta(G)$ with equality if and only if $G$ has a $\Delta(G)$-regular component. Now, the following result for spectral maximizers among graphs of given order and $d$-independence number is straightforward from the Perron-Frobenius Theorem. Because for any graph $H$ with $d$-independence number equal to $s$ where $ds$ is even, we can find a $d$-regular graph $G$, such that $H$ is a subgraph of $G\vee K_{n-s}$.
\begin{proposition}
 Let $i_d = s$. If $ds$ is even, then $G\vee K_{n-s}$ has the maximum spectral radius in $\mathcal{D}_{n, s}$, where $G$ is any $d$-regular graph on $s$ vertices.    
\end{proposition}

We now prove an upper bound on the value of the dissociation number $\tau(G)$ of a regular graph $G$  in terms of its eigenvalues, similar to the Hoffman ratio bound for the independence number. Additionally, we derive a lower bound on $\tau(G)$ for any connected graph $G$ using the probabilistic method inspired by a result of Caro and Wei (see Theorem 1 on p.91 in \cite{nogaalonspencer}). 
\begin{proposition}\label{prop: hoffman type bound}
    Let $G = (V, E)$ be a $k$-regular graph on $n$ vertices. Then 
    $$\tau(G) \leq \frac{n(1-\lambda_{min}(G))}{k-\lambda_{min}(G)}.$$
\end{proposition}
\begin{proof}
    Let $S$ be a largest dissociation set in $G$. Suppose $G[S]$ has $m$ edges and hence $\tau-2m$ isolated vertices. Consider the partition of $V$ into the following two sets:  $S$ and $V\setminus S$. The quotient matrix of this partition is
    $$Q = \begin{bmatrix}
        \frac{2m}{\tau}& k - \frac{2m}{\tau}\\
        \frac{k\tau - 2m}{n-\tau}& k - \frac{k\tau -2m}{n-\tau}
    \end{bmatrix}.$$ The eigenvalues of $Q$ are $\lambda_1(Q) = k$ and $\lambda_2(Q) = \frac{k\tau^2 -2mn}{\tau(\tau-n)}$.  By eigenvalue interlacing, we have $\lambda_2(Q) \geq \lambda_{min}(G).$ 
    Denote by  $\lambda = \lambda_{min}(G)$. Therefore,
    \begin{align*}
    0&\geq   \tau^2(k-\lambda) + \tau n \lambda -2mn\\
&\geq \tau^2(k-\lambda) + \tau(\lambda - 1)n
\end{align*}
since $\tau - m > \frac{\tau}{2}.$  Now a simple calculation gives the result.  
\end{proof}
See also \cite{acyclicnumber} for an upper bound on the \emph{acyclic number} which is the maximum number of vertices that induce a forest in a graph. 

\begin{proposition}\label{probab lower bound}
Let $G= (V,E)$ be a connected graph. Its dissociation number $$\tau(G)\geq 2 \left\lceil \sum_{e = \{u,v\}\in E}\frac{1}{(d(u) + d(v))\Delta(G) -1}\right\rceil .$$     
\end{proposition}
\begin{proof}
 For $e\in E$, let $D_i(e)$ be the subset of edges in $E$ that are at a distance (in the line graph of $G$) at most $i$ from $e.$ Pick a total ordering $<$ of $E$ uniformly at random. Define
 $$I = \{e\in E \; : \; e<e' \text{ for all } e'\in D_2(e)\}.$$ Let $X_e$ be the indicator random variable for $e\in I$ and $X = \sum_{e\in E} X_e = |I|.$  Probability $\mathbf{P}[e = \{u,v\}\in I] \geq \frac{1}{(d(u) + d(v))\Delta(G)-1}$. Hence
 $$\mathbf{E}[X] \geq \sum_{e = \{u,v\}\in E}\frac{1}{(d(u) + d(v))\Delta(G) -1} . $$ Therefore, there exists a total ordering for which $$|I|\geq \left\lceil\sum_{e = \{u,v\}\in E}\frac{1}{(d(u) + d(v))\Delta(G) -1}\right\rceil.$$ Note that the subgraph induced by the vertices incident to edges in $I$ has maximum degree 1, therefore $\tau(G)\geq 2|I|$.
\end{proof}

\section*{Acknowledgments}
The authors are grateful to Sebastian Cioab\u{a}, David Conlon, and Michael Tait for their helpful comments and suggestions. 
\bibliographystyle{abbrv}
\bibliography{bibliography}
\end{document}